\documentclass[10pt]{article}
\usepackage{amsmath, amssymb, cases, color,amsthm, mathrsfs}
\usepackage{hyperref}
\usepackage{todonotes}

\usepackage{bbm}
\usepackage{stmaryrd}
\usepackage{amsmath, amssymb,amscd, mathtools,cases, color, amsthm }
\usepackage{amsfonts}
\usepackage{graphicx}
\usepackage[top=1in, bottom=1.25in, left=1.10in, right=1.10in]{geometry}
\usepackage{subfig}
\usepackage{graphicx}
\usepackage{pgfplots}
\usetikzlibrary{patterns}
\usepgfplotslibrary{fillbetween}

\usepackage{fancyhdr}
\usepackage{bbm}
\usepackage{lipsum}
\usepackage{amsfonts}
\usepackage{graphicx}
\usepackage{epstopdf}
\usepackage{algorithmic}

\usepackage{hyperref}
\usepackage{upref}
\hypersetup{linkcolor=blue, colorlinks=true,citecolor=red}
\usepackage{authblk}
\usepackage{amsmath, amssymb, cases, color,amsthm}
\usepackage{hyperref}

\newtheorem{thm}{Theorem}[section]
\newtheorem{prop}[thm]{Proposition}
\newtheorem{lem}[thm]{Lemma}

\newtheorem{rmk}[thm]{Remark}

\newcommand{\abs}[1]{{|{#1}|}}

\theoremstyle{definition}

\theoremstyle{remark}
\numberwithin{equation}{section}

\newcommand{\by}{\mathbf{y}}
\newcommand{\bz}{\mathbf{z}}

\newcommand{\BE}{\begin{equation}}
\newcommand{\EEQ}{\end{equation}}
\newcommand{\rfb}[1]{\mbox{\rm
		(\ref{#1})}\ifx\undefined\stillediting\else:\fbox{$#1$}\fi}

\newfont{\roma}{cmr10 scaled 1200}

\newcommand{\nline}  {{\mathbb N}}

\newcommand{\rline}  {{\mathbb R}}

\newcommand{\dd}  {{\rm d}\hbox{\hskip 0.5pt}}
\renewcommand{\leq} {\leqslant}
\renewcommand{\geq} {\geqslant}
\newcommand{\mm}    {{\hbox{\hskip 0.5pt}}}
\newcommand{\m}     {{\hbox{\hskip 1pt}}}

\newcommand{\bluff} {{\hbox{\raise 15pt \hbox{\mm}}}}
\newcommand{\sbluff}{{\hbox{\raise 10pt \hbox{\mm}}}}
\newcommand{\Om}    {{\Omega}}
\renewcommand{\div} {{\rm div\,}}
\newcommand{\eps}    {{\varepsilon}}

\newcommand{\dx}     {\mathrm{d}x} 
\newcommand{\ds}     {\mathrm{d}s}

\newcommand{\prt}      {{\partial}}

\DeclarePairedDelimiter{\norm}{\lVert}{\rVert}

\newcommand{\Ascr} {\mathcal{A}}
\newcommand{\Bscr} {\mathcal{B}}
\newcommand{\Cscr} {\mathcal{C}}
\newcommand{\Dscr} {\mathcal{D}}
\newcommand{\Escr}{\mathcal{E}}

\newcommand{\Pscr} {\mathcal{P}}

\newcommand{\bbm}[1]{\left[\begin{matrix} #1 \end{matrix}\right]}

\title{Existence of strong solutions for a perfect elastic beam interacting with Navier-Stokes equations}

\author{Sebastian Schwarzacher$^{1,2}$}

\author{Pei Su$^{1}$}

\affil{\footnotesize$^1${\em Department of Mathematical Analysis, Faculty of Mathematics and Physics}\\
		{\em  Charles University }\\
	{\em Sokolovská 83, 186 75 Praha 8, Czech Republic}\\
{\em Email address:} peisu@karlin.mff.cuni.cz}

\affil{\footnotesize$^2${\em Department of Mathematics,
	Analysis and Partial Differential Equations}\\
	{\em Uppsala University }\\
	{\em L\"agerhyddsv\"agen 1,
	752 37 Uppsala, Sweden }
\\ {\em Email address:} schwarz@karlin.mff.cuni.cz}

\date{\today}
	
\begin{document}	
\maketitle

\begin{scriptsize}
	\abstract{	
A perfectly elastic beam is situated on top of a two dimensional fluid canister. The beam is deforming in accordance to an interaction with a Navier-Stokes fluid. Hence a hyperbolic equation is coupled to the Navier-Stokes equation. The coupling is partially of geometric nature, as the geometry of the fluid domain is changing in accordance to the motion of the beam. Here the existence of a unique strong solution for large initial data and all times up to geometric degeneracy is shown. For that an a-priori estimate on the time-derivative of the coupled solution is introduced. For the Navier-Stokes part it is a borderline estimate in the spirit of Ladyzhenskaya applied directly to the in-time differentiated system. }
\end{scriptsize}
\vspace{+3mm}




	
\section{Introduction}\label{sec_intro}

When a viscous fluid is interacting with a perfectly elastic solid, this results in a dissipative system. The dissipation of the system is however reduced to the fluid, while the solid evolution alone would be hyperbolic. It results in a coupling between a dissipative and non-dissipative partial differential equation, where the coupling is inherently non-linear. To determine whether the dissipation (or the irreversibility) of a system of PDE's suffices to produce regular solutions is an important and much studied problem in continuum mechanics. Indeed, many prominent open questions connect to that problem. 
Accordingly, in the field of fluid-structure interactions it is essential to understand to what extent the fluid dissipation is damping the solid motion~\cite{kaltenbacher,sunny}. In the physical scenario considered here, where the solid deformation is actually changing the Eulerian domain of definition of the fluid, the system becomes immanently non-linear. Hence, as the evolution of perfectly elastic solids is hyperbolic, the fluid dissipation becomes crucial to show any regularity beyond energy bounds.

In this paper we demonstrate that in some cases the non-linear coupling of a perfect elastic solid with a viscous fluid allows for strong solutions. 
Indeed, here the existence of a smooth solution for large times and large data is shown for an elastic beam interacting with the Navier-Stokes equation. It is the maybe simplest cases of such a non-linearly coupled fluid-solid interactions where the solid possesses {\em no dissipation}. As we will explain below, it seems that the proposed strategy is quite general and has the potential to be applied to other more complex situations.

We consider the interaction between a perfectly elastic beam and a two-dimensional fluid governed by the incompressible Navier-Stokes equation. See Figure \ref{fig_ALE} for the typical setting and \eqref{fluideq}--\eqref{initiald} for the coupled system of partial differential equations.
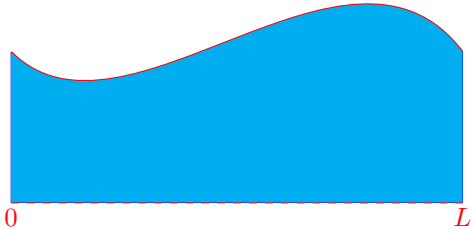
\begin{figure}[!h]\centering
	\vspace{-1cm}
	\begin{tikzpicture}[scale=1.0]
	\draw[thick,red]   (8, 2) .. controls (9.5, 0.5) and (12.5, 4) .. (14, 2);
	\draw[blue](8,2)--(8,0)--(14,0)--(14,2);
	\draw[dashed, thick, red] (8,0)--(14,0);
	\path node at (10.5,1.2) {\textcolor{blue}{$\Omega_h(t)$}};
	\path node at (8,-0.2) {\textcolor{red}{\bf $0$}};
	\path node at (14,-0.2) {\textcolor{red}{\bf $L$}};
	\draw[->,thick, red] (12.5,0)--(12.5,2.6);
	\path node at (12.85,1.3) {\textcolor{red}{\bf $h(t)$}};
	\fill[color=cyan,opacity=.1]  (8, 2) .. controls (9.5, 0.5) and (12.5, 4) .. (14, 2) -- (14,0) -- (8,0) -- (8,2);
	\end{tikzpicture}
	\caption{1D beam interacting with a 2D fluid}\label{fig_ALE}
\end{figure}    
For this setting, we show the existence of a strong solution for arbitrary large times up to the point where the fluid domain faces a topological change and with natural assumptions on the initial data, see Theorem~\ref{thm:main}. 
In particular the present work extends the available theory on strong solutions with vertical displacement of the plate in the following ways: 
\begin{itemize}
	\item It extends the theory of global strong solutions for visco-elastic solids shown in the seminal paper~\cite{grandmont2016existence} to the regime of purely elastic solids.
	\item It shows the short and long time-existence of the missing case in \cite{GraHilLaq19}, namely the case of an {\em elastic beam}. Interestingly, while the existence for short times for solids governed by the wave equation is shown there, the case of an elastic beam seems to require different methods. 
	\item It extends and improves the existence of strong solutions for short times shown by \cite{badra2019gevrey, badra2022gevrey} for the {\em elastic beam equation} to large time and to less regular conditions on the (large) initial data.	
\end{itemize}
Further the work shows that the {\em global weak solution} for an elastic beam interacting with the Navier-Stokes equation that was constructed in \cite{CasGraHil19} is a {\em strong solution} up to its first contact.  

The {\em method proposed in this work} is independent of the previous works in this setting. We wish to mention that the central regularity estimate presented here is of {\em borderline type} for the fluid equation, due to the scaling of the two dimensional Navier-Stokes equation. However, with regard to the elastic structure the estimates allow to speculate for further generalizations.

For an overview on previous efforts in fluid-structure interactions please refer to~\cite{sunny,kaltenbacher}. For the here considered coupling between a hyperbolic solid equation with a viscous fluid equation, many results showing the existence for {\em strong solutions} involving fixed geometries are available~\cite{ZhaZua,KukavicaTuffahaZiane,
	BarGruLasTuf07,BarGruLasTuff,AvalosTriggiani,
	AvalosLasieckaTriggiani}. Further results, including {\em variable geometries} for short times and/or small data can be found here~\cite{CSS1,CSS2,BdV1,BouGue10comp,IgnatovaKukavica,	ChuLasWeb13,IgnKukLasTuf17,GraHilLaq19,badra2019gevrey,badra2022gevrey,maity2021existence}. 
Previous results for {\em variable geometries} on long-time and large data for hyperbolic solids are reduced to the frame-work of weak-solutions. This includes three-dimensional fluids~\cite{grandmont2008existence,muha2013existence,lengeler2014weak}, non-linear shells~\cite{MuSc}, tangential deformations~\cite{KamSchSpe22}, global solutions~\cite{CasGraHil19} and other types of fluids~\cite{BreSch18,BreSch21} for instance. Finally we wish to mention that there are many results when the solid is considered to be a rigid body inside the Navier-Stokes fluid, see~\cite{Galdi1,Tuc1,Gunz,Des,Tak, Hieber,GerHil10,TakTucWei15,GlaSue15,MuhNecRad23} and the references therein.

If the solid is assumed to be {\em viscous} more results are available. We focus here on the question of regularity.
For a beam that is elastic and viscous interacting with the 2D Navier-Stokes equations, global smooth solution exist for arbitrary times~\cite{grandmont2016existence}. Smoothness was shown recently for a viscous elastic shell interacting with Navier-Stokes equations once the solution overcomes some regularity threshold~\cite{BreMenSchSu23}. Further we mention some results on the existence of weak solutions for the full physical setting of a deforming visco-elastic solid inside a fluid with the same dimension~\cite{BenKamSch23}.

{\em The major technical improvement in this paper is an a-priori estimate on the spacial gradient of the time-derivative}. This is achieved by analyzing the "in-time-differentiated system''. Only in this way the hyperbolic structure equation is conserved to an extend to close an estimate. Even so the coupling is extraordinarily non-linearly depending on the solution, the fluid viscosity just suffices to estimate the related terms. It is worth noticing that criticality is already true for the 2D Navier-Stokes equations without interaction. Indeed,
the existence proof of Ladyzhenskaya for a strong solution depended on a sharp border-line estimate in two dimensions. This estimate known as Ladyzhenskaya estimate is an interpolation, which we use in this paper at many instances.
What was not clear to us at the beginning of our investigations, is that an estimate on the time-derivative of the solution (combined with its energy estimate) does directly allow to show the existence of strong solutions; once the initial data is smooth enough to perform this estimate. This seems to be a different path then the one originally developed and later adapted to fluid-structure interactions~\cite{grandmont2016existence, breit2022regularity}.
The approach introduce here to fluid-structure interactions with time-changing geometry using "in-time-differentiation" seems to allow for further generalizations of the solid equation. Indeed, the here presented estimates are not of borderline-type with regard of the solid impact on the fluid. This has a clear technical reason as it circumvents the well known fact that the second-order time-derivative of a solution is not a good test-function for hyperbolic equations.

\subsection{Statement of the problem}
We consider a two-dimensional canister filled with viscous
incompressible fluid and its top surface is formed of an elastic
beam. 
As usual, 
we assume that the beam only deforms in the vertical direction on
the surface. The fluid domain, denoted by $\Om_h$, is defined 
as
$$\Om_h=\left\{(x,y)\in\rline^2\left|\m x\in(0, L), \m\m y\in(0, 
h(t,x))\right.\right\}, $$
where $h(t,x)$ is the height of the fluid column. Let the fluid
depth be $1$ when the system is at the equilibrium state. We
introduce the deformation of the beam $\eta(t,x)$, thereby the
fluid column is $h(t,x)=1+\eta(t,x)$. The fluid is supposed
to be homogeneous with density $\rho_f$ and constant viscosity 
$\mu$. The velocity and the pressure in the fluid are denoted by 
$u(t,x)$ and $p(t,x)$, respectively. With the above notation, 
the viscous incompressible Navier-Stokes equations in $\Om_h$, 
for all $t\in(0, T)$, read
\begin{equation}\label{fluideq}
\left\{\begin{aligned}
&\rho_f(\prt_t u+u\cdot\nabla u)-\div \sigma(u, p)=0 &\quad\text{ for all }
(x,y)\in\Om_h,\\
&\div u=0 &\quad\text{ for all }(x,y)\in \Om_h.
\end{aligned}
\right.
\end{equation}
The Cauchy stress tensor $\sigma(u,p)$ is defined by
$$\sigma(u,p)=2\mu D(u)-p\mathbb{I}_{2\times 2}, $$
where $\mathbb{I}_{2\times 2}$ is identity matrix of order $2$ and $D(u)$ is the 
deformation tensor:
$$D(u)=\frac{1}{2}\left(\nabla u+\left(\nabla u\right)^\intercal
\right). $$
Let the density of the beam be $\rho_s$ (constant), the linearly 
elastic beam equation is
\begin{equation}\label{beameq}
\rho_s\prt_t^2h-\beta\prt_x^2h+\alpha\prt_x^4 h=\phi(u,p,h) 
\quad\text{ for all }(t,x)\in(0, T)\times (0, L),
\end{equation}
where $\beta$ and $\alpha$ are positive coeficients on the 
properties of the beam. For more information on the derivation of this equation see~\cite{ciarlet2000theory,muha2013existence}.
The kinematic condition means that the fluid on the surface is 
consistent with the motion of the beam, i.e.
\begin{equation}\label{kineq}
u(t, x, h(t,x))=\prt_th\m {\rm e}_2 \quad\text{ for all }(t,x)\in(0, T)\times 
(0, L).
\end{equation}
The dynamic condition is to balance the forces on the surface:
\begin{equation}\label{dynamiceq}
\phi(u,p,h)(t,x)=-{\rm e_2}\cdot\sigma(u,p)(t,x, h(t,x)) (-\prt_x
h\m{\rm e_1}+{\rm e_2}).
\end{equation}
In the above expressions, $({\rm e_1}, {\rm e_2})$ is the 
canonical basis in $\rline^2$. 
To make the estimates easier to follow, we assume that the 
fluid and the structure are $L-$periodic in $x-$direction, i.e.\ for every $(t,x,y)\in (0, T)
\times \Om_h$ 
\begin{equation}\label{periodic}
u(t,x,y)=u(t,x+L, y),\quad h(t,x)=h(t,x+L), \quad (\prt_xh)(t,x)=(\prt_xh)(t, x+L).
\end{equation} 
Moreover, we consider no-slip boundary condition on the bottom of
the fluid domain:
\begin{equation}\label{noslip}
u(t,x,0)=0\quad\text{ for all }(t,x)\in(0, T)\times (0, L).
\end{equation}
To close the system, we propose the initial data as follows:
\begin{equation}\label{initiald}
u(0,x,y)=u_0(x,y), \quad h(0,x)=h_0(x),\quad (\prt_th)(0, x)=
h_1(x)\quad\text{ for all }(x,y)\in\Om_{h_0}.
\end{equation}

Further, we assume that the initial data $(u_0, h_0, h_1)$ satisfy the following compatibility conditions:
\begin{equation}\label{initialcond}
\begin{aligned}
&u_0(x,0)=0, \quad u_0(x, h_0(x))=h_1(x){\rm e_2}\quad  &\text{for all } x\in (0, L),\\	
&\div u_0=0\quad &\text{for all } (x,y)\in \Om_{h_0},\\
&\min_{x\in (0, L)} h_0(x)>\delta>0, \quad \int_0^L h_1(x)\dd x=0,
\end{aligned}
\end{equation}
which allow a strong solution in accordance to the compatibility conditions derived in Subsection~\ref{ssec:comp} below.

\subsection{Main results}
The main result of the paper is an estimate that implies the existence of a strong solution for smooth but arbitrary large data up to the point of collision between the beam and the bottom of the fluid.  
\begin{thm}
	\label{thm:main}
	Let $(u_0, h_0, h_1)$ satisfy \rfb{initialcond} and $h_0\in H^4(0,L)$, $h_1\in H^2(0,L)$, $u_0\in H^2(\Om_{h_0})$, there exists unique strong solution to the system \eqref{fluideq}--\eqref{initiald} that satisfies the following a-priori estimate:
	\begin{equation*}
	\begin{aligned}
	&\lVert\prt_tu\rVert_{L^\infty(0, T; L^2(\Om_h))}^2+
	\lVert\nabla\prt_tu\rVert^2_{L^2(0, T; L^2(\Om_h))}+
	\lVert\prt_t^2h\rVert^2_{L^\infty(0, T; L^2(0, L))}+\lVert\prt_th\rVert^2_{L^\infty(0, T; H^2(0, L))}\\
	&+\lVert u\rVert_{L^\infty(0, T; H^2(\Om_h))}^2+
	\lVert u \rVert^2_{L^4(0, T; H^\frac{5}{2}(\Om_h))} + \lVert p\rVert_{L^\infty(0, T; H^1(\Om_h))}^2+
	\lVert p \rVert^2_{L^4(0, T; H^\frac{3}{2}(\Om_h))} +
	\lVert h \rVert^2_{L^\infty(0, T; H^4(0, L))}
	\\
	&\leq C\left(h_{\min}, \lVert u_0
	\rVert_{L^2(\Om_{h_0})}, \lVert 
	h_0\rVert_{L^2(0, L)}, \lVert h_1\rVert_{L^2(0, L)}\right)
	\Big(\lVert u_0
	\rVert_{H^2(\Om_{h_0})}^2+ \lVert 
	h_0\rVert_{H^4(0, L)}^2 + \lVert h_1\rVert_{H^2(0, L)}^2\Big),
	\end{aligned}
	\end{equation*}
	as long as $h(t,x)\geq h_{\min}$ for all $(t,x)\in (0, T)\times (0, L)$.
\end{thm}
The theorem above shows in particular that for any $h_0>\delta$ satisfying the regularity assumptions above, there exists a minimal time interval $[0,T_0]$ for which a strong solution is guaranteed; this is demonstrated in Remark~\ref{rem:time} below. The estimate {\em conserves} all quantities appearing on the left-hand-side and {\em improves} the regularity of the fluid due to its viscous term.
Hence, the estimate holds up to the point of a collision between the beam and the bottom of the fluid canister. We follow here the convention that a strong solution means that all quantities in the above PDEs are valid almost everywhere; actually all quantities are at least in $L^\infty(L^2)$. 

We wish to point out that the theory in~\cite{grandmont2016existence} implies that {\em viscous beams} can not touch the bottom of the fluid-domain. This is not known for the hyperbolic problem and does not follow from the smoothness shown in this paper directly. It therefore remains a deep and difficult question, whether contact is possible or not, when a hyperbolic beam is considered. 

\subsection{Organization of the paper}
The paper is organized as follows. In Section \ref{section2} we introduce some notation and function spaces and preliminary analysis used later in the work. This follows the main part of the paper, i.e. Section \ref{sectionproof}, which is the formal a-priori estimate using the time-derivative of the coupled system transferred to a fixed geometry. This rather involving estimate follows a technical Section \ref{Garlekinsec}, where a strong solution {\em satisfying the formal a-priori estimate} is constructed using a Galerkin approximation for which the formal a-priori estimate is valid rigorously. Due to the geometrically coupled setting this construction has to be done with particular care, which itself could well be of independent interest.

\section{Preparation and preliminaries}\label{section2}
\subsection{Notation} 
We introduce some notation in this part which will be used throughout this paper.
For two non-negative quantities $f$ and $g$, we write $f\lesssim g$  if there is a $c>0$ such that $f\leq\,c g$. If necessary, we specify particular dependencies.

We consider function spaces that are periodic in the first spacial variable $x$.
For an open set $\mathcal O\subset\rline^2$ we denote  by $L^p(\mathcal O)$ and $W^{k,p}(\mathcal O)$ for $p\in[1,\infty]$ and $k\in\mathbb N$, the usual Lebesgue and Sobolev spaces over $\mathcal O$. For $p\in[1,\infty)$, the fractional Sobolev space (Sobolev-Slobodeckij space) with differentiability $s>0$ with $s\notin\mathbb N$ will be denoted by $W^{s,p}(\mathcal O)$. As usual, we use the notation $H^s(\mathcal{O}):=W^{s,2}(\mathcal{O})$ for the case $p=2$. 
The function spaces of continuous or $\alpha$-H\"older-continuous functions, $\alpha\in(0,1)$,
are denoted by $C(\overline{\mathcal O})$ or $C^{0,\alpha}(\overline{\mathcal O})$ respectively, where $\overline{\mathcal O}$ is the closure of $\mathcal O$. Similarly, we write $C^1(\overline{\mathcal O})$ and $C^{1,\alpha}(\overline{\mathcal O})$.

We denote $\by=(x, y)$ for the space variable in the time changing domain $\Om_h$ and $\bz=(x, z)$ for the space variable in $\Om_1=[0,L]\times [0,1]$, respectively. 
For a Banach space $X$, we use the shorthand $L^p_tX$ for $L^p(I;X)$ for $s\in \rline$. For instance, we write $L^p_t(W^{1,p})$ for $L^p(I;W^{1,p}(\mathcal O))$. Similarly, $W^{s,p}_t(X)$ stands for $W^{s,p}(I;X)$. We will use the shorthand notations $L^p_\by$ (or $L^p_\bz$) and $W^{s,p}_\by$ (or $W^{s,p}_\bz$) in the case of $2$-dimensional domains (typically spaces defined over $\Omega_h\subset\rline^2$ or $\Omega_1\subset\rline^2$). Finally, for vector-valued function $f$, we use $f^1$ and $f^2$ represent its first and second component, respectively. Hence we denote $f=\begin{bmatrix} f^1 & f^2\end{bmatrix}^\intercal$.
\subsection{Compatibility conditions}
\label{ssec:comp}
Note that by using the divergence-free condition of the fluid-velocity and the boundary conditions \rfb{kineq} and \rfb{noslip}, we derive that
\begin{align*}
\prt_th(t,x)&=\prt_y\int_0^{h(t,x)}u^2(t,x,y)\, dy=-\int_0^{h(t,x)}\prt_x u^1(t,x,y)\dd y
\\
&=-\prt_x\left(\int_0^{h(t,x)}u^1(t,x,y)\dd y\right) + \underbrace{\prt_xh(t,x)u^1(t,x,h(t,x))}_{=0}.
\end{align*}
According to the $x-$periodic setting \rfb{periodic}, we thereby obtain that
$$\int_0^L\prt_th(t,x)\dd x=0\quad\text{ for all }t\in(0, T). $$
This together with the beam equation \rfb{beameq} further implies that
\begin{equation}\label{zeromeansource}
\int_0^L\phi(u, p, h)(t,x)\dd x=0\quad\text{ for all }t\in(0, T).	
\end{equation}

\subsection{Change of variables}\label{sub1}
We notice that, via a change of variables, the fluid-beam system can 
be rewirtten into a fixed spatial domain. For every function $f$, 
we set
\begin{equation}\label{newvar}
\hat f(t,x,z)=f(t,x,h(t,x)z).
\end{equation}
The system \rfb{fluideq}--\rfb{initiald} can be transformed in the 
domain $\Om_1$:
$$\Om_1=\left\{(x,z)\in\rline^2\left|\m x\in(0, L),\m\m z\in(0, 1)
\right.\right\}.$$
Using the relation \rfb{newvar} and the equations \rfb{fluideq}--
\rfb{beameq}, it is not difficult to derive that the new system for $(\hat u, \hat p, h)$
defined in $\Om_1$ satisfies
\begin{subequations}\label{hatsys}
	\begin{alignat}{15}
	&\rho_fh\m\prt_t\hat u+\rho_f\m(\hat u-\prt_t\chi_h)\cdot(B_h
	\nabla)\hat u-\mu\div \left((A_h\nabla)\hat u\right)+(B_h
	\nabla)\hat p=0,\label{fluid}\\
	&\div (B_h^\intercal\hat u)=0,\label{newdiv}\\
	& \rho_s\prt_t^2 h-\beta\prt_x^2h+\alpha\prt_x^4h=\hat \phi(\hat u, 
	\hat p, h)(t,x),\label{beamnew}
	\end{alignat}
\end{subequations}
with the boundary conditions
\begin{equation}\label{hatinit}
\hat u(t,x,1)=\prt_th\m{\rm e_2}\quad\text{ for all }(t,x)\in(0, T)\times(0, L),
\end{equation}
and the periodic conditions, for every $(t,x,z)\in (0, T)\times \Om_1$,
\begin{equation}\label{priohat}
\hat u(t, x, z)=\hat u(t,x+L, z), \quad h(t,x)=h(t, x+L).
\end{equation}
To derive the structure equation, by using the condition $u^1(t, x, h(t,x))=0$ and $\div u=0$, we note that 
$$(\nabla u)^\intercal(t, x, h(t,x))(-\prt_xh{\rm e_1}+{\rm e_2})\cdot {\rm e_2}=0. $$ 
This implies that the force term $\phi(u, p, h)$ can be simplified as
$$\phi(u,p,h)(t,x)=p(t,x,h(t,x))-\mu{\rm e_2}\cdot \nabla u(t,x,h(t,x))(-\prt_xh{\rm e_1}+{\rm e_2}). $$
Hence, after change of variables the force term $\hat \phi(\hat u, \hat p, h)$ in \rfb{beamnew}, for every $(t,x)
\in(0, T)\times(0, L)$, reads 
\begin{equation*}\label{hatphi}
\hat \phi(\hat u, \hat p, h)(t,x)=-{\rm e_2}\cdot\left(\mu(A_h\nabla)
\hat u-\hat pB_h\right)(t,x,1)\m{\rm e_2}.
\end{equation*} 
The corresponding initial data are thereby given by $(\hat u_0, h_0, 
h_1)$ where $\hat u_0(x,z)=u_0(x, h_0z)$.
The matrices $A_h$, $B_h$ and $\chi_h$ in \rfb{hatsys} are as follows:
\begin{equation}\label{Ah}
\begin{aligned}
A_h=\begin{bmatrix}h & -z \prt_xh \\ -z \prt_x h & \frac{1}{h}+
\frac{(z \prt_x h)^2}{h} \end{bmatrix}&, \qquad B_h=\begin{bmatrix} 
h & -z \prt_xh  \\ 0 &1 \end{bmatrix},\qquad
\chi_h=&\begin{bmatrix} x & hz \end{bmatrix}^\intercal.
\end{aligned}
\end{equation}
For the derivation of the system \rfb{hatsys}--\rfb{priohat}, please refer to for instance \cite[Section 2]{GraHilLaq19} and \cite[Section 4]{lequeurre2011existence}  for more details.

For the above change of variables, we have the following regularity 
relation between the old (without hat) and new functions (with hat).

\begin{prop}\label{relareg}
	Assume that $\min_{(t,x)\in (0, T)\times (0, L)} h(t, x)>h_{\min}>0$. 
	Then we have:
	\begin{itemize}
		\item the mapping $f\mapsto \hat f$ introduced in \rfb{newvar} is 
		a linear homeomorphism from $L^2((0, T)\times \Om_h)$ onto 
		$L^2(0, T; L^2(\Om_1))$;
		\item For every $f\in C^\infty([0, T]\times [0, L])$, the following 
		inequalities hold:
		\begin{align*}\lVert \nabla \hat f\rVert_{L^2(\Om_1)}&\leq C(\lVert h
		\rVert_{W^{1, \infty}_x}, \lVert h^{-1}\rVert_{L^\infty_x})\lVert 
		\nabla f\rVert_{L^2(\Om_h)},
		\\
		\lVert\prt_t\hat f\rVert_{L^2(\Om_1)}&\leq C(\lVert h^{-1}\rVert_{L^\infty_x})\lVert 
		\prt_t f\rVert_{L^2(\Om_h)}+C(\lVert h^{-1}\rVert_{L^\infty_x}, \lVert\prt_t h\rVert_{L^\infty_x})\lVert\nabla f\rVert_{L^2(\Om_h)},
		\\	
		\lVert\nabla^2 \hat f\rVert_{L^2(\Om_1)}&\leq C(\lVert h\rVert_{W^{1,\infty}_x}, \lVert h^{-1}\rVert_{L^\infty_x}, \lVert\prt_x^2 h\rVert_{L^2_x})\lVert\nabla^2 f\rVert_{L^2(\Om_h)}
		\\
		\lVert\nabla\prt_t\hat f\rVert_{L^2{(\Om_1)}}&\leq C(\lVert h
		\rVert_{W^{1, \infty}_x}, \lVert h^{-1}\rVert_{L^\infty_x})\lVert 
		\nabla\prt_t f\rVert_{L^2(\Om_h)}
		+C(\lVert\prt_t h\rVert_{L^\infty_x}, \lVert h\rVert_{W^{1, \infty}_x}, 
		\lVert h^{-1}\rVert_{L^\infty_x})\lVert \nabla^2f\rVert_{L^2(\Om_h)}\\
		&\quad +C(\lVert\prt_t\prt_x h\rVert_{L^\infty_x}, \lVert
		\prt_t h\rVert_{L^\infty_x}, \lVert h^{-1}\rVert_{L^\infty_x})
		\lVert\nabla f\rVert_{L^2(\Om_h)}.
		\end{align*}
		Conversely, 
		\begin{align*}
		\lVert \nabla f\rVert_{L^2(\Om_h)}&\leq C(\lVert h
		\rVert_{W^{1, \infty}_x}, \lVert h^{-1}\rVert_{L^\infty_x})\lVert 
		\nabla \hat f\rVert_{L^2(\Om_1)},
		\\		
		\lVert\prt_t f\rVert_{L^2(\Om_h)}
		&\leq C(\lVert h\rVert_{L^\infty_x})\lVert 
		\prt_t \hat f\rVert_{L^2(\Om_1)}+C(\lVert h^{-1}\rVert_{L^\infty_x}, \lVert\prt_t h\rVert_{L^\infty_x})\lVert\nabla\hat f\rVert_{L^2(\Om_1)},
		\\
		\lVert\nabla^2 f\rVert_{L^2(\Om_h)}
		&\leq C(\lVert h\rVert_{W^{1,\infty}_x}, \lVert h^{-1}\rVert_{L^\infty_x}, \lVert\prt_x^2 h\rVert_{L^2_x})\lVert\nabla^2 \hat f\rVert_{L^2(\Om_1)}, 
		\\
		\lVert\nabla\prt_t f\rVert_{L^2{(\Om_h)}}&\leq C(\lVert h
		\rVert_{W^{1, \infty}_x}, \lVert h^{-1}\rVert_{L^\infty_x})\lVert 
		\nabla\prt_t \hat f\rVert_{L^2(\Om_1)}
		+C(\lVert\prt_t h\rVert_{L^\infty_x}, \lVert h\rVert_{W^{1, \infty}_x}, 
		\lVert h^{-1}\rVert_{L^\infty_x})\lVert \nabla^2\hat f\rVert_{L^2(\Om_1)}\\
		&\quad+C(\lVert\prt_t\prt_x h\rVert_{L^\infty_x}, \lVert
		\prt_t h\rVert_{L^\infty_x}, \lVert h^{-1}\rVert_{L^\infty_x})
		\lVert\nabla \hat f\rVert_{L^2(\Om_1)}.
		\end{align*}
	\end{itemize}
\end{prop}
\begin{proof}
	The proof follows directly from the calculations by using the change 
	of variables and the chain rule for taking derivatives.
\end{proof}
\begin{rmk}\label{relahatenergy}
	{\rm Let $h\in L^\infty(0, T; H^2(0, L))\cap W^{1, \infty}
		(0, T; L^2(0, L))$, according to Proposition \ref{relareg}, we 
		obtain that
		$$\lVert \hat f\rVert_{L^\infty(0, T; L^2(\Om_1))}\sim \lVert f
		\rVert_{L^\infty(0, T; L^2(\Om_h))}, $$
		and
		$$\lVert \hat f\rVert_{L^2(0, T; H^1(\Om_1))}\sim \lVert f
		\rVert_{L^2(0, T; H^1(\Om_h))}, $$
		which means that each one can be controlled by the other one
		multiplied by a positive constant.
	}
\end{rmk}
\subsection{Energy estimate}
We introduce here the standard energy estimate for the system 
\rfb{fluideq}--\rfb{initiald}. For that the total mechanical energy
of the fluid-beam system \rfb{fluideq}--\rfb{initiald} is denoted by 
$E_{\rm tot}$, as follows:
\begin{equation*}\label{energfor}
E_{\rm tot}(t)=\frac{1}{2}\int_0^L\left(\rho_s|\prt_th(t,x)|^2
+\beta|\prt_xh(t,x)|^2+\alpha|\prt_x^2h(t,x)|^2\right) \dd x+\frac{1}{2}
\int_{\Om_h}\rho_f|u(t,x,y)|^2\dd \by.
\end{equation*}
For strong solutions the energy inequality is known to satisfy the following identity:
%
%
\begin{equation}\label{enerinequality}
E_{\rm tot}(t)+\int_0^t\int_{\Om_h}|\nabla u|^2\dd \by\dd t=
E_{\rm tot}(0)\quad\text{ for all }t\in (0, T).
\end{equation}
It can be derived by multiplying the beam equation with $\partial_t h$ and the fluid equation with $u$.
Recalling Remark \ref{relahatenergy}, the energy estimate 
\rfb{enerinequality} implies that
\begin{equation}\label{energyhat}
\begin{aligned}
&\lVert\prt_th\rVert^2_{L^\infty(0, T); L^2(0, L))}+\lVert\prt_x 
h\rVert^2_{L^\infty(0, T; L^2(0, L))}+
\lVert\prt_x^2h\rVert^2_{L^\infty(0, T; L^2(0, L))}\\
&\quad +\lVert\hat u\rVert^2_{L^\infty(0, T; L^2(\Om_1))}+\lVert\nabla 
\hat u\rVert^2_{L^2(0, T; L^2(\Om_1))}
\leq \lVert h_1\rVert_{L^2(0, L)}^2+\lVert h_0\rVert_{H^2(0, L)}^2+\lVert\hat u_0\rVert_{L^2(\Om_1)}^2:=C_0. 
\end{aligned}
\end{equation}

\subsection{H\"older continiuity in time-space}
We will also use the following interpolation to gain continuity in space and in time. Let $|x-y|<\frac{r}{2}$ and $0<t_2-t_1\leq r^\frac{3}{2}$, then
\begin{align}
\label{eq:hoeldercont}
\begin{aligned}
|h(t_1,x)-h(t_2,y)|& \leq  \Big|h(t_1,x)-\frac{1}{r}\int_{x}^{x+r}h(t_1,s)\, \ds\Big|
+\Big|\frac{1}{r}\int_{x}^{x+r}h(t_1,s)\, \ds -\frac{1}{r}\int_{x}^{x+r}h(t_2,s)\, \ds\Big|
\\
&
\quad + \Big|h(t_2,y)-\frac{1}{r}\int_{x}^{x+r}h(t_2,s)\, \ds\Big|
\\
& \leq 2\sqrt{C_0}r +\frac{1}{r}\int_{x}^{x+r}|h(t_1,s)-h(t_2,s)|\, \ds
\\
&\leq 2\sqrt{C_0}r + \frac{1}{r}\int_{x}^{x+r}\int_{t_1}^{t_2}|\partial_\tau h(\tau,s)|\, \dd \tau\, \dd s
\\
&\leq 2\sqrt{C_0}r + \Big(\frac{|t_2-t_1|}{r}\Big)^\frac{1}{2}\bigg(\int_{x}^{x+r}\int_{t_1}^{t_2}|\partial_\tau h(\tau,s)|^2\, \dd \tau\, \dd s\bigg)^\frac{1}{2}
\\
&\leq \sqrt{C_0}(2r+\frac{|t_2-t_1|}{\sqrt{r}})\leq 3\sqrt{C_0}r.
\end{aligned}
\end{align}
This implies that we have the continuous embedding: 
$$L^\infty(0, T; H^2(0, L))\cap W^{1,\infty}(0, T; L^2(0, L))\hookrightarrow C^{0, \alpha}([0, T]; C^1[0, L]),$$ 
for $0<\alpha\leq \frac{2}{3}.$ The estimate also implies that necessarily for some time no contact between the beam and the bottom may occur.
\begin{rmk}[Minimal interval of existence]
	\label{rem:time}
	{\em 
		The energy estimate implies the uniform bounds of 
		\[
		\norm{\partial_t h}^2_{L^\infty(0, T; L^2(0, L))}+\norm{h}^2_{L^\infty(0, T; H^2(0, L))}\leq C_0,
		\]
		just in dependence of the initial data. This allows the following interpolation estimate. Assume that $h_0(x)\geq \delta$ for all $x\in (0, L)$, then
		for $r>0$ to be chosen later we find
		\begin{align*}
		h(t,x)& \geq  h(t,x)-\frac{1}{r}\int_{x}^{x+r}h(t,s)\, \ds+\frac{1}{r}\int_{x}^{x+r}h(t,s)\, \ds-\frac{1}{r}\int_{x}^{x+r}h_0(s)\, \ds+ \delta
		\\
		& \geq \delta-\sqrt{C_0} r +\frac{1}{r}\int_{x}^{x+r}\left(h(t,s)-h_0(s)\right)\, \ds
		\geq \delta-\sqrt{C_0} r+\frac{1}{r}\int_{x}^{x+r}\int_0^t\partial_\tau h(\tau,s)\, \dd \tau\, \dd s
		\\
		&\geq \delta-\sqrt{C_0} r-\Big(\frac{t}{r}\Big)^\frac{1}{2}\bigg(\int_{x}^{x+r}\int_0^t|\partial_\tau h(\tau,s)|^2\, \dd \tau\, \dd s\bigg)^\frac{1}{2}
		\geq \delta-\sqrt{C_0}\Big(r+\frac{t}{\sqrt{r}}\Big)\geq \delta-\sqrt{C_0}t^{\frac{2}{3}},
		\end{align*}
		for $r=t^{\frac{2}{3}}$. Hence as long as 
		$T<\left(\delta/\sqrt{C_0}\right)^{\frac{3}{2}}$ no contact is possible, thereby there exists $h_{\min},\m h_{\max}>0$ such that $h\in [h_{\min}, h_{\max}]$ for every $(t,x)\in (0, T)\times (0, L)$. 
	}
\end{rmk}


\section{Formal time-derivative estimate}\label{sectionproof}

In this section, we formally present the time-derivative estimate, which will be rigorously verified in Section \ref{Garlekinsec} by the Galerkin approximation.
And the proof is exaclty corresponding to the proof of 
Theorem \ref{thm:main-disrete} in the discrete level. After passing to the limit, this eventually holds in the continuous level.

In this central part of the paper we shall obtain the estimate for $\prt_t\hat u$ based 
on the system \rfb{hatsys}--\rfb{priohat}. The estimate we aim for is 
\begin{equation}\label{formal-est}
\begin{aligned}
&\rho_f\lVert\prt_t\hat u\rVert_{L^\infty(0, T; L^2(\Om_1))}^2+\mu
\lVert\nabla\prt_t\hat u\rVert^2_{L^2(0, T; L^2(\Om_1))}+\rho_s
\lVert\prt_t^2h\rVert_{L^\infty(0, T; L^2(0, L))}^2\\
&\quad+\beta\lVert\prt_t\prt_x h\rVert_{L^\infty(0, T; L^2(0, L))}^2+\alpha
\lVert\prt_t\prt_x^2 h \rVert^2_{L^\infty(0, T; L^2(0, L))}
\\
&\leq C(h_{\min},C_0)\left(\lVert(\prt_t\hat u)(0)\rVert_{L^2(\Om_1)}^2 +  \lVert(\prt_t^2 h)(0)\rVert_{L^2(0, L)}^2+ \lVert 
h_0\rVert_{H^2(0, L)}^2\right).
\\
\end{aligned}
\end{equation}
Rigorously it follows from Theorem~\ref{thm:main-disrete}, where the estimate is performed on the Galerkin level.


\subsection{Excluding the pressure}
Note that $\prt_t\hat u$ 
is not a suitable test function since it does not satisfy the 
modified divergence free condition \rfb{newdiv}, we thereby to take 
its perturbation in the form of $\prt_t\hat u+G$. Therefore, we derive 
in what follows the specific expression for $G$.

Based on the modified divergence free condition \rfb{newdiv}, we have
$$0=\div(\prt_t(B_h^\intercal\hat u))=\div(\prt_tB_h^\intercal\hat u)
+\div(B_h^\intercal\prt_t\hat u), $$
which gives that $-\div(B_h^\intercal\prt_t\hat u)=\div(\prt_tB_h^
\intercal\hat u)$.
Recalling the matrix $B_h$ defined in \rfb{Ah}, we compute that 
$$\prt_tB_h^\intercal\hat u=\begin{bmatrix}\prt_th & 0 \\ -\prt_t
\prt_xhz & 0\end{bmatrix}\begin{bmatrix}\hat u^1 \\ \hat u^2 
\end{bmatrix} =\begin{bmatrix}\prt_th \hat u^1 \\ -\prt_t
\prt_xhz\hat u^1 \end{bmatrix}.$$
Let $G=B_h^{-\intercal}\prt_tB_h^\intercal\hat u$, then we get
\begin{equation}\label{G}
G=\begin{bmatrix} \frac{1}{h} & 0 \vspace{+1mm}\\ \frac{\prt_xhz}{h} & 
1\end{bmatrix}\begin{bmatrix}\prt_th\hat u^1 \\ -\prt_t\prt_xhz\hat u^1 
\end{bmatrix}=\begin{bmatrix}\frac{1}{h}\prt_th\hat u^1 \vspace{+1mm}\\ 
\frac{z}{h}\prt_th\prt_xh\hat u^1-\prt_t\prt_xhz\hat u^1 \end{bmatrix}. 
\end{equation}
With $G$ defined in \rfb{G}, we have 
$$\div(B_h^\intercal(\prt_t\hat u+G))=\div(B_h^\intercal\prt_t\hat u)
+\div(B_h^\intercal G)=0, $$
which implies that $\prt_t\hat u+G$ is a suitable test function for
the fluid equation \rfb{fluid}.

\begin{lem}\label{pressurelemma}
	For $G$ given in \rfb{G}, there exists $\varphi$ determined by $\hat u$ and $G$ such that $\varphi(x, 1)
	=0$ and
	\begin{equation*}\label{dealpre}
	\div\left(\prt_t B_h^\intercal(\prt_t\hat u+G)\right)=\div(B_h^\intercal
	\varphi),
	\end{equation*}
	where $\varphi$ is 
	\begin{equation}\label{findvarphi}
	\varphi=\begin{bmatrix} \varphi^1 \\ \varphi^2\end{bmatrix} =
	\begin{bmatrix} \frac{1}{h}\prt_th\prt_t\hat u^1+\frac{1}{h^2}
	|\prt_th|^2\hat u^1 \vspace{+2mm}\\ \frac{z}{h}\prt_th\prt_xh
	\prt_t\hat u^1+\frac{z}{h^2}|\prt_th|^2\prt_xh\hat u^1-z\prt_t
	\prt_xh\prt_t\hat u^1-\frac{z}{h}\prt_th\prt_t\prt_xh \hat u^1\end{bmatrix}.
	\end{equation}
\end{lem}

\begin{proof}
	The above $\varphi$ can be obtained directly by setting $\varphi=
	B_h^{-\intercal}\prt_t B_h^\intercal(\prt_t\hat u+G)$ with $G$ in 
	\rfb{G}.
\end{proof}

In the following we collect terms of an a-priori estimate that naturally relates to the time-derivative system. For that we collect various signed terms for the left hand side and error terms we put on the right hand side. The main effort will be to estimate the error terms.
We start by taking the derivative of \rfb{fluid} with respect to the time variable and multiplying the resulting equation by $\prt_t\hat u+G$. We then pick $t\in (0, T)$ and integrate over $(0,t)\times \Omega_1$,
\begin{equation}\label{firstineq}
\begin{aligned}
\int_0^t\int_{\Om_1}\left\{\rho_f\prt_th\prt_t\hat u+\rho_fh\prt_
t^2\hat u+\rho_f\prt_t\left((\hat u-\prt_t\chi_h)\cdot (B_h\nabla)
\hat u\right)-\mu\div((\prt_tA_h\nabla)\hat u)\right.\\
\quad\left.
-\mu\div((A_h\nabla)\prt_t\hat u)+\prt_tB_h \nabla\hat p+B_h 
\nabla\prt_t\hat p\right\}\m(\prt_t\hat u+G)\m\dd\bz\dd t=0.
\end{aligned}
\end{equation}
Note that we have
$$\int_0^t\int_{\Om_1}\rho_f\prt_th\prt_t\hat u(\prt_t\hat u+G)\m\dd\bz\dd t=
\underbrace{\int_0^t\int_{\Om_1}\rho_f\prt_th|\prt_t\hat u|^2\m\dd \bz\dd t}_{:=2R_1}
+\underbrace{\int_0^t\int_{\Om_1}
	\rho_f\prt_th\prt_t\hat u\cdot G \m\dd\bz\dd t}_{:=R_2},$$
and 
$$\int_0^t\int_{\Om_1}\rho_fh\prt_t^2\hat u(\prt_t\hat u+G)\m\dd\bz\dd t=\underbrace{\int_0^t
	\int_{\Om_1}\frac{\rho_f}{2}\frac{\dd }{\dd t}\left(|\prt_t\hat u|^2h
	\right)\m\dd\bz\dd t}_{:=L_1}
+\underbrace{\int_0^t\int_{\Om_1}
	\left(\rho_fh\prt_t^2\hat u\cdot G-\frac{\rho_f}{2}\prt_th|\prt_t\hat u|^2\right)\m\dd\bz\dd t}_{:=R_3-R_1}. $$

For the terms on $A_h$, we derive by Gauss theorem using the boundary conditions of $\hat u$ and $G$; in particular the fact that $(\partial_t \hat u + G)(t,x,1)=\partial_t^2h{\rm e_2}$ (as $G(t,x,1)=0$ 
since it only depends on $\hat u^1$, see \rfb{G}). We find that
\begin{align*}
&-\mu\int_0^t\int_{\Om_1}\div((\prt_tA_h\nabla)\hat u)\cdot(\prt_t
\hat u+G)\m\dd\bz\dd t\\
&=-\underbrace{\mu\int_0^t\int_0^L\left(((\prt_tA_h\nabla)\hat u)(t,x,1)\m\prt_t^2
	h{\rm e_2}\right)\cdot{\rm e_2}\m\dd x\dd t}_{:=(*)} +\underbrace{\mu\int_0^t\int_{\Om_1}((\prt_t
	A_h\nabla)\hat u):(\nabla\prt_t\hat u+\nabla G)\m\dd \bz\dd t}_{:=R_4},
\end{align*} 
and 
\begin{align*}
&-\mu\int_0^t\int_{\Om_1}\div((A_h\nabla)\prt_t\hat u)\cdot(\prt_t
\hat u+G)\m\dd\bz\dd t=-\underbrace{\mu\int_0^t\int_0^L\left(((A_h\nabla)\prt_t\hat u)(t,x,1)\m
	\prt_t^2h{\rm e_2}\right)\cdot{\rm e_2}\m\dd x\dd t}_{:=(**)}\\
&\quad +\underbrace{\mu\int_0^t\int_{\Om_1}((A_h\nabla)\prt_t\hat u):\nabla\prt_t\hat 
	u\m\dd \bz\dd t}_{:=L_2}
+\underbrace{\mu\int_0^t\int_{\Om_1}((A_h\nabla)\prt_t\hat u):\nabla G\m\dd\bz\dd t}_{:=R_5}.
\end{align*}

For the two pressure terms in \rfb{firstineq}, using integration 
by parts for space we arrive, using $\div(B_h^\intercal(\prt_t\hat u+G))=0$ and $\hat 
u(t,x,1)=\prt_th{\rm e_2}$,  at
\begin{equation*}
\int_0^t\int_{\Om_1} B_h \nabla\prt_t\hat p\cdot (\prt_t\hat u+G)\m\dd\bz\dd t
=\underbrace{\int_0^t\int_0^L\left((\prt_t\hat p\m B_h)(t,x,1)\m\prt_t^2h{\rm e_2}\right)
	\cdot{\rm e_2}\m\dd x\dd t}_{:=(***)}.
\end{equation*}
Moreover, integration by parts implies (using the structure of $G$)
\begin{align}\label{pressdeal}
&\int_0^t\int_{\Om_1}\m\prt_t B_h \nabla\hat p \cdot(\prt_t\hat u+G)\m\dd \bz\dd t\nonumber\\
&=
\underbrace{\int_0^t\int_0^L\left((\hat p\m\prt_t B_h)(t,x,1)\m\prt_t^2h{\rm e_2}\right)
	\cdot{\rm e_2}\m\dd x\dd t}_{=0}
-\int_0^t\int_{\Om_1}\hat p\m\div(\prt_t B_h^\intercal(\prt_t\hat u+G))\m\dd\bz\dd t.
\end{align} 
Note that $\div\left(\prt_t B_h^\intercal(\prt_t\hat u+G)\right)\neq 0$, 
we thus need to estimate the second term on the right hand side 
of \rfb{pressdeal}. Multiplying the momentum equation of the fluid by $\varphi$ introduced in \rfb{findvarphi}, we find using integration by parts and Lemma \ref{pressurelemma}, that
\begin{align*}
\int_0^t\int_{\Om_1}\hat p\m\div\left(\prt_t B_h^\intercal(\prt_t\hat 
u+G)\right)\m\dd\bz\dd t&=\underbrace{\int_0^t\int_{\Om_1}\rho_fh\m\prt_t\hat u\cdot\varphi\m\dd\bz\dd t}_{:=R_6}
+\underbrace{\int_0^t\int_{\Om_1}\rho_f(\hat u-\prt_t\chi_h)\cdot(B_h\nabla)\hat 
	u\cdot\varphi\m\dd\bz\dd t}_{:=R_7}\nonumber\\
&\quad+\underbrace{\int_0^t\int_{\Om_1}\mu (A_h\nabla)\hat u:\nabla\varphi\m\dd\bz\dd t}_{:=R_8}.
\end{align*}

We combine the above by taking the time-derivative of the beam equation \rfb{beamnew},
which gives that for every $(t, x)\in (0, T)\times (0, L)$,
\begin{align*}
\rho_s\prt_t^3h-\beta\prt_t\prt_x^2h+\alpha\prt_t\prt_x^4h=-{\rm e_2}\cdot
\left\{\mu(\prt_tA_h\nabla)\hat u+\mu (A_h\nabla)\prt_t\hat u-
\prt_t\hat pB_h-\hat p\prt_tB_h\right\}(t,x,1){\rm e_2}.
\end{align*}
Taking the inner product of the above equation with $\prt_t^2h$,
we have (using that $ {\rm e_2}\cdot \partial_t B_h {\rm e_2}=0$)
\begin{align}\label{beamtimeder}
\int_0^t\int_0^L\left\{\rho_s\prt_t^3h-\beta\prt_t\prt_x^2h+\alpha
\prt_t\prt_x^4h\right.\left.+\m{\rm e_2}\cdot\left(\mu (\prt_tA_h\nabla)\hat u+\mu (A_h\nabla)\prt_t
\hat u-\prt_t\hat p B_h\right)(t,x,1)\m{\rm e_2}\right\}
\prt_t^2h\m\dd x\dd t=0.
\end{align}
By taking intergration by parts and using the boundary conditions, 
we obtain that
\begin{align*}
&\int_0^t\int_0^L\rho_s\prt_t^3h\prt_t^2h\m\dd x\dd t=\frac{\rho_s}{2}\int_0^t\int_0^L
\frac{\dd}{\dd t}\left|\prt_t^2h\right|^2\m\dd x\dd t, \\
&-\int_0^t\int_0^L\beta\prt_t\prt_x^2h\m\prt_t^2h\m\dd x\dd t=\int_0^t\int_0^L
\beta\prt_t\prt_xh\m\prt_t^2\prt_xh\m\dd x\dd t=\frac{\beta}{2}\int_0^t\int_0^L
\frac{\dd}{\dd t}\left|\prt_t\prt_xh\right|^2\m\dd x\dd t, \\
&\int_0^t\int_0^L\alpha\prt_t\prt_x^4h\m\prt_t^2h\m\dd x\dd t=\alpha\int_0^t
\int_0^L\prt_t\prt_x^2h\m\prt_t^2\prt_x^2h\m\dd x\dd t=\frac{\alpha}{2}
\int_0^t\int_0^L\frac{\dd }{\dd t}\left|\prt_t\prt_x^2h\right|^2\m\dd x\dd t.
\end{align*}
We add the equations \rfb{firstineq} 
and \rfb{beamtimeder} together. Note that this implies that the terms $(*),(**)$ and $(***)$ cancel with the beam equation. Hence we find, for every 
$t\in(0, T)$, that
\begin{align*}
&\frac{\rho_s}{2}\int_0^L\left|(\prt_t^2h)(t)
\right|^2\m\dd x+\frac{\beta}{2}\int_0^L\left|(\prt_t\prt_xh)(t)\right|^2\m\dd x+
\frac{\alpha}{2}\int_0^L\left|(\prt_t\prt_x^2h)(t)\right|^2\m\dx
\nonumber\\
&\qquad +\underbrace{\frac{\rho_f}{2}\int_{\Om_1}h(t)|\prt_t\hat u(t)|^2\m\dd\bz}_{=L_1}
+\underbrace{\mu\int_0^t\int_{\Om_1}
	(A_h\nabla)\prt_t\hat u:\nabla\prt_t\hat u\m\dd\bz\dd t}_{=L_2}
\end{align*}
\begin{align}\label{sum}
&=-\underbrace{\int_0^t\int_{\Om_1}\frac{\rho_f}{2}\prt_th\left|\prt_t\hat u
	\right|^2\m\dd\bz\dd t}_{=R_1}-\underbrace{\int_0^t\int_{\Om_1}\rho_f\prt_th\prt_t\hat u\cdot G\m\dd\bz\dd t}_{=R_2}-
\underbrace{\int_0^t\int_{\Om_1}\rho_fh\prt_t^2\hat u\cdot G\m\dd\bz\dd t}_{=R_3}\nonumber\\
&\quad -\underbrace{\mu\int_0^t\int_{\Om_1}\left((\prt_t A_h\nabla)\hat u\right):
	\left(\nabla\prt_t\hat u+\nabla G\right)\m\dd\bz\dd t}_{=R_4}-\underbrace{\mu\int_0^t\int_{\Om_1}
	(A_h\nabla)\prt_t\hat u:\nabla G\m\dd\bz\dd t}_{=R_5}
\nonumber\\
&\quad+\underbrace{\int_0^t\int_{\Om_1}\rho_fh\prt_t\hat u\cdot\varphi\m\dd\bz\dd t}_{=R_6}
+
\underbrace{\int_0^t
	\int_{\Om_1}\rho_f(\hat u-\prt_t\chi_h)\cdot(B_h\nabla)\hat u
	\cdot \varphi\m\dd\bz\dd t}_{=R_7}
+\underbrace{\mu\int_0^t\int_{\Om_1}(A_h\nabla)\hat u:\nabla\varphi\m\dd\bz\dd t}_{=R_8}\nonumber\\
&\quad -\underbrace{\int_0^t\int_{\Om_1}\prt_t\left(\rho_f(\hat u-\prt_t\chi_h)
	\cdot(B_h\nabla)\hat u\right)\cdot(\prt_t\hat u+G)\m\dd\bz\dd t}_{:=R_9}\nonumber\\
&\quad+\frac{\rho_s}{2}\int_{\Om_1}h_0\left|(\prt_t\hat u)(0)\right|^2\m\dd \bz+
\frac{\rho_s}{2}\int_0^L|(\prt_t^2h)(0)|^2\m\dd x+
\frac{\beta}{2}\int_0^L|\prt_xh_1|^2\m\dx+\frac{\alpha}{2}\int_0^L
|\prt_x^2h_1|^2\m\dx,
\end{align}
where $\varphi$ and $G$ has been introduced in \rfb{findvarphi} and \rfb{G}, respectively. Please note that $R_9$ is a term that is left from \eqref{firstineq}.
Under the no-contact assumption $\min_{(t, x)\in (0, T)\times (0, L)}h(t,x)>h_{\min}$,
together with the energy estimate \rfb{energyhat}, 
we notice that there exists $c_1$, $c_2>0$ such that 
$$c_1\mathbb{I}_{2\times 2}\leq A_h(t,x)\leq c_2\mathbb{I}_{2\times 2} \quad\text{ for all }(t, x)\in(0, T)\times 
\Om_1. $$
Hence, if the left hand side of \eqref{sum} is bounded it implies the same bounds for the left hand side of \eqref{formal-est}.
Moreover, together with the continuous embedding $H^1_t\hookrightarrow 
L^\infty_t$ we can additionally put $\lVert\nabla\hat u
\rVert_{L^\infty(0, T; L^2(\Om_1))}^2$ on the left side of \rfb{sum}. 

\subsection{Estimating the right hand side of \rfb{sum}}
In the remaining part, we focus on the estimate of the right 
hand side of \rfb{sum}.
To present the proof clearly, we divide it into four steps below. In the following we use without further notice that $\partial_xh,h$ and $\frac{1}{h}$ are uniformly bounded in space time.

{\bf Step 1. The estimate of the first three terms on the right side 
	of \rfb{sum}.} 
For the first term on the right side of \rfb{sum}, we have
\begin{equation}\label{first}
\begin{aligned}
\abs{R_1}&\lesssim\int_0^t\int_0^1 \lVert\prt_th\rVert_{L_x^2}\lVert\prt_t\hat 
u\rVert_{L^4_x}^2\dd z\dd t\\
&\lesssim \int_0^t\int_0^1\lVert\prt_th\rVert_{L^2_x}\lVert\prt_t\hat 
u\rVert_{L^2_x}^{\frac{3}{2}}\lVert\prt_t\prt_x\hat u\rVert_{L^2_x}
^{\frac{1}{2}}\m\dd z\dd t\\
&\lesssim \lVert\prt_th\rVert_{L^\infty_t(L^2_x)}\left(\int_0^t\int_0
^1\left(\eps\lVert\prt_t\prt_x\hat u\rVert_{L^2_x}^2+C(\eps)\lVert\prt_t\hat 
u\rVert_{L^2_x}^2\right)\right)\m\dd z\dd t\\
&\leq \eps\int_0^t\lVert\prt_t\prt_x\hat u\rVert_{L^2_\bz}
^2\dd t+C(\eps)\int_0^t\lVert \prt_t\hat u\rVert_{L^2_\bz}^2\dd t,
\end{aligned}
\end{equation}
where we used the Young's inequality and the boundedness of $\prt_th$ 
in $L^\infty_t(L^2_x)$ (see \rfb{energyhat}), as well as the following 
interpolation inequality:
$$\lVert\prt_t\hat u\rVert_{L^4_x}\lesssim\lVert\prt_t\hat u\rVert_{L^2_x}
^{\frac{3}{4}}\lVert\prt_t\prt_x\hat u\rVert_{L^2_x}^{\frac{1}{4}}
\quad\text{ for all }(t, x)\in(0, T)\times(0, L). $$

For $R_2$ we use the formula 
of $G$ in \rfb{G} and derive that
\begin{equation}\label{2nd+}
R_2
=\int_0^t\int_{\Om_1}\left(\frac{1}{h}\left|\prt_th\right|^2\hat u^1\prt_t
\hat u^1
+\frac{z}{h}|\prt_th|^2\prt_xh\m\hat u^1\prt_t\hat u^2-\prt_th\m
\prt_t\prt_xh z\hat u^1\prt_t\hat u^2\right)\m\dd\bz\dd t.
\end{equation}
Observe that we need 
to estimate the first and the last integrals on the right side of 
\rfb{2nd+}. Firstly,
\begin{equation}\label{2nd1}
\begin{aligned}
\int_0^t\int_{\Om_1}\frac{1}{h}\left|\prt_th\right|^2\hat u^1\prt_t
\hat u^1\m\dd\bz\dd t&\lesssim\int_0^t\lVert\prt_th\rVert_{L^\infty_x}^2\lVert\hat u\rVert_
{L^2_\bz}\lVert\prt_t\hat u\rVert_{L^2_\bz}\m\dd t\\
&\lesssim\int_0^t\lVert\prt_t h\rVert_{L^2_x}\lVert\prt_t\prt_xh\rVert_
{L^2_x}\lVert\hat u\rVert_{L^2_\bz}\lVert\prt_t\hat u\rVert_{L^2
	_\bz}\m\dd t\\
&\lesssim\lVert\prt_th\rVert_{L^\infty_t(L^2_x)}\lVert\hat u\rVert_{L^\infty_t
	(L^2_\bz)}\int_0^t\lVert\prt_t\prt_xh\lVert_{L^2_x}
\lVert\prt_t\hat u\rVert_{L^2_\bz}\m\dd t\\
&\lesssim\int_0^t\left(\lVert\prt_t\prt_xh\rVert_{L^2_x}^2+\lVert\prt_t
\hat u\rVert_{L^2_\bz}^2\right)\m\dd t,
\end{aligned}
\end{equation}
where we used the interpolation inequality as follows:
\begin{equation}\label{ineqinfty}
\lVert\prt_th\rVert_{L^\infty_x}\lesssim \lVert\prt_th\rVert_{L^2_x}
^{\frac{1}{2}}\lVert\prt_t\prt_xh\rVert_{L^2_x}^{\frac{1}{2}}
\quad\text{ for all }(t,x)\in(0, T)\times (0, L).
\end{equation} 
This 1D interpolation \rfb{ineqinfty} we will use frequently below and please see \cite[Theorem 5.9]{adams2003sobolev} or \cite{brezis2018gagliardo} for reference.
Then we estimate the last integral in \rfb{2nd+}:
\begin{equation}\label{2nd2}
\begin{aligned}
\int_0^t\int_{\Om_1}\prt_th\m\prt_t\prt_xh\m\hat u^1\prt_t\hat u^2\m\dd\bz\dd t
&\leq \int_0^t\lVert\prt_th\rVert_{L^\infty_x}\lVert\prt_t\prt_xh\rVert
_{L^\infty_x}\lVert\hat u\rVert_{L^2_\bz}\lVert\prt_t\hat u\rVert_
{L^2_\bz}\m\dd t\\
&\lesssim\int_0^t\lVert\prt_th\rVert_{L^2_x}^{\frac{1}{2}}\lVert\prt_t
\prt_xh\rVert_{L^2_x}\lVert\prt_t\prt_x^2h\rVert_{L^2_x}^{\frac{1}{2}}
\lVert\hat u\rVert_{L^2_\bz}\lVert\prt_t\hat u\rVert_{L^2_\bz}\m\dd t\\
&\lesssim \int_0^t\lVert\prt_th\rVert_{L^2_x}\lVert\prt_t\prt_x^2h
\rVert_{L^2_x}\lVert\hat u\rVert_{L^2_\bz}\lVert\prt_t\hat u
\rVert_{L^2_\bz}\m\dd t\\
&\lesssim\int_0^t\left(\lVert\prt_t\hat u\rVert_{L^2_\bz}^2+\lVert
\prt_t\prt_x^2h\rVert^2_{L^2_x}\right)\m\dd t.
\end{aligned}
\end{equation}
In the above estimate, we used in particular the interpolation
inequality:
$$\lVert\prt_t\prt_xh\rVert_{L^2_x}\leq C\lVert\prt_th\rVert_{L^2_x}
^{\frac{1}{2}}\lVert\prt_t\prt_x^2h\rVert_{L^2_x}^{\frac{1}{2}}.  $$
Combining with the estimate \rfb{2nd1} and \rfb{2nd2}, we thereby 
obtain the estimate for \rfb{2nd+} as follows: 
\begin{equation*}\label{2ndf}
\abs{R_2}\lesssim\int_0^t
\left(\lVert\prt_t\hat u\rVert_{L^2_\bz}^2+\lVert\prt_t\prt_xh
\rVert_{L^2_x}^2+\lVert\prt_t\prt_x^2h\rVert^2_{L^2_x}\right)\m\dd t.
\end{equation*}

To estimate $R_3$ we 
first take an integration by parts with respect to the time variable, 
which gives that
\begin{equation}\label{3rd1}
\begin{aligned}
R_3
&=\int_{\Om_1}\rho_fh\prt_t\hat u\cdot G\m\dd\bz\dd t-\int_0^t\int_{\Om_1}\rho_f
\prt_th\prt_t\hat u(t)\cdot G(t)\m\dd\bz\dd t\\
&\quad-\int_0^t\int_{\Om_1}\rho_fh\prt_t\hat u\cdot\prt_t G\m\dd\bz\dd t
-\int_{\Om_1}\rho_fh_0(\prt_t\hat u)(0)\cdot G(0)\m\dd \bz.
\end{aligned}
\end{equation}
Using the expression of $G$ in \rfb{G}, we start to estimate the right 
side of \rfb{3rd1} by using interpolation inequalities. We first have
\begin{equation}\label{3right1}
\int_{\Om_1}\rho_fh\prt_t\hat u\cdot G\m\dd\bz=\int_{\Om_1}\rho_f\left(\prt_th
\hat u^1\prt_t\hat u^1+z\prt_th\prt_xh\hat u^1\prt_t\hat u^2-zh\prt_t
\prt_xh\hat u^1\prt_t\hat u^2\right)\dd\bz.
\end{equation}
and then
\begin{equation}\label{3sub1}
\begin{aligned}
\int_{\Om_1}\rho_f\prt_th\hat u^1\prt_t\hat u^1\m\dd\bz
&\lesssim\lVert\prt_th
\rVert_{L^\infty_x}\lVert\hat u\rVert_{L^2_\bz}\lVert\prt_t\hat u
\rVert_{L^2_\bz}\\
&\lesssim \lVert\prt_th\rVert_{L^2_x}^{\frac{3}{4}}\lVert\prt_t\prt_xh
\rVert_{L^2_x}^{\frac{1}{4}}\lVert\prt_t\hat u\rVert_{L^2_\bz}\\
&\leq \eps\lVert\prt_t\hat u\rVert_{L^2_\bz}^2+C(\eps)\lVert\prt_th
\rVert_{L^2_x}^{\frac{3}{2}}\lVert\prt_t\prt_xh\rVert_{L^2_x}^{
	\frac{1}{2}}\\
&\leq \eps\lVert\prt_t\hat u\rVert_{L^\infty_t(L^2_\bz)}^2+\delta\lVert\prt_t
\prt_xh\rVert_{L^\infty_t(L^2_x)}^2+C(\eps, \delta)\lVert\prt_th\rVert_{L^\infty_t(L^2_x)}^2,
\end{aligned}
\end{equation}
where we used twice Young's inequality. The second term on the right 
side of \rfb{3right1} can be estimated in a similar way as \rfb{3sub1}. 
For the last term in \rfb{3right1}, we also have
\begin{equation}\label{3sub2}
\begin{aligned}
\int_{\Om_1}\rho_f z h\prt_t\prt_xh\hat u^1\prt_t\hat u^2\m\dd \bz
&\lesssim
\lVert\prt_t\prt_xh\rVert_{L^\infty_x}\lVert\hat u\rVert_{L^2_\bz}
\lVert\prt_t\hat u\rVert_{L^2_\bz}\\
&\lesssim\lVert\prt_th\rVert_{L^2_x}^{\frac{1}{4}}\lVert\prt_t
\prt_x^2h\rVert_{L^2_x}^{\frac{3}{4}}\lVert\prt_t\hat u\rVert_
{L^2_\bz}\\
&\leq \eps\lVert\prt_t\hat u\rVert_{L^2_\bz}^2+C(\eps)\lVert\prt_th
\rVert_{L^2_x}^{\frac{1}{2}}\lVert\prt_t\prt_x^2h\rVert_{L^2_x}
^{\frac{3}{2}}\\
&\leq \eps\lVert\prt_t\hat u\rVert_{L^\infty_t( L^2_\bz)}^2+\delta\lVert\prt_t
\prt_x^2h\rVert_{L^\infty_t(L^2_x)}^2+C(\eps, \delta)\lVert\prt_th\rVert_{L^\infty_t(L^2_x)}^2.
\end{aligned}
\end{equation}
Combining with the estimate \rfb{3sub1} and \rfb{3sub2}, we obtain 
that
\begin{equation}\label{3rd1f}
\int_{\Om_1}\rho_fh\prt_t\hat u\cdot G\m\dd\bz\leq \eps\lVert\prt_t\hat u\rVert_{L^\infty_t( L^2_\bz)}^2+\delta\lVert\prt_t
\prt_x^2h\rVert_{L^\infty_t(L^2_x)}^2+C(\eps, \delta)\lVert\prt_th\rVert_{L^\infty_t(L^2_x)}^2.
\end{equation}

The second term on the right of \rfb{3rd1} is the same with $R_2$ term, which has been estimated in \rfb{2nd+}. 
The third term on the right side of \rfb{3rd1} is
\begin{equation*}\label{3rd3rd}
\begin{aligned}
&\int_0^t\int_{\Om_1}\rho_fh\prt_t\hat u\cdot\prt_t G\m\dd\bz\dd t\\
&=\int_0^t\int
_{\Om_1}\left\{\rho_fh\prt_t\hat u^1\left(-\frac{1}{h^2}|\prt_th|^2\hat u^1
+\frac{1}{h}\prt_t^2h\hat u^1+\frac{1}{h}\prt_th\prt_t\hat u^1\right)\right.\\
&\quad +\rho_fh\prt_t\hat u^2\left(-\frac{z}{h^2}|\prt_th|^2\prt_xh\hat u^1
+\frac{z}{h}\prt_t^2h\prt_xh\hat u^1+\frac{z}{h}\prt_th\prt_t\prt_xh
\hat u^1\right.\\
&\quad +\left.\left.\frac{z}{h}\prt_th\prt_xh\prt_t\hat u^1-\prt_t^2\prt_xhz\hat 
u^1-\prt_t\prt_xhz\prt_t\hat u^1\right)\right\}\m\dd\bz\dd t.
\end{aligned}
\end{equation*}
In view of \eqref{first} and \eqref{2nd1} the terms left to be estimated are:
\begin{equation}\label{threetype}
\int_0^t\int_{\Om_1}\prt_t^2h\hat u^1\prt_t\hat u^1\m\dd\bz\dd t,\quad \int_0^t
\int_{\Om_1}\rho_fh\prt_t\hat u^2\prt_t^2\prt_xh\hat u^1\m\dd\bz\dd t,\quad \int_0^t
\int_{\Om_1}\rho_fzh\prt_t\prt_xh\prt_t\hat u^1\prt_t\hat u^2\m\dd\bz\dd t.
\end{equation}
Using interpolation inequality and Young's inequality, we obtain that
\begin{equation}\label{3rd3rdsam1}
\begin{aligned}
\int_0^t\int_{\Om_1}\prt_t^2h\hat u^1\prt_t\hat u^1\m\dd\bz\dd t
&\lesssim \int_0^t\lVert\prt_t^2h\rVert_{L^2_x}\int_0^1\lVert\hat u
\rVert_{L^2_x}^{\frac{1}{2}}\lVert\prt_x\hat u\rVert_{L^2_x}^{\frac{1}{2}}
\lVert\prt_t\hat u\rVert_{L^2_x}\m\dd z\dd t\\
&\lesssim \int_0^t\lVert\prt_t^2h\rVert_{L^2_x}\lVert\hat u\rVert_{L^2
	_\bz}^{\frac{1}{2}}\lVert\prt_x\hat u\rVert^{\frac{1}{2}}_{L^2
	_\bz}\lVert\prt_t\hat u\rVert_{L^2_\bz}\m\dd t\\
&\lesssim\int_0^t\lVert\prt_t^2h\rVert_{L^2_x}^2\m\dd t+\int_0^t\lVert\nabla\hat u\rVert_{L^2_\bz}^2\lVert\prt_t
\hat u\rVert_{L^2_\bz}^2\m\dd t.
\end{aligned}
\end{equation}
For the other two in \rfb{threetype}, we need to take the integration 
by parts with repsect to $x$, which gives that
\begin{equation}\label{3rd3rdsam2}
\int_0^t\int_{\Om_1}\rho_fh\prt_t\hat u^2\prt_t^2\prt_xh\hat u^1\m\dd\bz\dd t
=-\int_0^t\int_{\Om_1}\rho_f\prt_t^2h\left(\prt_xh\prt_t\hat u^2\hat 
u^1+h\prt_t\prt_x\hat u^2\hat u^1+h\prt_t\hat u^2\prt_x\hat u^1\right)\m\dd\bz\dd t.
\end{equation}
The first term on the right of \rfb{3rd3rdsam2} can be estimated as 
\rfb{3rd3rdsam1}. The second term on the right of \rfb{3rd3rdsam2} can 
be estimate as:
\begin{equation*}
\label{I}
\begin{aligned}
\int_0^t\int_{\Om_1}\rho_fh\prt_t^2h\prt_t\prt_x\hat u^2\hat u^1\m\dd\bz\dd t&\lesssim
\int_0^t\int_0^1\lVert\prt_t^2h\rVert_{L^2_x}\lVert\prt_t\prt_x\hat u
\rVert_{L^2_x}\lVert\hat u\rVert_{L^\infty_x}\dd z\dd t\\
&\lesssim\int_0^t\lVert\prt_t^2h\rVert_{L^2_x}\lVert\prt_t\prt_x\hat u
\rVert_{L^2_\bz}\lVert\hat u\rVert_{L^2_\bz}^{\frac{1}{2}}\lVert
\prt_x\hat u\rVert_{L^2_\bz}^{\frac{1}{2}}\m\dd t\\
&\leq \eps\int_0^t\lVert\prt_t\prt_x\hat u\rVert_{L^2_\bz}^2\m\dd t+C(\eps)\int_0^t
\lVert\nabla\hat u\rVert_{L^2_\bz}^2\lVert\prt_t^2h\rVert_{L^2_x}^2\m\dd t.
\end{aligned}
\end{equation*}
Similarly, we estimate the last term on the right side of \rfb{3rd3rdsam2}:
\begin{equation*}
\label{II}
\begin{aligned}
\int_0^t\int_{\Om_1}\rho_fh\prt_t^2h\prt_t\hat u^2\prt_x\hat u^1\m\dd\bz\dd t&\lesssim
\int_0^t\int_0^1\lVert\prt_t^2h\rVert_{L^2_x}\lVert\prt_x\hat u\rVert_
{L^2_x}\lVert\prt_t\hat u\rVert_{L^\infty_x}\m\dd z\dd t\\
&\lesssim  \int_0^t\lVert\prt_t^2h\rVert_{L^2_x}\lVert\prt_x\hat u\rVert_
{L^2_\bz}\lVert\prt_t\hat u\rVert_{L^2_\bz}^{\frac{1}{2}}\lVert
\prt_t\prt_x\hat u\rVert_{L^2_\bz}^{\frac{1}{2}}\m\dd t\\
&\leq \eps\int_0^t\lVert\prt_t\prt_x\hat u\rVert_{L^2_\bz}^2\m\dd t+C(\eps)\int_0^t
\lVert\prt_x\hat u\rVert_{L^2_\bz}^2\lVert\prt_t^2h\rVert_{L^2_x}^2\m\dd t.
\end{aligned}
\end{equation*}
We also use the integration by parts for the last integral in \rfb{threetype}:
\begin{equation}\label{now}
\int_0^t
\int_{\Om_1}\rho_fzh\prt_t\prt_xh\prt_t\hat u^1\prt_t\hat u^2\m\dd\bz\dd t
=-\int_0^t\int_{\Om_1}\rho_fz\prt_th\left(h\prt_t\prt_x\hat u^1\prt_t\hat 
u^2+h\prt_t\hat u^1\prt_t\prt_x\hat u^2+\prt_xh\prt_t\hat u^1\prt_t\hat 
u^2\right)\m\dd\bz\dd t.
\end{equation}
The critical term in \rfb{now} is estimated as follows:
\begin{equation}\label{needlater}
\begin{aligned}
\int_0^t\int_{\Om_1}\prt_th\prt_t\hat u\prt_t\prt_x\hat u\m\dd\bz\dd t&\leq \int_0^t
\int_0^1\lVert\prt_th\rVert_{L^2_x}\lVert\prt_t\hat u\rVert_{L^\infty_x}
\lVert\prt_t\prt_x\hat u\rVert_{L^2_x}\dd z\dd t\\
&\lesssim \int_0^t\lVert\prt_t\hat u\rVert_{L^2_\bz}^{\frac{1}{2}}\lVert
\prt_t\prt_x\hat u\rVert_{L^2_\bz}^{\frac{3}{2}}\dd t\\
&\leq \eps\int_0^t\lVert\prt_t\prt_x\hat u\rVert_{L^2_\bz}^2\m\dd t+C(\eps)\int_0^t
\lVert\prt_t\hat u\rVert_{L^2_\bz}^2\m\dd t.
\end{aligned}
\end{equation}

Combining with the estimate \rfb{3rd3rdsam1}--\rfb{now}, we have 
\begin{equation}\label{3rd3rdfin}
\begin{aligned}
\int_0^t\int_{\Om_1}\rho_fh\prt_t\hat u\cdot\prt_t G\m\dd \bz\dd t
&\leq \eps\int_0^t\lVert\prt_t\prt_x\hat u\rVert_{L^2_\bz}^2\dd t+C(\eps)\int_0^t
\left(\lVert\prt_t\hat u\rVert_{L^2_\bz}^2+\lVert\prt_t^2h\rVert_{L^2_x}^2
\right)\m\dd t\\
&\quad+C(\eps)\int_0^t(\lVert\prt_t
\hat u\rVert_{L^2_\bz}^2+\lVert\prt_t^2h\rVert_{L^2_x}^2)\lVert\nabla\hat 
u\rVert_{L^2_\bz}^2\dd t.
\end{aligned}
\end{equation}

Putting together the estimate \rfb{3rd1f}, \rfb{2nd+} and \rfb{3rd3rdfin}, 
we derive that 
\begin{equation*}\label{3rdfinal}
\begin{aligned}
\abs{R_3}
&\leq \eps\lVert\prt_t\hat u\rVert_{L^\infty_t(L^2_\bz)}+\delta\lVert\prt_t\prt_x^2h
\rVert_{L^\infty_t(L^2_x)}^2+\eps\int_0^t\lVert\prt_t\prt_x\hat u\rVert_{L^2_\bz}^2\m\dd t\\
&\quad+C(\eps, \delta)\int_0^t
\left(\lVert\prt_t\hat u\rVert_{L^2_\bz}^2+\lVert\prt_t^2h\rVert_{L^2_x}^2
+\lVert\prt_t\prt_xh\rVert_{L^2_x}^2+\lVert\prt_t\prt_x^2h\rVert^2_{L^2_x}
\right)\m\dd t\\
&\quad+C(\eps)\int_0^t\left(\lVert\prt_t
\hat u\rVert_{L^2_\bz}^2+\lVert\prt_t^2h\rVert_{L^2_x}^2\right)\lVert\nabla
\hat u\rVert_{L^2_\bz}^2\m\dd t\\
&\quad-\int_{\Om_1}\rho_fh_0(\prt_t\hat u)(0)\cdot G(0)\dd \bz.
\end{aligned}
\end{equation*}
The last integral with the initial data above can be handled in the following way:
\begin{equation*}
\begin{aligned}
\int_{\Om_1}\rho_fh_0(\prt_t\hat u)(0)\cdot G(0)\m\dd\bz&\lesssim
\int_{\Om_1}\left(h_1\hat u_0(\prt_t\hat u)(0)+h_1\prt_xh_0\hat u_0(\prt_t\hat u)(0)+h_0\prt_xh_1(\prt_t\hat u)(0) \right)\m\dd\bz\\
& \leq C(C_0, \lVert h_1\rVert_{H^2(0, L)})(C_0+\lVert(\prt_t\hat u)(0)\rVert_{L^2_\bz}^2),
\end{aligned}
\end{equation*} 
where we used the expression \rfb{3right1} and $C_0$ has been introduced in \rfb{energyhat}.

{\bf Step 2. The estimate of the two terms $R_4$ and $R_5$ that are related to $\nabla G$ in \rfb{sum}.}
Recalling the 
matrix $A_h$ and $G$ introduced in \rfb{Ah} and \rfb{G}, we compute $\prt_tA_h \nabla\hat u:(\nabla\prt_t\hat u+\nabla G)$.
With that we can list the terms appearing in $R_4$ below:
\begin{equation*}\label{needtoestim}
\begin{aligned}
&a_1:=\int_0^t\int_{\Om_1}|\prt_th||\prt_x\hat u^1||\prt_t\prt_x\hat u^1|\m\dd\bz\dd t, \quad 
&a_2:=\int_0^t
\int_{\Om_1}|\prt_th|^2|\hat u^1||\prt_x\hat u^1|\m\dd\bz\dd t,\\
&a_3:=\int_0^t\int_{\Om_1}|\prt_t
h||\prt_t\prt_xh||\hat u^1||\prt_x\hat u^1|\m\dd\bz\dd t,\quad
&a_4:=\int_0^t\int_{\Om_1}|\prt_th|^2|\prt_x\hat u^1|^2\m\dd\bz\dd t,\\
&a_5:=\int_0^t\int_{\Om_1} 
|\prt_t\prt_xh||\prt_z\hat u^2||\prt_t\prt_z\hat u^1|\m\dd\bz\dd t,\quad 
&a_6:=\int_0^t\int_{\Om_1}|\prt_t
\prt_xh|^2|\hat u^1||\prt_x\hat u^1|\m\dd\bz\dd t,
\\
&a_7:= \int_0^t\int_{\Om_1}|\prt_th||\prt_t\prt_xh||\prt_z\hat u^2||\prt_z\hat u^1|\m\dd\bz\dd t, \quad 
&a_8:=\int_0^t\int_{\Om_1}|\prt_th|^2|\prt_x^2h||\hat u^1||\prt_x\hat u^2|\m\dd\bz\dd t,\\
&a_9:=\int_0^t\int_{\Om_1}|\prt_th||\prt_t\prt_x^2h||\hat u^1||\prt_x\hat u^2|\m\dd\bz\dd t, \quad
&a_{10}:= \int_0^t\int_{\Om_1}|\prt_th||\prt_x^2h||\prt_t\prt_xh
||\hat u^1||\prt_x\hat u^1|\m\dd\bz\dd t,\\
&a_{11}:=
\int_0^t\int_{\Om_1}|\prt_t\prt_xh||\prt_t\prt_x^2h||\hat u^1||\prt_x\hat u^1|\m\dd\bz\dd t, \quad 
&a_{12}:=\int_0^t\int_{\Om_1}|\prt_t\prt_xh|^2|\prt_x\hat u^1|^2\m\dd\bz\dd t
\end{aligned}
\end{equation*}
We do select the most critical terms of above.
First $a_1$ is subcritical to $a_5$.
Next we find that $a_2$ and $a_8$ are subcritical to $a_{10}$. Further $a_3$, $a_4$, $a_6$ and $a_7$ are subcritical to $a_{12}$. Finally $a_9$ is subcritical to $a_{11}$. Hence we need to estimate
\begin{align*}
&A_5:=\int_0^t\int_{\Om_1} 
\abs{\prt_t\prt_xh}\abs{\nabla\hat u}\abs{\prt_t\nabla\hat u}\m\dd\bz\dd t,
&A_{10}:= \int_0^t\int_{\Om_1}\abs{\prt_th}\abs{\prt_x^2h}\abs{\prt_t\prt_xh}\abs{\hat u}\abs{\nabla\hat u}\m\dd\bz\dd t,
\\
&A_{11}:=
\int_0^t\int_{\Om_1}\abs{\prt_t\prt_xh}\abs{\prt_t\prt_x^2h}\abs{\hat u}\abs{\nabla\hat u}\m\dd\bz\dd t,\quad  
&A_{12}:=\int_0^t\int_{\Om_1}|\prt_t\prt_xh|^2\abs{\nabla\hat u}^2\m\dd\bz\dd t.
\end{align*}  
We begin with $A_5$ where we find that
\begin{equation*}\label{needtoestim5}
\begin{aligned}
\abs{A_5}&\leq 
\int_0^t\lVert\prt_t\prt_xh\rVert_{L^\infty_x}\lVert\nabla\hat u\rVert_{L^2
	_\bz}\lVert\prt_t\nabla\hat u\rVert_{L^2_\bz}\m\dd t\\
&\lesssim  \int_0^t\lVert\prt_t\prt_x^2h\rVert_{L^2_x}\lVert\nabla\hat u\rVert_{
	L^2_\bz}\lVert\prt_t\nabla\hat u\rVert_{L^2_\bz}\m\dd t\\
&\leq \eps\int_0^t\lVert\prt_t\nabla\hat u\rVert_{L^2_\bz}^2\m\dd t+C(\eps)\int_0^t
\lVert\prt_t\prt_x^2h\rVert_{L^2_x}^2\lVert\nabla\hat u\rVert_{L^2_\bz}^2\m\dd t.
\end{aligned}
\end{equation*}
Next we estimate using interpolation in $x$ and the energy bounds we find
\begin{align*}
\abs{A_{10}}&\leq \int_0^t\int_{\Om_1}\norm{\prt_th}_{L^\infty_x}\norm{\prt_t\prt_xh}_{L^\infty_x}\norm{\hat u}_{L^\infty_x}\norm{\nabla\hat u}_{L^2_x}\norm{\prt_x^2h}_{L^2_x}\m\dd\bz\dd t
\\
&\lesssim \int_0^t\int_{\Om_1}\norm{\prt_th}_{L^2_x}^\frac{1}{2}\norm{\prt_t\prt_x^2h}_{L^2_x}^\frac{1}{2}\norm{\prt_t\prt_xh}_{L^2_x}\norm{\hat u}_{L^2_x}^\frac{1}{2}\norm{\nabla\hat u}_{L^2_x}^\frac{3}{2}\m\dd\bz\dd t
\\
&\lesssim \int_0^t\int_{\Om_1}\norm{\prt_t\prt_x^2h}_{L^2_x}^\frac{3}{2}\norm{\nabla\hat u}_{L^2_x}^\frac{3}{2}\m\dd\bz\dd t
\\
&\lesssim \int_0^t \norm{\prt_t\prt_x^2h}_{L^2_x}^2\norm{\nabla\hat u}_{L^2_{\bz}}^2\dd t+C,
\end{align*}
which is suitable for a Gr\"onwall estimate.
Further we estimate similarly
\begin{align*}
\abs{A_{11}}&\leq
\int_0^t\int_{\Om_1}\norm{\prt_t\prt_xh}_{L^\infty_x}\norm{\prt_t\prt_x^2h}_{L^2_x}\norm{\hat u}_{L^\infty_x}\norm{\nabla\hat u}_{L^2_x}\m\dd\bz\dd t
\\
&\lesssim \int_0^t\int_{\Om_1}\norm{\prt_th}_{L^2_x}^\frac12\norm{\prt_t\prt_x^2h}_{L^2_x}^\frac{3}{2}\norm{\hat u}_{L^2_x}^\frac{1}{2}\norm{\nabla\hat u}_{L^2_x}^\frac{3}{2}\m\dd\bz\dd t\\
&\lesssim \int_0^t \norm{\prt_t\prt_x^2h}_{L^2_x}^2\norm{\nabla\hat u}_{L^2_{\bz}}^2\dd t+C.
\end{align*}
Finally we estimate
\begin{equation*}
\begin{aligned}
\abs{A_{12}}&\leq \int_0^t\lVert\prt_t\prt_xh
\rVert_{L^\infty_x}^2\lVert\nabla\hat u\rVert_{L^2_\bz}^2\m\dd t
\lesssim \int_0^t\lVert\prt_t\prt_x h\rVert_{L^2_x}\lVert\prt_t\prt_x^2h\rVert_{L^2_x}
\lVert\nabla\hat u\rVert_{L^2_\bz}^2\m\dd t\\
&\lesssim\int_0^t(\lVert\prt_t\prt_x^2h\rVert_{L^2_x}^2+\lVert\prt_t\prt_xh\rVert_{L^2_x}^2)\lVert\prt_x\hat u
\rVert_{L^2_\bz}^2\m\dd t.
\end{aligned}
\end{equation*}
%
Hence we conclude (applying Poincar\'e's inequality) that
\begin{equation*}\label{needtoestimfin}
\begin{aligned}
\abs{R_4}&=\int_0^t\int_{\Om_2}\prt_tA_h \nabla\hat u:(\nabla\prt_t\hat u+\nabla G)\m\dd\bz\dd t
\lesssim \int_0^t \norm{\prt_t\prt_x^2h}_{L^2_x}^2\norm{\nabla\hat u}_{L^2_{\bz}}^2\dd t+C,
\end{aligned}
\end{equation*}
is suitable for a Gr\"onwall estimate.

Now we continue estimating $R_5$ in \rfb{sum}. 
It can be easily checked that the terms appearing that have not been treated before are
\begin{equation*}\label{gradGterms}
\begin{aligned}
&b_1:=\int_0^t\int_{\Om_1}|\prt_th||\prt_xh||\hat u||\prt_t\prt_x\hat u|\m\dd\bz\dd t, \quad 
b_2:=\int_0^t\int_{\Om_1}|\prt_t\prt_xh||\hat u||\prt_t\prt_x\hat u|\m\dd\bz\dd t,\\
&b_3:=\int_0^t\int_{\Om_1}|\prt_th||\prt_x^2h||\hat u||\prt_t\prt_x\hat u|\m\dd\bz\dd t, \quad 
b_4:=\int_0^t\int_{\Om_1}|\prt_t\prt_x^2h||\hat u||\prt_t\prt_x\hat u|\m\dd \bz\dd t,
\end{aligned}
\end{equation*}
where critical are only $b_3$ and $b_4$.
%
We find
\begin{equation*}
\begin{aligned}
\abs{b_3}&\leq 
\int_0^t\int_0^1\lVert\prt_th\rVert_{L^\infty_x}\lVert\prt_x^2h\rVert_
{L^2_x}\lVert\hat u\rVert_{L^\infty_x}\lVert\prt_t\prt_x\hat u\rVert_{L^2_x}\m\dd z\dd t\\
&\lesssim\int_0^t\lVert\prt_t\prt_xh\rVert_{L^2_x}^{\frac{1}{2}}\lVert
\prt_x\hat u\rVert_{L^2_\bz}^{\frac{1}{2}}\lVert\prt_t\prt_x\hat u
\rVert_{L^2_\bz}\m\dd t\\
&\leq \eps\int_0^t\lVert\prt_t\prt_x\hat u\rVert_{L^2_\bz}^2\m\dd t+C(\eps)
\int_0^t\left(\lVert\nabla\hat u\rVert_{L^2_\bz}^2+\lVert\prt_t
\prt_xh\rVert_{L^2_x}^2\right)\m\dd t,
\end{aligned}
\end{equation*}
and
\begin{equation*}
\begin{aligned}
\abs{b_4}\leq \eps
\int_0^t\lVert\prt_t\prt_x\hat u\rVert_{L^2_\bz}^2\m\dd t+C(\eps)\int_0^t\lVert
\prt_t\prt_x^2h\rVert_{L^2_x}^2(1+\lVert\nabla\hat u\rVert_{L^2_\bz}^2)\m\dd t.
\end{aligned}
\end{equation*}

Therefore, we have the estimate:
\begin{equation*}\label{2ndnablaGfin}
\begin{aligned}
\int_0^t\int_{\Om_1}(A_h\nabla)\prt_t\hat u:\nabla G \m\dd\bz\dd t&\leq  \eps\int_0^t
\lVert\prt_t\prt_x\hat u\rVert_{L^2_\bz}^2\m\dd t+C(\eps)\int_0^t\lVert\prt_t
\prt_x^2h\rVert_{L^2_x}^2\lVert\nabla\hat u\rVert_{L^2_\bz}^2\m\dd t\\
&\quad+C(\eps)\int_0^t\left(\lVert\prt_t\prt_xh\rVert_{L^2_x}^2+\lVert\prt_t
\prt_x^2h\rVert_{L^2_x}^2+\lVert\nabla\hat u\rVert_{L^2_\bz}^2\right)\m\dd t.
\end{aligned}
\end{equation*}

{\bf Step 3. The estimate of the terms $R_6$, $R_7$ and $R_8$ depending on $\varphi$ in \rfb{sum}.}
Recall that $\varphi$ has been 
introduced in \rfb{findvarphi}. We first notice that
\begin{equation*}
\begin{aligned}
\rho_fh\prt_t\hat u\cdot\varphi&=\rho_f\prt_th|\prt_t\hat u^1|^2+\frac{
	\rho_f}{h}|\prt_th|^2\prt_t\hat u^1\hat u^1+\rho_fz\prt_th\prt_xh
\prt_t\hat u^1\prt_t\hat u^2\\
&\quad+\frac{\rho_fz}{h}|\prt_th|^2\prt_xh\hat u^1\prt_t\hat u^2-\rho_fhz
\prt_t\prt_xh\prt_t\hat u^1\prt_t\hat u^2-\rho_fz\prt_th\prt_t\prt_xh
\prt_t\hat u^2\hat u^1.
\end{aligned}
\end{equation*}
Compared with the terms we have estimated before, here we only need to consider the following integral
\begin{equation}\label{integrpartn}
\begin{aligned}
\int_0^t\int_{\Om_1}\rho_fhz\prt_t\prt_xh\prt_t\hat u^1\prt_t\hat u^2\m\dd\bz\dd t&=
-\int_0^t\int_{\Om_1}\rho_fz\prt_th\prt_xh\prt_t\hat u^1\prt_t\hat u^2\m\dd\bz\dd t\\
&\quad -\int_0^t\int_{\Om_1}\rho_fhz\prt_th\prt_t\prt_x\hat u^1
\prt_t\hat u^2\m\dd\bz\dd t
-\int_0^t\int_{\Om_1}\rho_fhz\prt_th\prt_t\hat u^1\prt_t
\prt_x\hat u^2\m\dd \bz\dd t,
\end{aligned}
\end{equation}
where we used the integration by parts for the space variable $x$. The 
first term on the right side of \rfb{integrpartn} has been estimated in 
\rfb{first}. Note that the other two terms on the right of \rfb{integrpartn} 
have the same structure and actually also have been considered in 
\rfb{needlater}. Thus, we have \vspace{-2mm}
\begin{equation*}\label{firstvarphifin}
\abs{R_6}\leq \eps\int_0^t
\lVert\prt_t\prt_x\hat u\rVert_{L^2_\bz}
^2\m\dd t+C(\eps)\int_0^t\lVert \prt_t\hat u\rVert_{L^2_\bz}^2\m\dd t.
\end{equation*}
%

For $R_7$, using the structure of the matrix $\chi_h$ and $B_h$ introduced in \rfb{Ah} and $\varphi$ 
in \rfb{findvarphi}, respectively, we observe that
the terms appearing 
are: \vspace{-1mm}
\begin{equation}\label{varphisecondlist}
\begin{aligned}
&c_1:=\int_0^t\int_{\Om_1}|\prt_th||\hat u||\prt_t\hat u||\prt_x\hat u|\m\dd \bz\dd t, \quad &c_2:=\int_0^t
\int_{\Om_1}|\prt_th|^2|\hat u|^2|\prt_x\hat u|\m\dd\bz\dd t,\\
& c_3:=\int_0^t\int_{\Om_1}
|\prt_th|^2|\prt_t\hat u||\prt_x\hat u|\m\dd\bz\dd t,
\quad &c_4:=\int_0^t\int_{\Om_1}|\prt_th|^3|\hat u||\prt_x\hat u|\m\dd\bz\dd t,\\
& c_5:=\int_0^t
\int_{\Om_1}|\prt_t\prt_xh||\hat u||\prt_t\hat u||\prt_z\hat u|\m\dd\bz\dd t, \quad &c_6:=\int_0^t
\int_{\Om_1}|\prt_th||\prt_t\prt_xh||\hat u|^2|\prt_z\hat u|\m\dd\bz\dd t, \\
&c_7:=\int_0^t\int_{\Om_1}|\prt_th||\prt_t\prt_xh||\prt_t\hat u||\prt_z\hat u|\m\dd\bz\dd t, \quad 
&c_8:=\int_0^t\int_{\Om_1}|\prt_th|^2|\prt_t\prt_xh||\hat u||\prt_z\hat u|\m\dd\bz\dd t.
\end{aligned}
\end{equation}
Now, $c_1$ is subcritical to $c_5$, $c_2$ to $c_6$, $c_3$ to $c_7$ and $c_4$ to $c_8$.
We estimate using again the interpolation 
inequality in 1D for $L^\infty$ as in \rfb{ineqinfty} but also, 
the interpolation inequality in 2D for $L^4$, also known as Ladyzhenskaya's estimate
$$\lVert\hat u\rVert_{L^4_\bz}\lesssim\lVert\hat u\rVert_{L^2_\bz}
^{\frac{1}{2}}\lVert\nabla\hat u\rVert_{L^2_\bz}^{\frac{1}{2}} 	
\quad\text{ for all }(x,z)\in\Om_1.$$
With that we find
\begin{equation*}
\begin{aligned}
c_5
&\leq \int_0^t\lVert\prt_t\prt_xh\rVert_{L^\infty_x}\lVert\prt_z
\hat u\rVert_{L^2_\bz}\lVert\hat u\rVert_{L^4_\bz}\lVert
\prt_t\hat u\rVert_{L^4_\bz}\m\dd t\\
&\lesssim\int_0^t\lVert\prt_t\prt_x^2h\rVert_{L^2_x}\lVert\nabla
\hat u\rVert_{L^2_\bz}^{\frac{3}{2}}\lVert\prt_t\hat u
\rVert_{L^2_\bz}^{\frac{1}{2}}\lVert\nabla\prt_t\hat u
\rVert_{L^2_\bz}^{\frac{1}{2}}\m\dd t\\
&\leq \eps\int_0^t\lVert\nabla\prt_t\hat u\rVert_{L^2_\bz}^2\m\dd t
+\int_0^t\lVert\prt_t\hat u\rVert_{L^2_\bz}^2\lVert\nabla
\hat u\rVert_{L^2_\bz}^2\m\dd t+C(\eps)\int_0^t\lVert\prt_t\prt_x^2h
\rVert_{L^2_x}^2\lVert\nabla\hat u\rVert_{L^2_\bz}^2\m\dd t,
\end{aligned}
\end{equation*}
where in the end we used Young's inequality (with three terms). 
For $c_6$ in \rfb{varphisecondlist}, we have using the same interpolations 
\begin{equation*}
\begin{aligned}
c_6
&\leq \int_0^t\lVert\prt_th\rVert_{L^\infty_x}\lVert\prt_t\prt_xh
\rVert_{L^\infty_x}\lVert\prt_z\hat u\rVert_{L^2_\bz}\lVert\hat u
\rVert_{L^4_\bz}^2\m\dd t
\lesssim\int_0^t(1+\lVert\prt_t\prt_x^2h\rVert_{L^2_x}^2)
\lVert\nabla\hat u\rVert_{L^2_\bz}^2\m\dd t.
\end{aligned}
\end{equation*}
Further \vspace{-1mm}
\begin{align*}
c_7
&\leq \int_0^t\int_0^1\lVert\prt_t\prt_xh\rVert_{L^\infty_x}\lVert
\prt_th\rVert_{L^4_x}\lVert\prt_t\hat u\rVert_{L^4_x}\lVert\prt_z
\hat u\rVert_{L^2_x}\m\dd z\dd t \\
&\lesssim\int_0^t\int_0^1\lVert\prt_t\prt_xh\rVert_{L^2_x}^{\frac{3}{4}}
\lVert\prt_t\prt_x^2h\rVert_{L^2_x}^{\frac{1}{2}}\lVert\prt_th\rVert_
{L^2_x}^{\frac{3}{4}}\lVert\prt_t\hat u\rVert_{L^2_x}^{\frac{3}{4}}\lVert
\prt_t\prt_x\hat u\rVert_{L^2_x}^{\frac{1}{4}}\lVert\prt_z\hat u\rVert
_{L^2_x}\m\dd z\dd t\\
&\leq \eps\int_0^t\lVert\prt_t\prt_x\hat u\rVert_{L^2_\bz}^2\m\dd t+C(\eps)\int_0^t
\lVert\prt_t\prt_xh\rVert_{L^2_x}^{\frac{6}{7}}\lVert\prt_t\prt_x^2h
\rVert_{L^2_x}^{\frac{4}{7}}\lVert\prt_t\hat u\rVert_{L^2_\bz}
^{\frac{6}{7}}\lVert\prt_z\hat u\rVert_{L^2_\bz}^{\frac{8}{7}}\m\dd t\\
&\leq \eps\int_0^t\lVert\prt_t\prt_x\hat u\rVert_{L^2_\bz}^2\m\dd t+\delta
\int_0^t\lVert\prt_t\hat u\rVert_{L^2_\bz}^2\m\dd t+C(\eps, \delta)\int_0^t(1+\lVert\prt_t
\prt_x^2h\rVert_{L^2_x}^{2})\lVert\prt_z\hat u\rVert_{L^2_\bz}^2\m\dd t,
\end{align*}
where we used the uniform boundedness of $\lVert\prt_th\rVert_{L^2_x}$, the following interpolation inequality:
$$\lVert\prt_t\prt_xh\rVert_{L^2_x}\lesssim\lVert\prt_th\rVert_{L^2_x}^
{\frac{1}{2}}\lVert\prt_t\prt_x^2h\rVert_{L^2_x}^{\frac{1}{2}}. $$
and Young's inequality.
Finally, the last term is estimated by
\begin{equation*}\label{smallvarphilast}
\begin{aligned}
c_8&\leq 
\int_0^t\lVert\prt_th\rVert_{L^\infty_x}^2\lVert\prt_t\prt_xh\rVert_
{L^\infty_x}\lVert\hat u\rVert_{L^2_\bz}\lVert\prt_z\hat u\rVert_
{L^2_\bz}\m\dd t\\
&\lesssim \int_0^t\lVert\prt_t\prt_x^2h\rVert_{L^2_x}^2\m\dd t+\int_0^t(1+
\lVert\nabla\hat u\rVert_{L^2_\bz}^2)\lVert\prt_t\prt_xh\rVert_{L^2_x}^2\m\dd t.
\end{aligned}
\end{equation*}
All together we have the estimate
\begin{equation*}\label{smallvarphi1fin}
\begin{aligned}
\abs{R_7}
&\leq \eps\int_0^t\lVert\nabla\prt_t\hat u\rVert_{L^2_\bz}^2\m\dd t
+C(\eps)\int_0^t\left(\lVert\prt_t\hat u\rVert_{L^2_\bz}^2\m\dd t+\lVert\prt_t\prt_x^2h
\rVert_{L^2_x}^2\right)\m\dd t\\
&+C(\eps, \delta)\int_0^t\left(1+\lVert\prt_t\hat u\rVert_{L^2_\bz}^2+\lVert\prt_t
\prt_xh\rVert_{L^2_x}^2+\lVert\prt_t\prt_x^2h\rVert_{L^2_x}^2+\lVert
\prt_t\hat u\rVert_{L^2_\bz}^2\right)\lVert\nabla\hat u\rVert_
{L^2_\bz}^{2}\m\dd t.
\end{aligned}
\end{equation*}

Now we continue the estimate for the last integral including $\varphi$ 
in \rfb{sum}, namely $R_8$.
Evaluating this term one realizes that the terms not considered before are
$$\int_0^t\int_{\Om_1}|\prt_th||\prt_x\hat u||\prt_t\hat u|\m\dd\bz\dd t, \quad \int_0^t\int_
{\Om_1}|\prt_t\prt_xh||\prt_t\hat u||\prt_z\hat u|\m\dd \bz\dd t,  $$
$$\int_0^t\int_{\Om_1}
|\prt_t\prt_x^2h||\prt_t\hat u||\prt_x\hat u|\m\dd\bz\dd t, \quad \int_0^t\int_{\Om_1}|\prt_th||\prt_x^2h||\prt_t\hat u||\prt_x\hat u|\m\dd\bz\dd t. $$
The first two terms above are subcritical to $c_5$. The third and the last terms is estimated as follows
\begin{equation*}
\begin{aligned}
\int_0^t\int_{\Om_1}|\prt_t\prt_x^2h||\prt_t\hat u||\prt_x\hat u|\m\dd\bz\dd t
&\leq \int_0^t\int_0^1\lVert\prt_t\prt_x^2h\rVert_{L^2_x}\lVert\prt_t\hat u
\rVert_{L^\infty_x}\lVert\prt_x\hat u\rVert_{L^2_x}\m\dd z\dd t\\
&\leq \int_0^t\lVert\prt_t\prt_x^2h\rVert_{L^2_x}\lVert\prt_t\hat u\rVert_
{L^2_\bz}^{\frac{1}{2}}\lVert\prt_t\prt_x\hat u\rVert_{L^2_\bz}^
{\frac{1}{2}}\lVert\nabla\hat u\rVert_{L^2_\bz}\m\dd t\\
&\leq \eps\int_0^t\lVert\prt_t\prt_x\hat u\rVert_{L^2_\bz}^2\m\dd t+\int_0^t
\lVert\prt_t\prt_x^2h\rVert_{L^2_x}^2\lVert\nabla\hat u\rVert_{L^2_\bz}^2\m\dd t
+C(\eps)\int_0^t\lVert\prt_t\hat u\rVert_{L^2_\bz}^2\m\dd t.
\end{aligned}
\end{equation*}
and 
\begin{equation*}
\begin{aligned}
\int_0^t\int_{\Om_1}|\prt_th||\prt_x^2h||\prt_t\hat u||\prt_x\hat u|\m\dd\bz\dd t&\leq \int_0^t\int_0^1\lVert\prt_th\rVert_{L^\infty_x}\lVert\prt_x^2h\rVert_{L^2_x}\lVert\prt_t\hat u\rVert_{L^\infty_x}\lVert\prt_x\hat u\rVert_{L^2_x}\dd z\dd t\\
&\leq \int_0^t\lVert\prt_t\prt_xh\rVert_{L^2_x}^{\frac{1}{2}}\lVert\prt_t\hat u\rVert_{L^2_\bz}^{\frac{1}{2}}\lVert\prt_t\prt_x\hat u\rVert_{L^2_\bz}^{\frac{1}{2}}\lVert\prt_x\hat u\rVert_{L^2_\bz}\dd t\\
&\leq \eps\int_0^t\lVert\prt_t\prt_x\hat u\rVert_{L^2_\bz}^2\dd t+\int_0^t\lVert\prt_t\hat u\rVert_{L^2_\bz}^2\dd t+\int_0^t\lVert\prt_t\prt_xh\rVert_{L^2_x}^2\lVert\prt_x\hat u\rVert_{L^2_\bz}^2\dd t.
\end{aligned}
\end{equation*}

Therefore, we have
\begin{equation*}\label{lastvarphifin}
\begin{aligned}
\abs{R_8}&\leq \eps\int_0^t\lVert
\prt_t\prt_x\hat u\rVert_{L^2_\bz}^2\m\dd t
+C(\eps)\int_0^t
\lVert\prt_t\hat u\rVert_{L^2_\bz}^2\m\dd t\\
&\quad+C(\eps)\int_0^t\left(\lVert\prt_t\prt_x^2h\rVert_{L^2_x}^2+\lVert\prt_t\prt_xh
\rVert_{L^2_x}^2\right)\lVert\nabla\hat u\rVert_{L^2_\bz}^2\m\dd t.
\end{aligned}
\end{equation*}

{\bf Step 4. The estimate of the time-derivative of the convective term in \rfb{sum}.} This last part is the most critical estimate as it is related to the non-linearity of the Navier-Stokes equation. We split the terms appearing in $R_9$ in three sub-parts:
\begin{equation*}
\begin{aligned}
&\prt_t\left(\rho_f(\hat u-\prt_t\chi_h)
\cdot(B_h\nabla)\hat u\right)\cdot(\prt_t\hat u+G)=
\rho_f\underbrace{(\prt_t\hat u-\prt_t^2\chi_h)\cdot(B_h\nabla)\hat u\cdot(\prt_t\hat u+G)}_{:=r_1}
\\
&\quad +\underbrace{(\hat u-
	\prt_t\chi_h)\cdot(\prt_tB_h\nabla)\hat u\cdot(\prt_t\hat u+G)}_{:=r_2}
+\underbrace{(\hat u-\prt_t\chi_h)\cdot(B_h\nabla)\prt_t\hat u
	\cdot(\prt_t\hat u+G)}_{:=r_3}.
\end{aligned}
\end{equation*}
We analyze 
the above three forms one by one. We first have
\begin{equation}\label{lastterm1}
\begin{aligned}
r_1
&=\left(h\prt_t\hat u^1\prt_x\hat u^1-\prt_xhz\prt_t\hat u^1\prt_x\hat u^2
+\prt_t\hat u^2\prt_x\hat u^2-\prt_t^2hz\prt_x\hat u^2\right)\left(
\prt_t\hat u^1+\frac{1}{h}\prt_th\hat u^1\right)\\
&\quad+\left(h\prt_t\hat u^1\prt_z\hat u^1-\prt_xhz\prt_t\hat u^1
\prt_z\hat u^2+\prt_t\hat u^1\prt_z\hat u^2-\prt_t^2hz\prt_z\hat u^2\right)
(\prt_t\hat u^2+\frac{z}{h}\prt_th\prt_xh\hat u^1-\prt_t
\prt_xhz\hat u^1).
\end{aligned}
\end{equation}
The terms in \rfb{lastterm1} that appear are summarized by
\begin{align*}
&r_{11}:=\int_0^t\int_{\Om_1}|\prt_t\hat u|^2\abs{\prt_x\hat u}\m\dd\bz\dd t, \quad r_{12}:=\int_0^t\int_{\Om_1}
|\prt_th||\prt_t^2h||\hat u||\prt_z\hat u|\m\dd\bz\dd t, \\ &r_{13}:=\int_0^t\int_{\Om_1}|\prt_t^2h||\prt_t
\prt_xh||\hat u||\prt_z\hat u|\m\dd\bz\dd t,
\end{align*}
where $r_{12}$ is subcritical to $r_{13}$.
We first find by Ladyzhenskaya's interpolation estimate that
\begin{equation*}
\begin{aligned}
r_{11}&\leq \int_0^t\lVert
\prt_x\hat u\rVert_{L^2_\bz}\lVert\prt_t\hat u\rVert_{L^4_\bz}^2\m\dd t\leq \eps\int_0^t\lVert\nabla\prt_t\hat u\rVert_{L^2_\bz}^2\m\dd t+C(\eps)
\int_0^t\lVert\prt_t\hat u\rVert_{L^2_\bz}^2\lVert\nabla\hat u
\rVert_{L^2_\bz}^2\m\dd t,
\end{aligned}
\end{equation*} 
and similarly to above (using 1D interpolation for $L^\infty$) we find
\begin{equation*}
\begin{aligned}
r_{13}
&\leq \int_0^t\int_0^1\lVert\prt_t^2h\rVert_{L^2_x}\lVert\prt_z\hat u
\rVert_{L^2_x}\lVert\prt_t\prt_xh\rVert_{L^\infty_x}\lVert\hat u\rVert_
{L^\infty_x}\m\dd z\dd t\\
&\lesssim \int_0^t\lVert\prt_t^2h\rVert_{L^2_x}\lVert\nabla\hat u\rVert_
{L^2_\bz}^{\frac{3}{2}}\lVert\prt_t\prt_x^2h\rVert_{L^2_x}\m\dd t\\
&\lesssim \int_0^t\lVert\prt_t^2h\rVert_{L^2_x}^2(1+\lVert\nabla\hat u\rVert_
{L^2_\bz}^2)\m\dd t+\int_0^t
\lVert\nabla\hat u\rVert_{L^2_\bz}^2\lVert\prt_t\prt_x^2h\rVert_{L^2_x}^2\m\dd t.
\end{aligned}
\end{equation*}
This allows to conclude that
\begin{equation*}
\begin{aligned}
\int_0^t\int_{\Omega_1}\abs{r_1}\dd \bz\dd t
&\leq \eps\int_0^t\lVert\nabla\prt_t\hat u\rVert_{L^2_\bz}^2\m\dd t+C(\eps)\int_0^t
\left(\lVert\prt_t^2h\rVert_{L^2_x}^2+\lVert\prt_t\prt_xh\rVert_{L^2_x}^2\right)\m\dd t\\
&\quad+C(\eps)\int_0^t\left(1+\lVert\prt_t^2h\rVert_{L^2_x}^2+\lVert\prt_t\prt_x^2h
\rVert_{L^2_x}^2+\lVert\prt_t\hat u\rVert_{L^2_\bz}^2\right)\lVert
\nabla\hat u\rVert_{L^2_\bz}^2\m\dd t.
\end{aligned}
\end{equation*}
Next we compute
\begin{equation*}
\begin{aligned}
r_2
&=\left(\hat u^1\prt_th\prt_x\hat u^1-\prt_t\prt_xhz\hat u^1\prt_x\hat 
u^2\right)\left(\prt_t\hat u^1+\frac{1}{h}\prt_th\hat u^1\right)\\
&\quad+\left(\prt_th\hat u^1\prt_z\hat u^1-\prt_t\prt_xhz\hat u^1\prt_z\hat 
u^2\right)\left(\prt_t\hat u^2+\frac{z}{h}\prt_th\prt_xh\hat u^1-
\prt_t\prt_xhz\hat u^1\right),
\end{aligned}
\end{equation*}
which can be estimated as $c_5$ and $c_7$.
Finally
\begin{equation*}
\begin{aligned}
r_3
&=\left(h\hat u^1\prt_t\prt_x\hat u^1-\prt_xhz\hat u^1\prt_t\prt_z\hat u^1
+\hat u^2\prt_t\prt_z\hat u^1-\prt_thz\prt_t\prt_z\hat u^1\right)
\left(\prt_t\hat u^1+\frac{1}{h}\prt_th\hat u^1\right)\\
&\quad+\left(h\hat u^1\prt_t\prt_x\hat u^2-\prt_xhz\hat u^1\prt_t\prt_z\hat u^2
+\hat u^2\prt_t\prt_z\hat u^2-\prt_thz\prt_t\prt_z\hat u^2\right)
\left(\prt_t\hat u^2+\frac{z}{h}\prt_th\prt_xh\hat u^1-
\prt_t\prt_xhz\hat u^1\right)
\end{aligned}
\end{equation*}
In this part, we need to estimate: 
\begin{equation*}
\begin{aligned}
&r_{31}:=\int_0^t\int_{\Om_1}\abs{\hat u}\abs{\prt_t\nabla\hat u}\abs{\prt_t\hat u}\m\dd\bz\dd t, \quad &r_{32}:=\int_0^t
\int_{\Om_1}\abs{\prt_th}|\hat u|^2\abs{\prt_t\nabla\hat u}\m\dd\bz\dd t,
\\
&r_{33}:=\int_0^t\int_{\Om_1}|\hat u|^2\abs{\prt_t\prt_xh}\abs{\prt_t\nabla\hat u}\m\dd\bz\dd t, \quad 
&r_{34}:=\int_0^t\int_{\Om_1}\abs{\prt_th}\abs{\prt_t\prt_xh}\abs{\hat u}\abs{\prt_t\nabla \hat u}\m\dd\bz\dd t,\\
&r_{35}=\int_0^t\int_{\Om_1}|\prt_th|^2|\hat u||\prt_t\prt_z\hat u|\dd \bz\dd t.
\end{aligned}
\end{equation*}
Notice that $r_{32}$ and $r_{35}$ is subcritical to $r_{33}$ and $r_{34}$, respectively.
We start (using Ladyzhenskaya's interpolation estimate once more)
\begin{equation*}
\begin{aligned}
r_{31}
&\leq \int_0^t\lVert\prt_t\nabla\hat u\rVert_{L^2_\bz}\lVert\prt_t\hat 
u\rVert_{L^4_\bz}\lVert\hat u\rVert_{L^4_\bz}\m\dd t\\
&\lesssim\int_0^t\lVert\nabla\prt_t\hat u\rVert_{L^2_\bz}^{\frac{3}{2}}
\lVert\prt_t\hat u\rVert_{L^2_\bz}^{\frac{1}{2}}\norm{\nabla\hat u}_{L^2_\bz}^{\frac{1}{2}}\m\dd t\\
&\leq \eps\int_0^t\lVert\nabla\prt_t\hat u\rVert_{L^2_\bz}^2\m\dd t+C(\eps)
\int_0^t\lVert\prt_t\hat u\rVert_{L^2_\bz}^2\lVert\nabla\hat u
\rVert_{L^2_\bz}^2\m\dd t.
\end{aligned}
\end{equation*}
Next
\begin{equation*}
\begin{aligned}
r_{33}
&\leq \int_0^t\lVert\prt_t\prt_xh\rVert_{L^\infty_x}\lVert\hat u\rVert_
{L^4_\bz}^2\lVert\prt_t\nabla\hat u\rVert_{L^2_\bz}\m\dd t\\
&\lesssim \int_0^t\lVert\prt_t\prt_x^2h\rVert_{L^2_x}^\frac{1}{2}\lVert\prt_t\prt_xh\rVert_{L^2_x}^\frac{1}{2}\lVert\hat u\rVert_
{L^2_\bz}\lVert\nabla\hat u\rVert_{L^2_\bz}\lVert\prt_t\nabla
\hat u\rVert_{L^2_\bz}\m\dd t\\
&\leq \eps\int_0^t\lVert\prt_t\prt_x\hat u\rVert_{L^2_\bz}^2\m\dd t+
C(\eps)\int_0^t\lVert\prt_t\prt_x^2h\rVert_{L^2_x}^2\lVert\nabla\hat u\rVert
_{L^2_\bz}^2\m\dd t,
\end{aligned}
\end{equation*}
where we used Young's and Poincar\'e's inequality in the last step. Finally
\begin{equation*}
\begin{aligned}
r_{34}
&\leq \int_0^t\lVert\hat u\rVert_{L^2_\bz}\lVert\prt_t\nabla\hat
u\rVert_{L^2_\bz}\lVert\prt_th\rVert_{L^\infty_x}\lVert\prt_t
\prt_xh\rVert_{L^\infty_x}\m\dd t\\
&\lesssim\int_0^t\lVert\prt_th\rVert_{L^2_x}^{\frac{1}{2}}\lVert\prt_t\prt_x
h\rVert_{L^2_x}\lVert\prt_t\prt_x^2h\rVert_{L^2_x}^{\frac{1}{2}}\lVert
\prt_t\nabla\hat u\rVert_{L^2_\bz}\m\dd t\\
&\lesssim\int_0^t\lVert\prt_th\rVert_{L^2_x}\lVert\prt_t\prt_x^2h\rVert_{L^2_x}\lVert
\prt_t\nabla\hat u\rVert_{L^2_\bz}\m\dd t\\
&\leq \eps\int_0^t\lVert\prt_t\nabla\hat u\rVert_{L^2_\bz}^2\m\dd t+C(\eps)
\int_0^t\lVert\prt_t\prt_x^2h\rVert_{L^2_x}^2\m\dd t.
\end{aligned}
\end{equation*}
The remaining integrals can be estimated directly. 

Putting all the estimates together and taking the supremum with respect to time on both sides of \rfb{sum}, we obtain from Gr\"onwall’s lemma that for $T>0$, 
\begin{equation*}\label{finalformula}
\begin{aligned}
&\sup_{t\in (0, T)}\int_{\Om_1}|\prt_t\hat u|^2\m\dd\bz+\int_0^T\int_{\Om_1}
\left|\nabla\prt_t\hat u\right|^2\m\dd\bz\dd t+\sup_{t\in (0, T)}\int_0^L\left(\rho_s\left|\prt_t^2h
\right|^2+\left|\beta\prt_t\prt_xh\right|^2+\alpha
\left|\prt_t\prt_x^2h\right|^2\right)\m\dd x\\
&\quad \lesssim \int_{\Om_1}\left|(\prt_t\hat u)(0)\right|^2\m\dd\bz+
\int_0^L\left(|(\prt_t^2h)(0)|^2+|\prt_xh_1|^2+
|\prt_x^2h_1|^2\right)\m\dd x+1
\end{aligned}
\end{equation*}
where the bound depends on $C_0$ defined in \rfb{energyhat} and $\min_{(t,x)\in [0,T]\times [0,L]}h(t,x)=:h_{\min}$ only.

\subsection{Formally obtaining the initial values for $\partial_t u$ and $\partial_t^2 h$}
Observe that we only need to know
\[
\norm{(\partial_t \hat{u})(0)}_{L^2(\Omega_1)}+\norm{(\partial_t^2 h)(0)}_{L^2_x}.
\]
{\em Assume for the moment that $\Omega(0)=\Omega_1$}. Then (formally) multiplying the equation at zero with $((\partial_t \hat{u})(0),
(\partial_t^2h)(0))$ we find that
\begin{align*}
&\rho_s\norm{(\partial_t \hat{u})(0)}_{L^2(\Omega_1)}^2+\rho_f\norm{(\partial_t^2 h)(0)}_
{L^2_x}^2
\\
&=\int_{\Omega_{1}} (\mu\Delta \hat{u}_0-\rho_f(\hat{u}_0\cdot\nabla) \hat{u}_0 )\cdot \partial_t \hat{u}(0)\, 
\dd \bz+\int_0^L ( \beta\partial_{x}^2h_0-\alpha\partial_{x}^4h_0) \partial_t^2h(0)\, \dd x\\
&\quad -2\mu\int_0^L {\rm e_2}\cdot D(u_0)(0,x,h_0)(-\prt_xh_0{\rm e_1}+{\rm e_2})(\prt_t^2h)(0)\dd x.
\end{align*} 
This implies a uniform estimate, if $\hat{u}_0\in H^2(\Omega_1)$ and $h_0
\in H^4(0,L)$. Moreover, it is linear in $\hat{u}_0\in H^2(\Omega_1)$ and $h_0
\in H^4(0,L)$. 
Rigorously the initial data for the time-derivative will be established in Step 2 in the Galerkin construction below.
%

\section{Construction of a strong solution}\label{Garlekinsec}

In order to make the estimates rigorous we have to derive the estimates on an approximate or mollified level. This is typical for higher-order in time estimates. Please note that due to the (weak-strong)-uniqueness constructing a smooth solution implies that it is unique.

The key that in order of being able to rigorously differentiating the equation in time is the correct treatment of the pressure. For that we introduce the following solenoidal substitute for the fluid velocity
$
\hat{v}=B_h^\intercal\hat{u},
$
and in particular please note that
\begin{align*}\label{eq:key-identity}
B_h^{-\intercal}\partial_t \hat{v}=\partial_t\hat{u}+B_h^{-\intercal}\partial_tB_h^\intercal\hat{u}=\partial_t\hat{u}+G,
\end{align*}
where $G$ is defined in \eqref{G}.

We use a Galerkin approximation on which we then rigorously can perform the key estimate. For that the following steps will now be performed:
\begin{enumerate}
	\item[4.1] We prepare the partial differential equation and its time-derivative such that it can be approximated by a Galerkin method. For that we introduce a weak equation in terms of solenoidal test functions and decouple the geometry from the equation.
	\item[4.2] We show the existence of a space-discrete and decoupled solution to the time-differentiated equation that solves an energy estimate. By a fixed-point argument we get the desired coupled approximation sequence.
	\item[4.3] We show that the a-priori estimates derived in \eqref{formal-est} is valid on the discrete level.
	\item[4.4] We construct a strong solution. For that we pass to the limit with the approximation and show further spatial regularity using the properties of the steady Stokes operator. This establishes a strong solution. 
\end{enumerate}
Hence the strategy is to derive a discretized weak formulation for $\hat{v}$, which we will approximate by a solenoidal Galerkin basis.
This next will be (formally) differentiated in time.

\subsection{Preperation}
%


Following the convention for Galerkin schemes we consider the following weak formulation for $(\hat u, h)$ a solution to
system \rfb{hatsys}--\rfb{priohat}. We introduce the space of $L$-periodic smooth functions
\[
C^\infty_{\text{per}}(0,L):=\{f\in C^\infty(\mathbb{R})|f(\cdot)\equiv f(\cdot+L)\},
\]
and 
\[
C^\infty_{\text{per},0}(\Omega_1):=\{f\in C^\infty(\mathbb{R}\times [0,1])|f(\cdot,y)\in C^\infty_{\text{per}}(0,L)\text{ and }f(\cdot,0)=0 \}.
\]
For every $( \phi,\Phi)\in C^\infty([0,T];C^\infty_{\text{per}}(0,L))\times
C^\infty([0, T];C^\infty_{\text{per},0}(\Omega_1)))$, with
$\div B^\intercal\Phi=0$ 
in $\Om_1$, $\Phi(t, x, 1)=\phi(t,x)\m {\rm e_2}$, $\Phi(t,x,0)=0$ and  $\Phi(t, 0, z)=\Phi(t, L, z)$,
we have for all time values
\vspace{+2mm}
\begin{equation}\label{weakformphi}
\begin{aligned}
&\int_{\Om_1}\rho_fh\prt_t\hat u\cdot \Phi\m\dd\bz+
\frac{1}{2}\int_{\Om_1}\rho_f\hat u(B_h\nabla)
\hat u\cdot \Phi\m\dd\bz-\frac{1}{2}\int_{\Om_1}\rho_f\hat 
u(B_h\nabla)\Phi\cdot\hat u\m\dd\bz\\
&+\frac{1}{2}\int_0^L\rho_f|\prt_th|^2\phi\m\dd x
-\int_{\Om_1}
\rho_f\prt_t\chi_h(B_h\nabla)\hat u\cdot\Phi\m\dd\bz
+\mu\int_{\Om_1}(A_h\nabla)\hat u:\nabla\Phi\m\dd\bz\\
&+\rho_s\int_0^L\prt_t^2h\phi\m\dd x-\beta\int_0^L\prt_x^2h
\phi\m\dd x+\alpha\int_0^L\prt_x^4h\phi\m\dx
=0,
\end{aligned}\vspace{+2mm}
\end{equation}
with the relation $\hat u(t, x, 1)=\prt_th\m {\rm e_2}$ for $(t, x)\in (0, T)\times (0, L)$.

{The above weak formulation can be derived by multiplying the strong equation with $\phi$ and $\Phi$ respectively, integrating in time-space and performing integration by parts. In particular we used the following 
	form for the convective term:\vspace{+2mm}
	\begin{equation}\label{convective}
	\begin{aligned}
	\int_{\Om_1}\rho_f\hat u(B_h\nabla)\hat u\cdot
	\Phi\m\dd\bz&=\frac{1}{2}\int_{\Om_1}\rho_f\hat u(B_h\nabla)
	\hat u\cdot\Phi\m\dd\bz
	-\frac{1}{2}\int_{\Om_1}\rho_f\hat u(B_h\nabla)\Phi
	\cdot\hat u\m\dd\bz +\frac{1}{2}\int_0^L\rho_f|\prt_th|^2\phi\m\dd x,
	\end{aligned}\vspace{+2mm}
	\end{equation}
	which plays an important role in the construction of the approximate 
	solution and in particular preserves the structure of the energy 
	estimate in the Galerkin procedure.
}

Following the weak formulation \rfb{weakformphi}, we define
$
\Psi(x,z)=\bbm{\Psi^1 & \Psi^2}^\intercal :=B_{h}^\intercal\Phi$,
which gives that $\div \Psi=0$ in $\Om_1$. It is worthwhile noting 
that $\Psi$ is independent of the time variable. 
Recalling the structure of the matrix $B_h$ (please refer to 
\rfb{Ah}), with $\Phi=\bbm{\Phi^1 &\Phi^2}^\intercal$ we have \vspace{+2mm}
$$\Psi=\begin{bmatrix} h & 0  \\ -z\prt_xh  & 1 \m 
\end{bmatrix}\begin{bmatrix} \Phi^1 \\ \Phi^2 \end{bmatrix}=
\begin{bmatrix} h\Phi^1 \\ -z\prt_xh\Phi^1+
\Phi^2\end{bmatrix}. \vspace{+2mm}$$
Note that $\Phi^1(x, 1)=0$, therefore we obtain that
$\Psi(x,1)=\Phi^2(x,1)\m{\rm e_2}$. This 
gives the corresponding weak formulation 
in terms of $\Psi$. For every $( \psi,\Psi)\in C^\infty_{\text{per}}(0,L))\times
C^\infty(\Omega_1)$, with
$\div \Psi=0$ 
in $\Om_1$, $\Psi( x, 1)=\psi(x)\m {\rm e_2}$, $\Psi(x,0)=0$ and  $\Psi( 0, z)=\Phi( L, z)$ we have for all times \vspace{+1mm}
\begin{equation}\label{weakformpsi}
\begin{aligned}
&\int_{\Om_1}\rho_fh\prt_t\hat u\cdot B_h^{-\intercal}
\Psi\m\dd\bz+\frac{1}{2}\int_{\Om_1}\rho_f\hat u(B_h\nabla)
\hat u\cdot B_h^{-\intercal}\Psi\m\dd\bz
-\frac{1}{2}\int_{\Om_1}\rho_f\hat u(B_h\nabla)
(B_h^{-\intercal}\Psi)\cdot\hat u\m\dd\bz\\
&+\frac{1}{2}\int_0^L\rho_f|\prt_th|^2\psi\m\dd x
-\int_{\Om_1}\rho_f\prt_t\chi_h(B_h\nabla)\hat u
\cdot B_h^{-\intercal}\Psi\m\dd\bz+\mu\int_{\Om_1}(A_h\nabla)
\hat u:\nabla(B_h^{-\intercal}\Psi)\m\dd\bz\\
&+\rho_s\int_0^L\prt_t^2h\psi\m\dd x-\beta\int_0^L\prt_x^2h\psi\dd x
+\alpha\int_0^L\prt_x^4h\psi\m\dd x
=0,
\end{aligned}
\end{equation}
with the relation $\hat u(t, x, 1)=\prt_th\m {\rm e_2}$ for $(t, x)\in (0, T)\times (0, L)$.

Next we differentiate \rfb{weakformpsi} in time. This implies for the same set of test functions and for every 
$t\in (0, T)$, which is given by\footnote{From here on we use scalar products for integrals to shorten notation.}
\begin{align*}
&\left\langle \rho_f\prt_th\prt_t\hat u, B_h
^{-\intercal}\Psi\right\rangle+\left\langle \rho_fh\prt_t^2\hat u, B_h^{-\intercal}\Psi\right\rangle+\left\langle \rho_fh
\prt_t\hat u, \prt_t B_h^{-\intercal}\Psi\right\rangle
\\
&+
\frac{1}{2}\left\langle \rho_f\prt_t\hat u(B_h\nabla)
\hat u, B_h^{-\intercal}\Psi\right\rangle+\frac{1}{2}
\left\langle \rho_f\hat u(\prt_t B_h \nabla)
\hat u, B_h^{-\intercal}\Psi\right\rangle
+\frac{1}{2}\left\langle\rho_f\hat u(B_h\nabla)\prt_t
\hat u, B_h^{-\intercal}\Psi\right\rangle
\\
&+\frac{1}{2}\left\langle \rho_f\hat u(B_h\nabla)\hat
u, \prt_t B_h^{-\intercal}\Psi\right\rangle
-
\frac{1}{2}\left\langle \rho_f\prt_t\hat u(B_h\nabla)
(B_h^{-\intercal}\Psi), \hat u\right\rangle
-
\frac{1}{2}\left\langle\rho_f\hat u(\prt_tB_h\nabla)
(B_h^{-\intercal}\Psi), 
\hat u\right\rangle
\\
&-\frac{1}{2}\left\langle \rho_f\hat u(B_h\nabla)
(\prt_t B_h^{-\intercal}\Psi), 
\hat u\right\rangle
-\frac{1}{2}\left\langle \rho_f\hat
u(B_h\nabla)(B_h^{-\intercal}\Psi), \prt_t
\hat u\right\rangle
+
\left\langle \rho_f\prt_t^2h\prt_th, 
\psi\right\rangle
\\
&-\left\langle\rho_f\prt_t^2\chi_h(B_h\nabla)\hat
u, B_h^{-\intercal}\Psi\right\rangle
-\left\langle
\rho_f\prt_t\chi_h(\prt_tB_h\nabla)\hat u,
B_h^{-\intercal}\Psi\right\rangle
-\left\langle \rho_f\prt_t\chi_h(B_h\nabla)\prt_t
\hat u, B_h^{-\intercal}\Psi\right\rangle
\\
&-\left
\langle \rho_f\prt_t\chi_h(B_h\nabla)\hat u,
\prt_t B_h^{-\intercal}\Psi\right\rangle
+\left\langle \mu (\prt_tA_h\nabla)\hat u,
\nabla(B_h^{-\intercal}\Psi)\right\rangle
+\left\langle
\mu(A_h\nabla)\prt_t\hat u, \nabla(B_h
^{-\intercal}\Psi)\right\rangle
\\
&+\left\langle \mu(A_h\nabla)
\hat u, \nabla(\prt_tB_h^{-\intercal}\Psi)
\right\rangle
+\left\langle\rho_s\prt_t^3h, \psi\right\rangle
-\left
\langle\beta
\prt_t\prt_x^2h, \psi\right\rangle
+\left\langle\alpha
\prt_t\prt_x^2h, \prt_x^2\psi\right\rangle
=0,\vspace{+2mm}
\end{align*}
and $\partial_t\hat u(t,x,1)=\partial_t^2h(t,x){\rm e_2}$ for all $(t,x)\in (0, T)\times (0, L)$.\vspace{+2mm}

Next this equation will be {\em decoupled}. 
Let us consider a geometry given by the function $\tilde{h}\in C^2([0,T]\times [0,L]; [h_{\min},h_{\max}])$. Respectively we define the coefficients 
$$\tilde{B}=B_{\tilde{h}},\qquad  \tilde{A}=A_{\tilde{h}}, \qquad \tilde{\chi}=\chi_{\tilde{h}},$$
which are now independent of the solution. Then we aim to solve the following decoupled system.

We look for the coupled solution $(\hat{u}, g)$ that is periodic in $[0,L]$ satisfying $\div {\tilde{B}}^{\intercal} \hat{u}=0$, $\hat u(t,x,1)=\prt_tg(t,x){\rm e_2}$, $\hat u(t,x,0)=0$,
and 
\begin{equation}\label{eq:decoupled}
\begin{aligned}
&\left\langle \rho_f\prt_t\tilde{h}\prt_t\hat u, {\tilde{B}}
^{-\intercal}\Psi\right\rangle+\left\langle \rho_f\tilde{h}\prt_t^2\hat u, {\tilde{B}}^{-\intercal}\Psi\right\rangle+\left\langle \rho_f\tilde{h}
\prt_t\hat u, \prt_t{\tilde{B}}^{-\intercal}\Psi\right\rangle\\
&+\frac{1}{2}\left\langle \rho_f\prt_t\hat u(\tilde{B}\nabla)
\hat u, {\tilde{B}}^{-\intercal}\Psi\right\rangle+\frac{1}{2}
\left\langle \rho_f\hat u(\prt_t{\tilde{B}}\nabla)
\hat u, {\tilde{B}}^{-\intercal}\Psi\right\rangle\\
&+\frac{1}{2}\left\langle\rho_f\hat u(\tilde{B}\nabla)\prt_t
\hat u, {\tilde{B}}^{-\intercal}\Psi\right\rangle
+\frac{1}{2}\left\langle \rho_f\hat u(\tilde{B}\nabla)\hat
u, \prt_t{\tilde{B}}^{-\intercal}\Psi\right\rangle\\
&-\frac{1}{2}\left\langle \rho_f\prt_t\hat u(\tilde{B}\nabla)
({\tilde{B}}^{-\intercal}\Psi), \hat u\right\rangle-
\frac{1}{2}\left\langle\rho_f\hat u(\prt_t\tilde{B}\nabla)
({\tilde{B}}^{-\intercal}\Psi), 
\hat u\right\rangle\\
&-\frac{1}{2}\left\langle \rho_f\hat u(\tilde{B}\nabla)
(\prt_t{\tilde{B}}^{-\intercal}\Psi), 
\hat u\right\rangle-\frac{1}{2}\left\langle \rho_f\hat 
u(\tilde{B}\nabla)({\tilde{B}}^{-\intercal}\Psi), \prt_t
\hat u\right\rangle\\
&+\frac{1}{2}\left\langle \rho_f\prt_t^2\tilde{h}\prt_tg, 
\psi\right\rangle
+\frac{1}{2}\left\langle \rho_f\prt_t\tilde{h}\prt_t^2g, 
\psi\right\rangle\\
&-\left\langle\rho_f\prt_t^2\tilde{\chi}(\tilde{B}\nabla)\hat
u, {\tilde{B}}^{-\intercal}\Psi\right\rangle-\left\langle
\rho_f\prt_t\tilde{\chi}(\prt_t\tilde{B}\nabla)\hat u,
{\tilde{B}}^{-\intercal}\Psi\right\rangle\\
&-\left\langle \rho_f\prt_t\tilde{\chi}(\tilde{B}\nabla)\prt_t
\hat u, {\tilde{B}}^{-\intercal}\Psi\right\rangle-\left
\langle \rho_f\prt_t\tilde{\chi}(\tilde{B}\nabla)\hat u,
\prt_t{\tilde{B}}^{-\intercal}\Psi\right\rangle\\
&+\left\langle \mu (\prt_t\tilde{A}\nabla)\hat u,
\nabla({\tilde{B}}^{-\intercal}\Psi)\right\rangle+\left\langle
\mu(\tilde{A}\nabla)\prt_t\hat u, \nabla({\tilde{B}}
^{-\intercal}\Psi)\right\rangle+\left\langle \mu(\tilde{A}\nabla)
\hat u, \nabla(\prt_t{\tilde{B}}^{-\intercal}\Psi)
\right\rangle\\
&+\left\langle\rho_s\prt_t^3g, \psi\right\rangle-\left
\langle\beta
\prt_t\prt_x^2g, \psi\right\rangle+\left\langle\alpha
\prt_t\prt_x^2g, \prt_x^2\psi\right\rangle
=0,
\end{aligned}
\end{equation}
for all times and every $( \psi,\Psi)\in C^\infty_{\text{per}}(0,L))\times
C^\infty(\Omega_1)$, with
$\div \Psi=0$ 
in $\Om_1$, $\Psi( x, 1)=\psi(x)\m {\rm e_2}$, $\Psi(x,0)=0$ and  $\Psi( 0, z)=\Phi( L, z)$.

\begin{rmk}
	{\rm We remark here that in \rfb{eq:decoupled}, by fixing the geometry $\tilde h$ and using $g$, we decouple $\langle \rho_f\prt_t^2h\prt_th, \psi\rangle$ into: 
		$$\frac{1}{2}\langle \rho_f\prt_t^2\tilde h\prt_tg, \psi\rangle+\frac{1}{2}\langle \rho_f\prt_t\tilde h\prt_t^2g, \psi\rangle.$$ 
		This formulation plays an important role in the derivation of the energy estimate of \rfb{weakform1st} in Proposition \ref{firstenergy}.
	}
\end{rmk}
\subsection{Galerkin approximations}
For every time-value we will seek a solution in the setting of the following Sobolev spaces.
The periodic Sobolev spaces are introduced as
\[
W^{k,p}_{per}(0,L):=\overline{C^\infty_{per}(0,L)}^{\norm{\cdot}_{W^{k,p}(0,L)}},
\]
for $1\leq p<\infty$ and $k\in \mathbb{N}$.

We introduce the space 
\[
\dot{H}^2_{per}(0, L):=\left\{ \check \psi\in H^2_{per}(0,L)|\int_0^L\check \psi\, \dd x=0\right\},
\]
which is a Hilbert space w.r.t.\ the scalar product $(\check \psi,\xi)_{\dot{H}^2(0,L)}:=\int_0^L\partial_x^2\check \psi\m\partial_x^2\xi\, \dd x$ (this is indeed a Hilbert space as the affine functions in periodic spaces are constants).  
 We take $(\check \psi_k)_{k\in \mathbb{N}}$ as a basis of $\dot{H}_{per}^2(0,L)$ that is orthonormal in $L^2(0,L)$.\footnote{These functions are precisely of the form $a_0\frac{\sin(\frac{k\pi}{L}x)}{k^2}$ and $a_0\frac{\cos(\frac{k\pi}{L}x)}{k^2}$, with $a_0$ depending on $L$.} 
 
Further for the fluid velocities we consider
\[
V_{per,0}^1(\Omega_1):=\left\{\hat\Psi\in H^1_{per,0}(\Omega_1)| \m \div \hat\Psi=0\right\},
\]
where the subscript "$0$" has the similar meaning with the one introduced above \rfb{weakformphi}, i.e. $\hat \Psi(\cdot, 0)=0$.
This Hilbert space $V_{per,0}^1(\Omega_1)$ is w.r.t.\ the scalar product
\[
(\hat \Psi,\hat w)_{V_{per,0}^1(\Omega_1)}=\int_{\Omega_1}\nabla\hat \Psi\cdot \nabla\hat w\, \dd x\, \dd y.
\]
We will decompose this space into
\[
V_{per,0,0}^1(\Omega_1):=\left\{\hat\Psi\in V_{per,0}^1(\Omega_1)|\m\hat\Psi(\cdot,1)=0\right\}
\]
and its orthogonal complement $\left(V_{per,0,0}^1(\Omega_1)\right)^\perp$. We will construct bases for each part separately using the Stokes operator.

First we collect all  eigenfunctions of the following Stokes problem:
\begin{equation*}
\left\{\begin{aligned}
&-\Delta\hat \Psi_k+\nabla\hat  p_k=\lambda_k\hat \Psi_k \quad 
&\text{in}\quad \Om_1,\\
&\div \hat \Psi_k=0 \quad &\text{in}\quad \Om_1,\\
&\hat \Psi_k(\cdot,0)=0=\hat\Psi_k(\cdot,1) \quad &\text{on}\quad [0,L],
\\
&\hat \Psi_k(\cdot,z)=\hat\Psi_k(\cdot+L,z)\quad &\text{for all }z\in [0,1].
\end{aligned}\right.
\end{equation*}
where $\lambda_k$ is the corresponding eigenvalue of the eigenfunction 
$\hat \Psi_k$ for every $k\in\nline$. This provides a smooth basis of 
$V_{per,0,0}^1(\Omega_1)$ which is orthonormal in $V_{per,0}^1(\Omega_1)$ and orthogonal in $L^2(\Omega_1)$, respectively.

Next we construct $(\check \Psi_k)_{k\in\nline}$ by extending $(\check \psi_k)_{k\in 
	\nline}$, the basis of $\dot{H}_{per}^2(0,L)$, through the system:
\begin{equation*}
\left\{\begin{aligned}
&-\Delta\check \Psi_k+\nabla\check p_k=0&\text{in}\quad \Om_1,\\
&\div \check \Psi_k=0 &\text{in}\quad \Om_1,\\
& \check \Psi_k=
\left\{\begin{aligned}
&\check \psi_k(x)\m {\rm e_2}\quad &(0, L)\times\{z=1\},\\
& 0 \quad & (0, L)\times\{z=0\}.
\end{aligned}\right.
\\
&\check \Psi_k(\cdot,z)=\check\Psi_k(\cdot+L,z)\quad &\text{for all }z\in [0,1].
\end{aligned}
\right.
\end{equation*}
These functions are smooth and linearly independent. Please observe that for $\hat w \in V_{per,0}^1(\Omega_1)$ we find that
\begin{align*}
0=&\int_{\Omega_1}(-\Delta\check \Psi_k+\nabla\check p_k) \hat w\, \dd x\, \dd y
= \int_{\Omega_1}\nabla \check \Psi_k \cdot \nabla \hat w\, \dd x\, \dd y\\
&\quad -\int_0^L(\prt_x\check \Psi_k^2)(\cdot, 1)\hat w^1(\cdot, 1)\dd x + \int_{0}^L (\check p_k -\partial_y\check \Psi_k^2)(\cdot, 1) \hat w^2(\cdot ,1)\, \dd x,
\end{align*}
which implies that $ \check \Psi_k\in  (V_{per,0,0}^1(\Omega_1))^\perp$.

Now we define $(\Psi_k, \psi_k)_{k\in\nline}$ in a way of enumeration by 
\begin{equation}\label{basispsik}
\left(\Psi_k, \psi_k\right):=
\left\{\begin{aligned}
&\left(\check \Psi_{\frac{k+1}{2}}, \check \psi_{\frac{k+1}{2}}\right)
\quad &k\m\m\m \text{odd}\\
&\left(\hat \Psi_{\frac{k}{2}}, 0\right)\quad &k\m\m\m\text{even}
\end{aligned}
\right.\qquad\text{ for all }k\in\nline.
\end{equation}
The $\text{span} \left\{(\Psi_k,\m\psi_k)\m\left|\m k\in \nline\right.
\right\}$
consists of all couples
$(\Psi, \psi)\in V^1_{per,0}(\Om_1)\times \dot{H}_{per}^2(0, L) $ with $\Psi(x,1)=\psi\m{\rm e_2}$ for every $x\in (0, L)$. We observe further that the space $\text{span} \{(\tilde{B}
^{-\intercal}\Psi_k, \psi_k)| k\in\nline\}$ consists of the test 
function pairs:
$$(\Phi, \phi)\in\left\{(\Phi, \phi)\in H^1_{per,0}(\Om_1)\times 
\dot{H}_{per}^2(0, L)|\m \div (\tilde{B}^{\intercal}\Phi)=0\m\m\m \text{in}
\m \m\m\Om_1, \m \m\m \Phi(t, x, 1)=\phi \m{\rm e_2}\right\}, $$ which are $C^2$ in space-time since 
$\tilde B$ is $C^2$ in space-time.
On that space we introduce the projection
$\Pscr_N$, that is defined as 
\begin{equation}\label{projec}
\begin{aligned}
&\Pscr_N\phi:=\sum_{k=1}^N P^k(\phi)\psi_k:=\sum_{k=1}^N(
\phi, \psi_k)_{H^2_{per,0}(0, L)}\psi_k
\\
&\Pscr_N\Phi:=\tilde{B}^{-\intercal}\sum_{k=1}^NP^k(\Phi)\Psi_k
=\tilde{B}^{-\intercal}\Big(\sum_{k\leq N,k\text{ odd}}P^k(\phi)\Psi_k+\sum_{k\leq N,k\text{ even}}(\Psi_k,\tilde{B}^{\intercal}\Phi)_{V^1_{per,0}(\Omega_1)}\Psi_k\Big).
\end{aligned}
\end{equation}
The following properties of the projection follow by standard Hilbert space theory and standard regularity theory for Stokes system:
\begin{itemize}
\item For any $m\in \mathbb{N}$, if (additionally) $\phi \in H^m$,$\norm{\Pscr_N\phi}_{H^m(0,L)}\leq c\norm{\phi}_{H^m(0,L)}$ and $\Pscr_N\phi\to \phi$ in $H^m(0,L)$.
\item We can decompose $\tilde{B}^{\intercal}\Phi=\Phi_0+\Phi_1=(\tilde{B}^{\intercal}\Phi-\Phi_1)+\Phi_1$, where 
\begin{equation*}
\left\{\begin{aligned}
&-\Delta\Phi_1+\nabla p_1=0&\text{in}\quad \Om_1,\\
&\div  \Phi_1=0 &\text{in}\quad \Om_1,\\
& \Phi_1=
\left\{\begin{aligned}
&\phi\m {\rm e_2}\quad &(0, L)\times\{z=1\},\\
& 0 \quad & (0, L)\times\{z=0\}.
\end{aligned}\right.
\\
&\Phi_1(\cdot,z)=\Phi_1(\cdot+L,z)\quad &\text{for all }z\in [0,1].
\end{aligned}
\right.
\end{equation*}
\item We find that $\Phi_1$ is approximated by its boundary values. In particular, by Stokes theory and the properties of the eigenfunctions, we find that
\[
\norm{\sum_{k\leq N,k\text{ odd}}P^k(\phi)\Psi_k-\Phi_1}_{H^{\frac{5}{2}}(\Omega_1)}\leq c \norm{\sum_{k\leq N,k\text{ odd}}P^k(\phi)\psi_k - \phi}_{H^2(0,L)}\to 0\text{ with }N\to \infty.
\]
\item If (additionally) $ \Phi\in H^2$, we find that $\Phi_0$ is also approximated in $H^2$ by the second projection part, i.e. for $k$ even. For that observe that
\[
(\Psi_k,\tilde{B}^{\intercal}\Phi)_{V^1_{per,0}(\Omega_1)}
=(\Psi_k,\Phi_0)_{V^1_{per,0}(\Omega_1)}+\underbrace{(\Psi_k,\Phi_1)_{V^1_{per,0}(\Omega_1)}}_{=0},
\]
where we used the equation of $\Phi_1$ and the fact that $\Psi_k$ has zero boundary values in $\{z=1\}$.
Hence we get by the properties of the basis 
$\{\Psi_k\}_{k\text{ even}}$, that
\[
\norm{\sum_{k\leq N,k\text{ even}}(\Psi_k,\tilde{B}^{\intercal}\Phi)_{V^1_{per,0}(\Omega_1)}\Psi_k-\Phi_0}_{H^2(\Omega_1)}\to 0\text{ with }N\to \infty.
\]
\item Combining the arguments above we find that
\[
\norm{\Pscr_N\Phi}_{H^2(\Omega_1)}\leq C\norm{\Phi}_{H^2(\Omega_1)}+c\norm{\phi}_{H^2(0,L)}.
\]
Moreover, $\Pscr_N\Phi\to \Phi$ in $H^2(\Omega_1)$ as $N\to \infty$. Note that based on the similar but simpler analysis, the above estimate for $\Pscr_N\Phi$ holds also for the $H^1$ and $L^2$-norm.
\end{itemize}

We make the following ansatz, for every fixed $N\in\nline$:
\begin{equation}\label{ansatz}
\begin{aligned}
\hat u_{N}(t,x,z):=\tilde{B}^{-\intercal}\sum_{k=1}^N
\alpha^k_N(t)\Psi_k(x,z) \quad\text{ for all }t\in(0, T),\\
g_{N}(t,x):=\int_0^t\sum_{k=1}^N\alpha_N^k(\tau)\psi_k(x)\dd 
\tau+\Pscr_Nh_0\quad\text{ for all }t\in(0, T).
\end{aligned}
\end{equation}

From the construction \rfb{ansatz}, we see that $\hat u_N(t,x,1)=\prt_tg_N\m{\rm e_2}$ for $(t, x)\in (0, T)\times (0, L)$.

We construct now the solution $(\hat u_N, g_N)$ in four steps.


{\bf Step 1: Existence of ${\boldsymbol \alpha_N}$ for a given geometry.}
Assume that $\tilde{h}\in C^2([0,T]\times [0,L]; [h_{\min},h_{\max}])$ is the given geometry and we still use the notation $\tilde{B}=B_{\tilde{h}}, \tilde{A}=A_{\tilde{h}}$ and $\tilde{\chi}=\chi_{\tilde{h}}$.

In what follows we seek for the a couple of discrete solutions 
$(\hat u_{N}, g_{N})$ of the form \rfb{ansatz} 
with time-dependent coefficients ${\boldsymbol \alpha_N}=(\alpha_N^k)
_{k=1}^N$, which solve the following discrete equation, for $k=1, 
\cdots, N$:
\begin{align}
\begin{aligned}
&\left\langle \rho_f\prt_t\tilde{h}\prt_t\hat u_{N}, 
\tilde{B}^{-\intercal}\Psi_k\right\rangle+\left\langle 
\rho_f\tilde{h}\prt_t^2\hat u_{N}, \tilde{B}^{-\intercal}
\Psi_k\right\rangle+\left\langle \rho_f\tilde h
\prt_t\hat u_{N}, \prt_t\tilde{B}^{-\intercal}
\Psi_k\right\rangle\\
&+\frac{1}{2}\left\langle \rho_f\prt_t\hat u_{N}\left(\tilde B\nabla\right)\hat u_{N}, \tilde{B}^{-\intercal}
\Psi_k\right\rangle+\frac{1}{2}\left\langle \rho_f\hat u_{N}
\left(\prt_t\tilde{B}\nabla\right)
\hat u_{N}, \tilde{B}^{-\intercal}\Psi_k\right\rangle\\
&+\frac{1}{2}\left\langle\rho_f\hat u_{N}\left(\tilde{B}
\nabla\right)\prt_t\hat u_{N}, \tilde{B}^{-\intercal}\Psi_k
\right\rangle+\frac{1}{2}\left\langle \rho_f\hat u_{N}
\left(\tilde{B}\nabla\right)\hat u_{N}, \prt_t\tilde{B}
^{-\intercal}\Psi_k\right\rangle\\
&-\frac{1}{2}\left\langle \rho_f\prt_t\hat u_{N}\left(\tilde{B}
\nabla\right)\left(\tilde{B}^{-\intercal}
\Psi_k\right), \hat u_{N}\right\rangle-\frac{1}{2}\left\langle
\rho_f\hat u_{N}\left(\prt_t\tilde{B}\nabla\right)\left(\tilde B^{-\intercal}\Psi_k\right), \hat u_{N}\right\rangle\\
&-\frac{1}{2}\left\langle \rho_f\hat u_{N}\left(\tilde{B}
\nabla\right)\left(\prt_t\tilde{B}^{-\intercal}\Psi_k\right), 
\hat u_{N}\right\rangle-\frac{1}{2}\left\langle \rho_f\hat u_{N}\left(\tilde{B}\nabla\right)\left(\tilde{B}^{-\intercal}
\Psi_k\right), \prt_t\hat u_{N}\right\rangle\\
&+\frac{1}{2}\left\langle \rho_f\prt_t^2\tilde{h}\prt_tg_{N}, 
\psi_k\right\rangle
+\frac{1}{2}\left\langle \rho_f\prt_t\tilde h\prt_t^2g_{N}, 
\psi_k\right\rangle
\\
&-\left\langle\rho_f\prt_t^2\tilde{\chi}(\tilde{B}\nabla)\hat 
u_{N}, \tilde{B}^{-\intercal}\Psi_k\right\rangle-\left\langle
\rho_f\prt_t\tilde{\chi}(\prt_t\tilde{B}\nabla)\hat u_{N},
\tilde{B}^{-\intercal}\Psi_k\right\rangle\\
&-\left\langle \rho_f\prt_t\tilde{\chi}(\tilde{B}\nabla)\prt_t
\hat u_{N}, \tilde{B}^{-\intercal}\Psi_k\right\rangle-\left
\langle \rho_f\prt_t\tilde{\chi}(\tilde{B}\nabla)\hat u_{N},
\prt_t\tilde{B}^{-\intercal}\Psi_k\right\rangle\\
&+\left\langle \mu (\prt_t\tilde{A}\nabla)\hat u_{N},
\nabla(\tilde{B}^{-\intercal}\Psi_k)\right\rangle+\left\langle
\mu(\tilde{A}\nabla)\prt_t\hat u_{N}, \nabla(\tilde{B}
^{-\intercal}\Psi_k)\right\rangle\\
&+\left\langle \mu(\tilde{A}\nabla)
\hat u_{N}, \nabla(\prt_t\tilde{B}^{-\intercal}\Psi_k)
\right\rangle+\left\langle\rho_s\prt_t^3g_{N}, \psi_k\right\rangle-\left
\langle\beta
\prt_t\prt_x^2g_{N}, \psi_k\right\rangle\\
&+\left\langle\alpha
\prt_t\prt_x^2g_{N}, \prt_x^2\psi_k\right\rangle
=0,
\end{aligned}
\label{weaktimedisctete}
\end{align}
where $(\Psi_k, \psi_k)$ are introduced in \rfb{basispsik}.  

According to the definition of $\hat u_{N}$ and $g_{N}$ 
in \rfb{ansatz}, we have
\begin{equation}\label{express}
\begin{aligned}
&\prt_t\hat u_{ N}=\sum_{j=1}^N\prt_t\tilde{B}^{-\intercal}
{\alpha_N^j}(t)\Psi_j+\sum_{j=1}^N\tilde{B}^{-\intercal}
{\alpha_N^j}'(t)\Psi_j,\\
&\prt_t^2\hat u_{ N}=\sum_{j=1}^N\prt_t^2\tilde{B}
^{-\intercal}{\alpha_N^j}(t)\Psi_j
+\sum_{j=1}^N2\prt_t\tilde{B}^{-\intercal}{\alpha_N^j}'(t)\Psi_j
+\sum_{j=1}^N\tilde{B}^{-\intercal}{\alpha_N^j}''(t)\Psi_j, \\
&\nabla\left(\prt_t\hat u_{ N}\right)=\sum_{j=1}^N{\alpha_N^j}'(t)
\nabla\left(\tilde{B}^{-\intercal}\Psi_j\right)+\sum_{j=1}^N
{\alpha_N^j}(t)\nabla\left(\prt_t\tilde{B}^{-\intercal}\Psi_j\right), \\
&\prt_tg_{ N}=\sum_{j=1}^N{\alpha_N^j}(t)\psi_j, \quad \prt_t^2
g_{ N}=\sum_{j=1}^N{\alpha_N^j}'(t)\psi_j, \quad \prt_t^3
g_{ N}=\sum_{j=1}^N{\alpha_N^j}''(t)\psi_j.
\end{aligned}
\end{equation}

We substitute \rfb{express} in the system \rfb{weaktimedisctete} and obtain 
the second-order nonlinear ODE system for ${\boldsymbol \alpha_N}
=(\alpha_N^k)_{k=1}^N$ as follows:
\begin{equation}\label{ODE}
\Ascr(t){\boldsymbol \alpha_N}''(t)+\Bscr(t){\boldsymbol \alpha_N}'(t)+
\Cscr(t){\boldsymbol \alpha_N}(t)+\left(\Dscr(t)\cdot{\boldsymbol \alpha_N}'(t)
\right){\boldsymbol \alpha_N}(t)+\left(\Escr(t)\cdot{\boldsymbol \alpha_N}(t)
\right){\boldsymbol \alpha_N}(t)=0,
\end{equation}
where the coefficients matrices $\Ascr=(\Ascr_{j,k})_{j,k=1}^N$, $\Bscr=(\Bscr_{j,k})_{j,k=1}^N$, 
$\Cscr=(\Cscr_{j,k})_{j,k=1}^N$, $\Dscr=(\Dscr_{j,k}^l)_{j,k, l=1}^N$ and 
$\Escr=(\Escr_{j,k}^l)_{j,k,l=1}^N$ are
\begin{equation*}
\Ascr_{j,k}=\left\langle\rho_f\tilde{h}\tilde B^{-\intercal}
\Psi_j, \tilde{B}^{-\intercal}\Psi_k\right\rangle+\left\langle\rho_s\psi_j, 
\psi_k\right\rangle,
\end{equation*}

\begin{equation*}
\begin{aligned}
\Bscr_{j,k}&=\left\langle\rho_f\prt_t\tilde{h}\tilde{B}^{-\intercal}
\Psi_j,\tilde{B}^{-\intercal}\Psi_k\right\rangle+2\left\langle\rho_f\tilde 
h\prt_t\tilde{B}^{-\intercal}\Psi_j, \tilde{B}^{-\intercal}\Psi_k
\right\rangle\\
&\quad+\left\langle\rho_f\tilde h\tilde{B}^{-\intercal}\Psi_j, \prt_t
\tilde{B}^{-\intercal}\Psi_k\right\rangle-\left\langle\rho_f\prt_t
\tilde{\chi}\tilde{B}\nabla(\tilde{B}^{-\intercal}\Psi_j), 
\tilde{B}^{-\intercal}\Psi_k\right\rangle\\
&\quad+\mu\left\langle \tilde{A}\nabla(\tilde{B}^{-\intercal}\Psi_j),
\nabla(\tilde{B}^{-\intercal}\Psi_k)\right\rangle
+\frac{1}{2}\left\langle\rho_f
\prt_t\tilde h\psi_j, \psi_k\right\rangle,
\end{aligned}
\end{equation*}

\begin{equation*}
\begin{aligned}
\Cscr_{j,k}&=\left\langle\rho_f\prt_t\tilde h\prt_t\tilde{B}^
{-\intercal}\Psi_j, \tilde{B}^{-\intercal}\Psi_k\right\rangle+\left\langle
\rho_f\tilde h\prt_t^2\tilde{B}^{-\intercal}\Psi_j, \tilde{B}^{-\intercal}
\Psi_k\right\rangle\\
&\quad+\left\langle\rho_f\tilde h\prt_t\tilde{B}^{-\intercal}\Psi_j,
\prt_t\tilde{B}^{-\intercal}\Psi_k\right\rangle-\left\langle\rho_f\prt_t^2
\tilde{\chi}\tilde{B}\nabla(\tilde{B}^{-\intercal}\Psi_j), 
\tilde{B}^{-\intercal}\Psi_k\right\rangle\\
&\quad-\left\langle\rho_f\prt_t\tilde{\chi}\prt_t \tilde{B}\nabla
(\tilde{B}^{-\intercal}\Psi_j),\tilde{B}^{-\intercal}\Psi_k\right
\rangle-\left\langle\rho_f\prt_t\tilde{\chi}\tilde{B}\nabla(\prt_t
\tilde{B}^{-\intercal}\Psi_j), \tilde{B}^{-\intercal}\Psi_k\right
\rangle\\
&\quad-\left\langle\rho_f\prt_t\tilde{\chi}\tilde{B}\nabla(\tilde B
^{-\intercal}\Psi_j),\prt_t\tilde{B}^{-\intercal}\Psi_k\right\rangle+
\left\langle\mu\prt_t\tilde{A}\nabla(\tilde{B}^{-\intercal}\Psi_j),
\nabla(\tilde{B}^{-\intercal}\Psi_k)\right\rangle\\
&\quad+\left\langle\mu \tilde{A}\nabla(\prt_t\tilde{B}^{-\intercal}\Psi_j),
\nabla(\tilde{B}^{-\intercal}\Psi_k)\right\rangle+\left\langle\mu \tilde{A}
\nabla(\tilde{B}^{-\intercal}\Psi_j), \nabla(\prt_t\tilde{B}
^{-\intercal}\Psi_k)\right\rangle\\
&\quad+\left\langle\beta\prt_x\psi_j, \prt_x\psi_k\right\rangle+\left\langle\alpha
\prt_x^2\psi_j, \prt_x^2\psi_k\right\rangle+\frac{1}{2}\left\langle\rho_f\prt_t^2
\tilde h\psi_j, \psi_k\right\rangle,
\end{aligned}
\end{equation*}

\begin{equation*}
\begin{aligned}
D_{j,k}^l&=\frac{1}{2}\left\langle\rho_f\tilde{B}^{-\intercal}\Psi_l
\tilde{B}\nabla(\tilde{B}^{-\intercal}\Psi_j), \tilde{B}
^{-\intercal}\Psi_k\right\rangle+\frac{1}{2}\left\langle\rho_f\tilde{B}^{-\intercal}\Psi_j\m \tilde{B}
\nabla(\tilde{B}^{-\intercal}\Psi_l), \tilde{B}^{-\intercal}\Psi_k
\right\rangle\\
&\quad-\frac{1}{2}\left\langle\rho_f\tilde{B}^{-\intercal}\Psi_l \tilde{B}
\nabla(\tilde{B}^{-\intercal}\Psi_j), \tilde{B}^{-\intercal}\Psi_k
\right\rangle-\frac{1}{2}\left\langle\rho_f\tilde{B}^{-\intercal}\Psi_j \tilde{B}
\nabla(\tilde{B}^{-\intercal}\Psi_k), \tilde{B}^{-\intercal}\Psi_l
\right\rangle,
\end{aligned}
\end{equation*}

\begin{equation*}
\begin{aligned}
\Escr_{j,k}^l&=\frac{1}{2}\left\langle\rho_f\prt_t\tilde{B}^{-\intercal}
\Psi_l \tilde{B}\nabla(\tilde{B}^{-\intercal}\Psi_j), \tilde{B}
^{-\intercal}\Psi_k\right\rangle+\frac{1}{2}\left\langle\rho_f\tilde{B}^{-\intercal}\Psi_j \prt_t \tilde{B}
\nabla(\tilde{B}^{-\intercal}\Psi_l), \tilde{B}^{-\intercal}\Psi_k
\right\rangle\\
&\quad+\frac{1}{2}\left\langle\rho_f\tilde{B}^{-\intercal}\Psi_j \tilde{B}
\nabla(\prt_t\tilde{B}^{-\intercal}\Psi_l), \tilde{B}^{-\intercal}
\Psi_k\right\rangle+\frac{1}{2}\left\langle\rho_f\tilde{B}^{-\intercal}\Psi_j \tilde{B}
\nabla(\tilde{B}^{-\intercal}\Psi_l), \prt_t\tilde{B}^{-\intercal}
\Psi_k\right\rangle\\
&\quad-\frac{1}{2}\left\langle\rho_f\prt_t\tilde{B}^{-\intercal}\Psi_l \tilde{B}
\nabla(\tilde{B}^{-\intercal}\Psi_j), \tilde{B}^{-\intercal}\Psi_k
\right\rangle-\frac{1}{2}\left\langle\rho_f\tilde{B}^{-\intercal}\Psi_j\prt_t \tilde{B}
\nabla(\tilde{B}^{-\intercal}\Psi_l), \tilde{B}^{-\intercal}\Psi_k
\right\rangle\\
&\quad-\frac{1}{2}\left\langle\rho_f\tilde{B}^{-\intercal}\Psi_j \tilde{B}
\nabla(\prt_t\tilde{B}^{-\intercal}\Psi_l), \tilde{B}^{-\intercal}
\Psi_k\right\rangle-\frac{1}{2}\left\langle\rho_f\tilde{B}^{-\intercal}\Psi_j \tilde{B}
\nabla(\tilde{B}^{-\intercal}\Psi_k), \prt_t\tilde{B}^{-\intercal}
\Psi_l\right\rangle.
\end{aligned}
\end{equation*}
We shall show in what follows that for given initial data ${\boldsymbol 
	\alpha_N}(0)$ and ${\boldsymbol \alpha_N}'(0)$, the 
differential equation \rfb{ODE} allows a unique solution ${\boldsymbol \alpha_N}$. 
The choice of the initial value ${\boldsymbol \alpha_N}'(0)$ will be discussed in Step~2 below.
We see that the matrix $\Ascr$ is obviously symmetric and for every 
$\xi\in\rline^N\setminus\{0\}$, we have
\begin{equation}\label{invertA}
\begin{aligned}
\Ascr\xi\cdot\xi&=\rho_f\tilde{h}\sum_{j,k=1}^N\xi_j\xi_k\left\langle
\tilde B^{-\intercal}\Psi_j, \tilde B^{-\intercal}\Psi_k
\right\rangle+\rho_s\sum_{j,k=1}^N\xi_j\xi_k
\left\langle\psi_j, \psi_k\right\rangle\\
&\geq C\left\langle \tilde B^{-\intercal}\sum_{j=1}^N\xi_j\Psi_j,
\tilde B^{-\intercal}\sum_{k=1}^N\xi_k\Psi_k\right\rangle+
c\sum_{k=1}^N\left\lVert\xi_k\psi_k\right\rVert^2\\
&> 0.
\end{aligned}
\end{equation}
Recalling that in the given geometry $\tilde h\in C^2([0, T]\times [0, L]; [h_{\min}, h_{\max}])$, all the coefficients in \rfb{ODE} are continuous in $t$ and 
all the non-linear quantites in \rfb{ODE} are locally Lipschitz continuous in 
${\boldsymbol \alpha_N}$ and ${\boldsymbol \alpha_N}'$. According to the 
Picard-Lindel\"of theorem, there is a unique solution of $\rfb{ODE}$ in short time. 
Therefore, we obtain a solution $(\hat u_{ N}, g_{ N})$ of 
\rfb{weaktimedisctete} which is given in form of \rfb{ansatz}.

{\bf Step 2: Choosing the initial values for the time-derivative.} For the initial data, we take
\begin{equation}\label{init}
\begin{aligned}
&\hat u_{ N}(0, x,z)=\Pscr_N\hat u_0,\quad g_{ N}(0,x)
=\Pscr_Nh_0, \quad (\prt_tg_{ N})(0,x)=\Pscr_Nh_1,
\end{aligned}
\end{equation}
where the projection $\Pscr_N$ has been defined in \rfb{projec}. 
From the definition of $\hat u_{ N}$ in 
\rfb{ansatz}, we note that actually ${\boldsymbol \alpha_N}(0)$ is determined 
by $\hat u_0$ and $h_0$, while ${\boldsymbol \alpha_N}'(0)$ is free. 
With the assumption \rfb{init} we choose the initial value 
${\boldsymbol \alpha_N}'(0)$ by considering the following equality:
\begin{equation}\label{initialdata}
\begin{aligned}
&\left\langle\rho_fh_0(\prt_t\hat u_{ N})(0),  B
^{-\intercal}_{h_0}\Psi_k\right\rangle+\frac{1}{2}\left\langle\rho_f\Pscr_N\hat 
u_0\left(B_{h_0}
\nabla\right)\Pscr_N\hat u_0, B^{-\intercal}_{h_0}\Psi_k\right\rangle\\
&-\frac{1}{2}\left\langle\rho_f\Pscr_N\hat u_0\left(B_{h_0}\nabla\right)(B
^{-\intercal}_{h_0}\Psi_k), \Pscr_N\hat u_0\right\rangle
+\frac{1}{2}\left\langle\rho_fh_1\Pscr_N h_1, \psi_k\right\rangle\\
&-\left\langle\rho_f\prt_t\chi_{h_0}\left(B_{h_0}\nabla\right)\Pscr_N\hat u_0,
B^{-\intercal}_{h_0}\Psi_k\right\rangle
+\mu\left\langle\left(A_{h_0}\nabla\right)\Pscr_N\hat u_0, \nabla\left(B^
{-\intercal}_{h_0}\Psi_k\right)\right\rangle\\
&+\rho_s\left\langle(\prt_t^2g_{ N})(0), \psi_k\right\rangle-\beta
\left\langle\prt_x^2\Pscr_Nh_0, \psi_k\right\rangle
+\alpha\left\langle\prt_x^4\Pscr_Nh_0, \psi_k\right\rangle
=0,
\end{aligned}
\end{equation}
where the projection operator $\Pscr_N$ has been introduced in \rfb{projec}.
Using the expression in \rfb{express} for $(\prt_t\hat u_{ N})(0)$ 
and $(\prt_t^2 g_{ N})(0)$, we immediately obtain from 
\rfb{initialdata} the equation for ${\boldsymbol \alpha_N}'(0)$:
\begin{equation}\label{detealpha}
M{\boldsymbol \alpha_N}'(0)=F({\boldsymbol \alpha_N}(0), \hat u_0, h_0, h_1),
\end{equation}
where the coefficient matrix $M=(M_{j,k})_{j,k=1}^N$ is given by
$$M_{j,k}=\rho_fh_0\left\langle B^{-\intercal}_{h_0}\Psi_j, B
^{-\intercal}_{h_0}\Psi_k\right\rangle+\rho_s\left\langle\psi_j, 
\psi_k\right\rangle.$$
By a similar analysis argument of $\Ascr$ in \rfb{invertA}, we know that 
the matrix $M$ is invertible. Therefore, the initial value 
${\boldsymbol \alpha_N}'(0)$ is uniquely determined by the equality \rfb{detealpha}.

Now we consider obtaining the initial value $(\prt_t\hat u_{ N})(0)$ 
and $(\prt_t^2 g_{ N})(0)$ from the assumption of the regularity 
of $\hat u_0$, $h_0$ and $h_1$. For this, we have the following proposition.

\begin{prop}\label{initprop}
	With the assumption in \rfb{init}, let the initial data $\hat u_0$, $h_0$ 
	and $h_1$ satisfy
	\begin{equation}\label{initialneed}
	\hat u_0\in H^2(\Om_1),\quad h_0\in H^4(0, L), \quad h_1\in H^2(0, L),
	\end{equation}
\end{prop}
then we have $(\prt_t\hat u_{ N})(0)\in L^2(\Om_1)$ and 
$(\prt_t^2g_{ N})(0)\in L^2(0, L)$.

\begin{proof}
	Still using the structure in \rfb{express} and recalling the projection 
	$\Pscr_N$ defined in \rfb{projec}, we first take an integration 
	by parts with respect to space for the third and the sixth terms in 
	\rfb{initialdata}:
	\begin{equation*}
	\begin{aligned}
	&-\left\langle\Pscr_N\hat u_0\left(B_{h_0}\nabla\right)(B
	^{-\intercal}_{h_0}\Psi_k), \Pscr_N\hat u_0\right\rangle\\
	&=-\int_{\Om_1}\div
	\left((B^{\intercal}_{h_0}\Pscr_N\hat u_0)\otimes 
	B^{-\intercal}_{h_0}\Psi_k\right)\cdot\Pscr_N\hat u_0\m\dd\bz\\
	&=-\int_0^L|\Pscr_N h_1|^2\psi_k\m\dd x+\int_{\Om_1}B^{-\intercal}_{h_0}\Psi_k\cdot \Pscr_N\hat 
	u_0(B_{h_0}\nabla)\Pscr_N\hat u_0\m\dd\bz,
	\end{aligned}
	\end{equation*}
	and
	\begin{equation*}
	\begin{aligned}
	&\left\langle\left(A_{h_0}\nabla\right)\hat u_0:\nabla\left(B^
	{-\intercal}_{h_0}\Psi_k\right)\right\rangle\\
	&=\int_0^L\left[A_{h_0}\nabla\hat u_0 B^{-\intercal}_{h_0}\Psi_k
	\right](z=1)\cdot {\rm e_2}\dd x-\int_{\Om_1}\div(A_{h_0}\nabla\hat u_0)
	\cdot B^{-\intercal}_{h_0}\Psi_k\m\dd\bz.
	\end{aligned}
	\end{equation*}
	
	Multiply the equality \rfb{initialdata} by $(\alpha_{N}^k)'(0)$ and sum 
	over $k=1, \cdots, N$. Note that we have
	\begin{align*}
	&\sum_{k=1}^N{\alpha_N^k}'(0)\psi_k=(\prt_t^2g_{ N})(0),
	\\
	&\sum_{k=1}^N B_{h_0}^{-\intercal}{\alpha_N^k}'(0)\Psi_k
	=
	B_{h_0}^{-\intercal}\prt_t\left(B_{h_0}^\intercal
	\hat u_{ N}(0)\right)=B_{h_0}^{-\intercal}\prt_tB_{h_0}
	^\intercal\hat u_{ N}(0)+(\prt_t\hat u_{ N})(0).
	\end{align*}
	Under the assumption \rfb{initialneed}, with \rfb{init} and \rfb{express}, we derive that
	\begin{equation*}
	\lVert(\prt_t\hat u_{ N})(0)\rVert_{L^2(\Om_1)}^2+\lVert(\prt_t^2
	g_{ N})(0)\rVert_{L^2(0, L)}^2\leq c(C_0)\left(\lVert\hat u_0\rVert_
	{H^2(\Om_1)}^2+ \lVert h_0\rVert_{H^4(0, L)}^2+\lVert h_1\rVert_{H^2(0,L)}^2\right),
	\end{equation*} 
	which ends the proof. In the above estimate, we used the properties of the projection discussed below \rfb{projec} for initial data.
\end{proof}


{\bf Step 3: Energy estimate for the discrete decoupled solution.}
Now we derive the energy estimate to extend the existence interval of 
${\boldsymbol \alpha_N}$. Taking integration of \rfb{weaktimedisctete} 
with respect to time variable and recalling the equation \rfb{initialdata}, 
we thereby obtain
\begin{equation}\label{weakform1st}
\begin{aligned}
&\int_{\Om_1}\rho_f\tilde{h}\prt_t\hat u_{ N}\cdot 
\tilde{B}^{-\intercal}\Psi_k\m\dd\bz+\frac{1}{2}\int_{\Om_1}\rho_f\hat
u_{ N}(\tilde{B}\nabla)\hat u_{ N}\cdot 
\tilde{B}^{-\intercal}\Psi_k\m\dd\bz\\
&-\frac{1}{2}\int_{\Om_1}\rho_f\hat u_{ N}(\tilde{B}\nabla)
\left(\tilde{B}^{-\intercal}\Psi_k\right)\cdot\hat u_{ N}\m\dd\bz
+\frac{1}{2}\int_0^L\rho_f\prt_t\tilde{h}\prt_tg_{ N}\psi_k\dd x\\
&-\int_{\Om_1}\rho_f\prt_t\tilde{\chi}(\tilde{B}\nabla)
\hat u_{ N}\cdot \tilde{B}^{-\intercal}\Psi_k\m\dd\bz
+\mu\int_{\Om_1}(\tilde{A}\nabla)\hat u_{ N}:\nabla
\left(\tilde{B}^{-\intercal}\Psi_k\right)\m\dd\bz\\
&+\rho_s\int_0^L\prt_t^2g_{ N}\psi_k\m\dd x-\beta\int_0^L
\prt_x^2g_{ N}\psi_k\m\dd x
+\alpha\int_0^L\prt_x^2g_{ N}\prt_x^2\psi_k\m\dd x
=0.
\end{aligned}
\end{equation}
\begin{prop}\label{firstenergy}
	For the solution to the ODE we have the following energy estimates:
	\begin{equation}\label{twoener1new}
	\begin{aligned}
	&h_{\min}\rho_f\lVert\hat u_{ N}\rVert^2_{L^\infty(0, T; L^2(\Om_1))}+
	\mu c(\tilde{h})\lVert\nabla\hat u_{ N}\rVert^2_{L^2(0, T; L^2(\Om_1))}\\
	&+
	\rho_s\lVert\prt_tg_{ N}\rVert^2_{L^\infty(0, T; L^2(0, L))}+\beta\lVert\prt_xg_{ N}\rVert^2_{L^\infty(0, T; L^2(0, L))}+
	\alpha\lVert\prt_x^2g_{ N}\rVert^2_{L^\infty(0, T; L^2(0, L))}
	\\
	&\leq C\left( \lVert\hat u_0\rVert_{L^2(\Om_1)}+\lVert h_1\rVert_{L^2(0,L)}^2
	+\lVert h_0\rVert_{H^2(0, L)}^2\right)=:\tilde{C}.
	\end{aligned}
	\end{equation}
	Please note  that the constant $\tilde{C}$ does not depend on $\tilde{h},h_{\min},h_{\max}$.
\end{prop}

\begin{proof}
	Multiply \rfb{weakform1st} by $\alpha_N^k(t)$ and sum over $k=1,\cdots, N$. 
	Note that
	$$\prt_t\tilde{h}=\div (\tilde{B}^\intercal\prt_t\tilde{\chi}), $$
	by using integration by parts we have
	\begin{equation*}
	\begin{aligned}
	\frac{1}{2}\int_{\Om_1}\rho_f\prt_t\tilde{h}|\hat u_{ N}|^2\m\dd\bz
	&=\frac{1}{2}\int_{\Om_1}\rho_f\div(\tilde{B}^\intercal\prt_t\tilde{\chi}) 
	|\hat u_{ N}|^2\m\dd\bz\\
	&=\frac{1}{2}\int_0^L\rho_f\prt_t\tilde{h}|\prt_t g_{ N}|^2\m\dd x
	-\int_{\Om_1}\rho_f\prt_t\tilde{\chi}(\tilde{B}\nabla)
	\hat u_{ N}\cdot\hat u_{ N}\m\dd\bz,
	\end{aligned}
	\end{equation*}
	which further implies that
	\begin{equation*}\label{recovertime}
	\int_{\Om_1}\rho_f \tilde{h}\hat u_{ N}\cdot \prt_t\hat
	u_{ N}\m\dd\bz+\frac{1}{2}\int_0^L\rho_f\prt_t\tilde{h}|\prt_t g_{ N}|^2\m\dd x
	-\int_{\Om_1}\rho_f\prt_t\tilde{\chi}(\tilde{B}\nabla)\hat
	u_{ N}\cdot\hat u_{ N}\m\dd\bz=\frac{1}{2}\frac{\dd}{\dd t}
	\int_{\Om_1}\rho_f\tilde{h}|\hat u_{ N}|^2\m\dd\bz.
	\end{equation*}
	Recalling the structure in \rfb{ansatz} and integrating the equation with 
	respect to time, we immediately obtain the energy balance \rfb{twoener1new}.
\end{proof}

{\bf Step 4: Performing the fixed-point theorem.}
We consider the following map for a fixed point. Let $h_{\min}, h_{\max}$ and $T$ be fixed by the initial conditions (See Remark~\ref{rem:time}). Similar with ${\boldsymbol \alpha_N}$ introduced in \eqref{ODE}, here we use the notation ${\boldsymbol \beta_N}:=\left(\beta_N^k\right)_{k=1}^N$. We take $K_N$ as the set of all vector-valued functions ${\boldsymbol \beta_N}$ satisfying the following conditions:
\begin{equation*}\label{KN}
K_N:=\left\{
\begin{aligned}
&{\boldsymbol \beta_N}\in C^1([0,T];\mathbb{R}^N)\left|\m h_\beta(t,x):=\int_0^t\sum_{k=1}^N\beta_N^k(\tau)\psi_k(x)\dd \tau+ \Pscr_N h_0 \right.\\
&\m\m\m\text{such that}\quad h_{\min}\leq h_\beta\leq h_{\max}, \quad \rho_s\lVert\prt_t h_\beta\rVert_{L^\infty(0, T; L^2(0, L))}^2\leq \tilde{C}.
\end{aligned} 
\right\},
\end{equation*} 
where the projection $\Pscr_N$ has been introduced in \rfb{projec}.
Taking such a function $h_\beta$ as the geometry, for $N$ given we consider the following map $F_N$:
\begin{align*}
F_N:K_N &\to C^2([0,T];\mathbb{R}^N)
\\
{\boldsymbol\beta_N}&\mapsto F_N({\boldsymbol \beta_N})={\boldsymbol \alpha_N},
\end{align*}
where ${\boldsymbol \alpha_N}$ is the unique solution to the ODE  \eqref{ODE}.
We shall use the Theorem of Schauder for fixed-point argument. For that we have to check the following points:
\begin{itemize}
	\item The set $K_N$ is convex and closed, as can be seen from its definition.
	\item The mapping is continuous and compact, as follows directly from classical ODE theory as the system is continuous and solutions are in $C^2[0,T]$ which is a compact subset of $C^1[0,T]$.
	\item The mapping is onto $K_N$, because of Remark~\ref{rem:time} and Proposition~\ref{firstenergy}. Indeed the energy estimate allows bounds on $\partial_tg_N$ and $\partial_x^2g_N$ independent of $\tilde{h}$. This implies in particular that for $T_0$ fixed by Remark~\ref{rem:time} and according choices of $h_{\min}$ and $h_{\max}$ that any function with bounded energy and with given initial values stays in the interval $[h_{\min},h_{\max}]$.
\end{itemize}
Hence the map $F_N$ possesses a fixed point in the set $K_N$. This means we find the existence of a coupled solution $(\hat u_{N}, g_N)$ to the system:

\begin{equation}\label{weaktimedisctete-coupled}
\begin{aligned}
&\left\langle \rho_f\prt_tg_N\prt_t\hat u_{N}, 
B_{g_{N}}^{-\intercal}\Psi_k\right\rangle
+\left\langle 
\rho_f{g_{N}}\prt_t^2\hat u_{N}, B^{-\intercal}_{g_{N}}
\Psi_k\right\rangle+\left\langle \rho_f{g_{N}}
\prt_t\hat u_{N}, \prt_tB^{-\intercal}_{g_{N}}
\Psi_k\right\rangle\\
&+\frac{1}{2}\left\langle \rho_f\prt_t\hat u_{N}\left(B_{g_{N}}
\nabla\right)\hat u_{N}, B^{-\intercal}_{g_{N}}
\Psi_k\right\rangle+\frac{1}{2}\left\langle \rho_f\hat u_{N}
\left(\prt_t B_{g_{N}}
\nabla\right)
\hat u_{N}, B_{g_{N}}^{-\intercal}\Psi_k\right\rangle\\
&+\frac{1}{2}\left\langle\rho_f\hat u_{N}\left(B_{g_{N}}
\nabla\right)\prt_t\hat u_{N}, B_{g_{N}}^{-\intercal}\Psi_k
\right\rangle+\frac{1}{2}\left\langle \rho_f\hat u_{N}
\left(B_{g_{N}}\nabla\right)\hat u_{N}, \prt_tB_{g_{N}}
^{-\intercal}\Psi_k\right\rangle\\
&-\frac{1}{2}\left\langle \rho_f\prt_t\hat u_{N}\left(B_{g_{N}}
\nabla\right)\left(B_{g_{N}}^{-\intercal}
\Psi_k\right), \hat u_{N}\right\rangle-\frac{1}{2}\left\langle
\rho_f\hat u_{N}\left(\prt_tB_{g_{N}}\nabla\right)\left(B_{g_{N}}
^{-\intercal}\Psi_k\right), \hat u_{N}\right\rangle\\
&-\frac{1}{2}\left\langle \rho_f\hat u_{N}\left(B_{g_{N}}
\nabla\right)\left(\prt_tB_{g_{N}}^{-\intercal}\Psi_k\right), 
\hat u_{N}\right\rangle-\frac{1}{2}\left\langle \rho_f\hat u_{N}\left(B_{g_{N}}\nabla\right)\left(B_{g_{N}}^{-\intercal}
\Psi_k\right), \prt_t\hat u_{N}\right\rangle\\
&+\left\langle \rho_f\prt_t^2g_N\prt_tg_{N}, 
\psi_k\right\rangle
-\left\langle\rho_f\prt_t^2\chi_{g_{N}}(B_{g_{N}}\nabla)\hat 
u_{N}, B_{g_{N}}^{-\intercal}\Psi_k\right\rangle-\left\langle
\rho_f\prt_t\chi_{g_{N}}(\prt_tB_{g_{N}}\nabla)\hat u_{N},
B_{g_{N}}^{-\intercal}\Psi_k\right\rangle\\
&-\left\langle \rho_f\prt_t\chi_{g_{N}}(B_{g_{N}}\nabla)\prt_t
\hat u_{N}, B_{g_{N}}^{-\intercal}\Psi_k\right\rangle-\left
\langle \rho_f\prt_t\chi_{g_{N}}(B_{g_{N}}\nabla)\hat u_{N},
\prt_tB_{g_{N}}^{-\intercal}\Psi_k\right\rangle\\
&+\left\langle \mu (\prt_t{A_{g_{N}}}\nabla)\hat u_{N},
\nabla(B_{g_{N}}^{-\intercal}\Psi_k)\right\rangle+\left\langle
\mu(A_{g_{N}}\nabla)\prt_t\hat u_{N}, \nabla(B_{g_{N}}
^{-\intercal}\Psi_k)\right\rangle\\
&+\left\langle \mu(A_{g_{N}}\nabla)
\hat u_{N}, \nabla(\prt_tB_{g_{N}}^{-\intercal}\Psi_k)
\right\rangle+\left\langle\rho_s\prt_t^3g_{N}, \psi_k\right\rangle-\left
\langle\beta
\prt_t\prt_x^2g_{N}, \psi_k\right\rangle\\
&+\left\langle\alpha
\prt_t\prt_x^2g_{N}, \prt_x^2\psi_k\right\rangle
=0,
\end{aligned}
\end{equation}
where $\hat u_N(t,x,1)=\prt_tg_N{\rm e_2}$ for all $(t,x)\in (0, T)\times (0, L)$.

\subsection{The higher order in time-estimate}
It can be seen from the definition of the ODE, that multiplying \eqref{weaktimedisctete-coupled} with $\alpha_N'$ is precisely testing the time-derivative equation with the coupled test-function $(\partial_t^2g_N,\partial_t\hat{u}_N+G_N)$ with $G_N$ similarly defined as in \rfb{G}, i.e. $G_N=B_{g_{N}}^{-\intercal}\prt_tB_{g_{N}}
^\intercal\hat u_{ N}$. Hence the formal estimate of Section~\ref{sectionproof} can be performed rigorously on the discrete level. Therefore, we have the following theorem. Please observe, that together with Proposition~\ref{initprop} this implies uniform higher order estimates.

\begin{thm}
	\label{thm:main-disrete}
	Under the conditions \eqref{initialneed} there exists a discrete solution to \eqref{weaktimedisctete-coupled}, that satisfies the first-order equation~\eqref{weakform1st} (for the coupled equation with $\tilde h$ replaced by $g_N$), the energy inequality \eqref{twoener1new} and the additional time-derivative estimate:
	\begin{equation}\label{twoener3'}
	\begin{aligned}
	&\rho_f\lVert\prt_t\hat u_{ N}\rVert_{L^\infty(0, T; L^2(\Om_1))}^2+\mu
	\lVert\nabla\prt_t\hat u_{ N}\rVert^2_{L^2(0, T; L^2(\Om_1))}+\rho_s
	\lVert\prt_t^2g_{ N}\rVert_{L^\infty(0, T; L^2(0, L))}\\
	&+\beta\lVert\prt_t\prt_xg_{N}\rVert_{L^\infty(0, T; L^2(0, L))}+\alpha
	\lVert\prt_t\prt_x^2g_{ N}\rVert^2_{L^\infty(0, T; L^2(0, L))}
	\\
	&\leq C(h_{\min},\lVert\hat u_0
	\rVert_{L^2(\Om_1)},\lVert 
	h_0\rVert_{H^2(0, L)}),\lVert 
	h_1\rVert_{L^2(0, L)})\Big(\lVert(\prt_t\hat u)(0)\rVert_{L^2(\Om_1)}^2+ \lVert(\prt_t^2 h)(0)\rVert_{L^2(0, L)}^2+\lVert h_1\rVert_{H^2(0, L)}^2\Big),
	\end{aligned}
	\end{equation}
	which holds in particular independent of $N$.
\end{thm} 
\begin{proof}
	The existence follows by the fixed point performed in the last subsection. 
	We are left to show the time-derivative estimate of $\hat u_{ N}$, i.e. \rfb{twoener3'}. 
	To do this, we first derive the following identities:
	\begin{equation*}\label{compute}
	\begin{aligned}
	&\sum_{k=1}^N{\alpha_N^k}'(t)\psi_k=\prt_t^2g_{ N},
	\\
	&\sum_{k=1}^N B_{g_{N}}^{-\intercal}{\alpha_N^k}'(t)\Psi_k
	=
	B_{g_{N}}^{-\intercal}\prt_t\left(B_{g_{N}}^\intercal
	\hat u_{ N}\right)=B_{g_{N}}^{-\intercal}\prt_tB_{g_{N}}
	^\intercal\hat u_{ N}+\prt_t\hat u_{ N},\\
	&\sum_{k=1}^N\prt_tB_{g_{N}}^{-\intercal}{\alpha_N^k}'(t)\Psi_k
	=\prt_tB_{g_{N}}^{-\intercal}\prt_t\left(B_{g_{N}}^{\intercal}
	\hat u_{ N}\right)
	=\prt_tB_{g_{N}}^{-\intercal}\prt_t
	B_{g_{N}}^\intercal\hat u_{ N}+\prt_tB_{g_{N}}^
	{-\intercal}B_{g_{N}}^\intercal\prt_t\hat u_{ N}.
	\end{aligned}
	\end{equation*}
	By using the fact that $\nabla(Av)=(\nabla A)v+
	A(\nabla v)$, where $A$ is a $2\times 2$ matrix and $v$ 
	is a $2\times 1$ matrix, we also derive that
	\begin{equation*}
	\begin{aligned}
	&\sum_{k=1}^N{\alpha_N^k}'(t)\nabla\left(B_{g_{N}}^{-\intercal}
	\Psi_k\right)
	=\nabla\left(B_{g_{N}}^{-\intercal}\prt_t\left(B_{g_{N}}
	^{\intercal}\hat u_{ N}\right)\right)\\
	&=\nabla B_{g_{N}}^{-\intercal}\prt_t B_{g_{N}}^\intercal
	\hat u_{ N}+\nabla B_{g_{N}}^{-\intercal}B_{g_{N}}
	^{\intercal}\prt_t\hat u_{ N}
	+B_{g_{N}}^{-\intercal}\nabla\prt_tB_{g_{N}}^{\intercal}\hat u_{ N}
	\\
	&\quad +B_{g_{N}}^{-\intercal}\prt_tB_{g_{N}}^\intercal\nabla\hat u_{ N}
	+B_{g_{N}}
	^{-\intercal}\nabla B_{g_{N}}^\intercal\prt_t\hat u_{ N}+\nabla\prt_t\hat u_{ N}
	\\
	\text{ and }&\sum_{k=1}^N{\alpha_N^k}'(t)\nabla\left(\prt_tB_{g_{N}}
	^{-\intercal}\Psi_k\right)=\nabla\left(\prt_tB_{g_{N}}^{-\intercal}\prt_t\left(B_{g_{N}}^
	\intercal\hat u_{ N}\right)\right)\\
	&=\nabla\prt_tB_{g_{N}}^{-\intercal}\prt_tB_{g_{N}}^\intercal
	\hat u_{ N}+\nabla\prt_tB_{g_{N}}^{-\intercal}B_{g_{N}}
	^\intercal\prt_t\hat u_{ N}+\prt_tB_{g_{N}}^{-\intercal}\nabla\prt_tB_{g_{N}}^\intercal
	\hat u_{ N}\\
	&\quad +\prt_tB_{g_{N}}^{-\intercal}\prt_tB_{g_{N}}^\intercal\nabla\hat u_{ N}
	+\prt_tB_{g_{N}}
	^{-\intercal}\nabla B_{g_{N}}^\intercal\prt_t\hat u_{ N}+\prt_tB_{g_{N}}^{-\intercal}B_{g_{N}}^\intercal\nabla\prt_t\hat u_{ N}.
	\end{aligned}
	\end{equation*}
	Now multiply \rfb{weaktimedisctete-coupled} by ${\alpha_N^k}'(t)$, 
	sum over $k=1,\cdots, N$ and substitute the above sum expressions in the resulting 
	equation, which allows to reconstruct \eqref{sum}. 
	%
	First taking the integration with time on $(0, t)$ for $t\in (0, T)$, we note that the terms kept on the left side are
	\begin{equation*}
	\begin{aligned}
	&\rho_f\lVert\prt_t\hat u_{ N}\rVert_{L^\infty(0, T; L^2(\Om_1))}^2+\mu
	\lVert\nabla\prt_t\hat u_{ N}\rVert^2_{L^2(0, T; L^2(\Om_1))}+\rho_s
	\lVert\prt_t^2g_{ N}\rVert_{L^\infty(0, T; L^2(0, L))}\\
	&+\beta\lVert\prt_t\prt_xg_{N}\rVert_{L^\infty(0, T; L^2(0, L))}+\alpha
	\lVert\prt_t\prt_x^2g_{ N}\rVert^2_{L^\infty(0, T; L^2(\Om_1))}
	.
	\end{aligned}
	\end{equation*} 
	which precisely corresponds to the left hand side of \eqref{sum}.
	We show in what follows that, multiplying \rfb{weaktimedisctete-coupled} by ${\alpha_N^k}'(t)$, the resulting discrete equation also has the same right hand side as \eqref{sum} in the continuous level. For that we first note from the above series structure that $\sum_{k=1}^N B_{g_{N}}^{-\intercal}{\alpha_N^k}'(t)\Psi_k$ represents $\prt_t\hat u_N+G_N$. With the definition of $\varphi$ in \rfb{findvarphi}, we derive from the formula of $G$ in \rfb{G} that
	\begin{equation*}
	\varphi=
	B_h^{-\intercal}\prt_t B_h^\intercal(\prt_t\hat u+B_h^{-\intercal}\prt_tB_h^\intercal\hat u)=-\prt_t B_h^{-\intercal}\prt_t B_h^{\intercal}\hat u
	-\prt_t B_h^{-\intercal}B_h^{\intercal}\prt_t\hat u,
	\end{equation*}
	which corresponds to $\sum_{k=1}^N\prt_tB_{g_{N}}^{-\intercal}{\alpha_N^k}'(t)\Psi_k$ in the discrete level. We used in the above the fact that $B_h^{-\intercal}\prt_t B_h^\intercal=-\prt_t B_h^{-\intercal}B_h^{\intercal}$. Hence, the series $\sum_{k=1}^N{\alpha_N^k}'(t)\nabla\left(\prt_tB_{g_{N}}
	^{-\intercal}\Psi_k\right)$ thereby corresponds to $\nabla\varphi$.
	Form here, we realized that the terms depending on $\varphi$ on the right hand side in \eqref{sum} appear here due to the time derivative of $\partial_t B_{g_{N}}^{-\intercal}$, which is precisely related to the pressure.
	
	Comparing the right hand sides of the continuous equation \rfb{sum} with the discrete equation, we realize that, after doing the integration by parts for the solid part, only the convective terms need to be clarified. For that we rewrite the conventive term as it appears in the continuous equation in \rfb{sum}, such that it connects to \rfb{weaktimedisctete-coupled} multiplied by ${\alpha_N^k}'(t)$. Using the structure \rfb{convective} we derive that
	\begin{equation*}
	\begin{aligned}
	-\frac{1}{2}\int_0^t\int_{\Om_1}\rho_f\prt_t\hat u(B_h\nabla)\hat u\cdot (\prt_t\hat u+G)\m\dd\bz=\frac{1}{2}\int_0^t\int_{\Om_1}\rho_f\prt_t\hat u(B_h\nabla)(\prt_t\hat u+G)\cdot \hat u\m\dd\bz-\frac{1}{2}\int_0^t\rho_f|\prt_t^2h|^2\prt_th\m\dd x,
	\end{aligned}
	\end{equation*}
	and
	\begin{equation*}
	-\frac{1}{2}\int_0^t\int_{\Om_1}\rho_f\hat u(B_h\nabla)\prt_t\hat u\cdot (\prt_t\hat u+G)\m\dd\bz=\frac{1}{2}\int_0^t\int_{\Om_1}\rho_f\hat u(B_h\nabla)(\prt_t\hat u+G)\cdot \prt_t\hat u\m\dd\bz-\frac{1}{2}\int_0^t\rho_f|\prt_t^2h|^2\prt_th\m\dd x,
	\end{equation*}
	\begin{equation*}
	-\frac{1}{2}\int_0^t\int_{\Om_1}\rho_f\hat u(\prt_tB_h\nabla)\hat u\cdot (\prt_t\hat u+G)\m\dd\bz=\frac{1}{2}\int_0^t\int_{\Om_1}\rho_f\hat u(\prt_tB_h\nabla)(\prt_t\hat u+G)\cdot \hat u\m\dd\bz,
	\end{equation*}
	where we observe that no boundary term in the last equality since $\prt_t B_h^\intercal\hat u\otimes \hat u(\prt_t\hat u+G)$ vanishes at $z=1$. 
	
	Based on the above analysis, we obtain that the resulting equation by multipying \rfb{weaktimedisctete-coupled} by ${\alpha_N^k}'(t)$ is exactly corresponds to the structure \rfb{sum} in the continuous sense. Moreover, $(\hat u_N, g_N)$ preserves the energy estimate \rfb{twoener1new}. Therefore, the proof follows line by line using the estimates of \eqref{sum}. This finishes the proof.
\end{proof}

\subsection{Existence of a strong solution}
By Theorem~\ref{thm:main-disrete} together with Proposition~\ref{initprop} we
got the time-derivative estimate on the discrete level. This allows us to reconstruct a weak solution, with additional regularity properties. Together this implies the existence of a strong solution.
With the help of \rfb{twoener1new} and \rfb{twoener3'}, we are able to deal with the limit procedure.

{\bf Limit passage for $\hat{u}$ and $h$.} 

Based on the uniform estimates of Proposotion \ref{firstenergy} and Theorem~\ref{thm:main-disrete}, 
sending $N\to\infty$, we obtain the following 
convergence, up to a subsequence, for every $T>0$, 
\begin{equation}\label{converg}
\begin{aligned}
\hat u_{ N} &\rightharpoonup \hat u \quad \text{weak-$\ast$ 
	in} \quad L^\infty(0, T; L^2(\Om_1))\cap L^2(0, T; H^1(\Om_1)),\\
\prt_t\hat u_{ N} &\rightharpoonup \prt_t\hat u
\quad \text{weak-$\ast$ 
	in} \quad L^\infty(0, T; L^2(\Om_1))\cap L^2(0, T; H^1(\Om_1)),\\
g_{ N}&\rightharpoonup h \quad \text{weak-$\ast$ in} \quad
L^\infty(0, T; H^2(0, L)),\\
\prt_tg_{ N} &\rightharpoonup \prt_t h \quad \text{weak-$\ast$ 
	in} \quad  L^\infty(0, T; H^2(0, L)),\\
\prt_t^2g_{ N} &\rightharpoonup \prt_t^2 h \quad \text{weak-$
	\ast$ in} \quad  L^\infty(0, T; L^2(0, L)).
\end{aligned}
\end{equation}
This implies by compactness that up to a subsequence, we have using the analysis \eqref{eq:hoeldercont} that
\begin{equation*}\label{convergstr}
\begin{aligned}
\hat u_{ N} &\to \hat u \quad\text{ in }\quad L^2((0,T)\times \Om_1),
\\
\prt_tg_{ N} &\to \prt_th \quad\text{ in }\quad C^{0,\alpha}([0, T]\times [0,L]),
\\
g_{ N} &\to h \quad \text{ in }\quad C^{1}([0, T]\times [0,L]),
\end{aligned}
\end{equation*}
for $0<\alpha<\frac{2}{3}$.
This allows us to pass to the limit with \eqref{weakform1st} (the version by replacing $\tilde h$ by $g_N$ due to the fixed point procedure)
and we thereby obtain
\begin{equation*}\label{eq:limiteq}
\begin{aligned}
&\int_{\Om_1}\rho_fh\prt_t\hat u \cdot 
B_h^{-\intercal}\Psi\m\dd\bz+\int_{\Om_1}\rho_f\hat
u (B_h \nabla)\hat u\cdot 
B_h^{-\intercal}\Psi\m\dd\bz
\\
&-\int_{\Om_1}\rho_f\prt_t\chi_h(B_h\nabla)
\hat u \cdot B_h^{-\intercal}\Psi\m\dd\bz
+\mu\int_{\Om_1}(A_h\nabla)\hat u:\nabla
\left(B_h^{-\intercal}\Psi\right)\m\dd\bz\\
&+\rho_s\int_0^L\prt_t^2h\psi\m\dd x-\beta\int_0^L
\prt_x^2h \psi\m\dd x
+\alpha\int_0^L\prt_x^2h\prt_x^2\psi\m\dd x
=0,
\end{aligned}
\end{equation*}
for all sufficiently smooth $\psi(t,x)=\Psi(t,x,1)$ with $\div \Psi=0$.

{\bf Regularity of $u$, $p$ and $h$.}

Here we rely on the properties of the steady Stokes equation. Let us consider the steady Stokes system
\begin{equation}\label{eq:Stokes}
\left\{\begin{aligned}
&\Delta u-\nabla\pi=f,	\\
&\div u=0,\\
&u|_{\partial{\Om_h}}=u_{\partial},
\end{aligned}\right.
\end{equation}
in a domain ${\Om_h}\subset\rline^2$ with unit normal ${\bf n}$. The result given in the following theorem is a maximal regularity estimate for the solution of \eqref{eq:Stokes} in terms of the right-hand side. We quote here the Stokes estimate from \cite[Theorem 3.1]{breit2022regularity} and \cite[Theorem 2.8]{BreMenSchSu23} for the system \rfb{eq:Stokes}. 
\begin{thm}\label{thm:stokessteady}
	Let $q\in(1,\infty)$, $s\geq 1+\frac{1}{q}$ and 
	\begin{align*}
	\varrho\geq q\quad\text{if}\quad q(s-1)\geq 3;\quad \varrho\geq \tfrac{2q}{q(s-1)-1}\quad\text{if}\quad q(s-1)< 3,
	\end{align*}
	such that 
	$3\big(\frac{1}{q}-\frac{1}{2}\big)+1\leq  s$.
	Suppose that ${\Om_h}$ is a $B^{\theta}_{\varrho,q}$-domain\footnote{For bounded domain $\mathcal{O}$, the Besov spaces are defined as $B^{\theta}_{\varrho,q}(\mathcal{O}):=\{f|_{\mathcal{O}}: f\in B^{\theta}_{\varrho,q}(\rline^n )\}$ with the norm $\lVert g\rVert_{B^{\theta}_{\varrho,q}(\mathcal{O})}=\inf\{\lVert f\rVert_{B^{\theta}_{\varrho,q}(\rline^n )}: f|_{\mathcal{O}}=g\}$.} and $h\in L^\infty(\bar I; B^{\theta}_{\varrho,q}\cap C^1[0, L])$ for some $\theta>s-1/q$ with locally small Lipschitz constant, $f\in W^{s-2,q}({\Om_h})$ and $u_{\partial}\in W^{s-1/q,q}(\partial{\Om_h})$ with $\int_{\partial{\Om_h}}u_\partial\cdot {\bf n}\dd S=0$. Then there is a unique solution to \eqref{eq:Stokes} with $\int_{\Omega_h} \pi \dx=0$ that satisfies
	\begin{align}\label{stokesmain}
	\|u\|_{W^{s,q}({\Om_h})}+\|\pi\|_{W^{s-1,q}({\Om_h})}\lesssim\|f\|_{W^{s-2,q}({\Om_h})}+\|u_{\partial}\|_{W^{s-1/q,q}(\partial{\Om_h})}.
	\end{align}
\end{thm}

\begin{rmk}
	{\rm We remark here that the above result in \cite{breit2022regularity} and \cite{BreMenSchSu23} is stated in a general $B^{\theta}_{\varrho,q}$-domain $\mathcal{O}$. It is also available for the moving domain $\Om_h$ with $h\in L^\infty(\bar I; B^{\theta}_{\varrho,q}\cap C^1[0, L])$, and it suffices to verify that $\prt\Om_h\in B^{\theta}_{\varrho,q}$. This has been explained in \cite[Remark 3.4]{breit2022regularity} and \cite[Remark 2.9]{BreMenSchSu23} in a detailed way.
	}
\end{rmk}

Based on the weakly lower semi-continuity of the weak convergence in \rfb{converg}, we note that $(\hat u, h)$ also satisfy the energy balance \rfb{twoener1new} and the time-derivative estimate \rfb{twoener3'}. 
We will explain in the following that this, together with Lemma \ref{thm:stokessteady}, implies more regularity for $\hat u$. Moreover, with the help of Proposition \ref{relareg}, we could derive more regularity directly for $u$. 

\begin{thm}\label{themfinal}
	Assume that the initial data $(u_0, h_0, h_1)$ satisfy $h_0\in H^4(0, L)$, $h_1\in H^2(0, L)$, $u_0\in H^2(\Om_{h_0})$ and $\min_{x\in(0, L)}h_0=\delta>0$, then the system \eqref{fluideq}--\eqref{initiald} admits a unique solution $( u, p, h)$, for $T>0$, satisfying 
	\begin{equation}\label{hatfinal}
	\begin{aligned}
	&h\in W^{1,\infty}(0, T; H^2(0, L))\cap W^{2, \infty}(0, T; L^2(0, L))\cap L^\infty(0, T; H^4(0, L)),\\
	&  u\in L^\infty(0, T; H^2(\Om_h))\cap L^4(0, T; H^{\frac{5}{2}}(\Om_h)),\\
	& p\in L^\infty(0, T; H^1(\Om_h))\cap L^4(0, T; H^{\frac{3}{2}}(\Om_h)).
	\end{aligned}
	\end{equation}
	This solution exists until the self-intersection of $\Om_h$.
\end{thm}
\begin{proof}
	The regularity of $h$ can be obtained directly from the estimate \rfb{twoener1new} and \rfb{twoener3'} that
	$$h\in W^{1,\infty}(0, T; H^2(0, L))\cap W^{2, \infty}(0, T; L^2(0, L)). $$ 
	For the initial data, recalling that $\hat u_0(x,z)=u_0(x, h_0z)=u_0(x,y)$, we notice from Proposition \ref{relareg} that $\lVert\hat u_0\rVert_{H^2(\Omega_{1})}$ is equivalent to $\lVert u_0\rVert_{H^2(\Om_{h_0})}$. This implies that the initial conditions \rfb{initialneed} is satisfied. The pressure is constructed in two parts. For that we decompose it into
	\begin{align}
	\label{eq:presdec}
	p(t,x)=p_0(t,x)+p_1(t)\text{ such that }\int_{\Omega_{h(t)}}p_0(t,x)\, \dx=0.
	\end{align}
	Now $p_0$ is directly constructed via Stokes equation, while $p_1$ is constant in space.
	
	To present the proof clearly, we divide it in the following steps.
	
	{\bf Step 1}: {\em We show that $u\in L^\infty(0, T; W^{2,r}(\Om_h))$ for every $1<r<2$.} Note that we have the regularity for $\hat u$: $$\hat u\in W^{1,\infty}(0, T; L^2(\Om_1))\cap H^1(0, T; H^1(\Om_1))$$ from the weak convergence \rfb{converg}. 
	This gives us that $ \nabla\hat u\in L^\infty(0, T; L^2(\Om_1))$. From Proposition \ref{relareg} we obtain that $\nabla u\in L^\infty(0, T; L^2(\Om_h))$. Using H\"older's inequality and Sobolev embedding $H^1(\Om_h)\hookrightarrow L^{\frac{2r}{2-r}}(\Om_h)$, for every $1<r<2$ we have 
	$$\lVert u\cdot\nabla u\rVert_{L^r(\Om_h)}\lesssim \lVert u\rVert_{L^{\frac{2r}{2-r}}(\Om_h)}\lVert\nabla u\rVert_{L^2(\Om_h)}\lesssim \lVert\nabla u\rVert_{L^2(\Om_h)}^2. $$
	Hence we obtain that $u\cdot\nabla u\in L^\infty(0, T; L^r(\Om_h))$ for every $1<r<2$.
	
	Now we apply the Stokes estimate \rfb{stokesmain} for $s=2$, $q=r\in (1,2)$, $f=\prt_t u+u\cdot \nabla u$ and $u_{\prt}=\prt_t h$. Thereby we have 
	\begin{equation}\label{restimate}
	\begin{aligned}
	\lVert u\rVert_{W^{2, r}(\Om_h)}+\lVert\m p_0\rVert_{W^{1,r}(\Om_h)}&\lesssim \lVert\prt_t u+u\cdot \nabla u\rVert_{L^r(\Om_h)}+\lVert\prt_t h\rVert_{W^{2-\frac{1}{r}, r}(0, L)}\\
	&\lesssim \lVert\prt_t u\rVert_{L^r(\Om_h)}+\lVert u\cdot \nabla u\rVert_{L^r(\Om_h)}+\lVert\prt_t h\rVert_{W^{2-\frac{1}{r}, r}(0, L)}.
	\end{aligned}
	\end{equation}
	Note that $\prt_t\hat u\in L^\infty(0, T; L^2(\Om_1))$, then through Proposition \ref{relareg} we get that $\prt_t u\in L^\infty(0, T; L^2(\Om_h))$. Moreover, recall that $\prt_t h\in L^\infty(0, T; H^2(0, L))$, using the continuous embedding $H^2(0, L)\hookrightarrow W^{2-\frac{1}{r}, r}(0, L)$,  we derive from \rfb{restimate} that $u\in L^\infty(0, T; W^{2,r}(\Om_h))$ for $1<r<2$. 
	
	{\bf Step 2}: {\em We prove qualitatively that $u\in L^\infty(0, T; H^2(\Om_h))$ and $p\in L^\infty(0, T; H^1(\Om_h))$.} Using the regularity $u\in L^\infty(0, T; W^{2,r}(\Om_h))$ and the continuous embedding $W^{2,r}(\Om_h)\hookrightarrow L^\infty(\Om_h)$ for $1<r<2$, we have $u\in L^\infty((0, T)\times \Om_h)$. Moreover, we note that 
	$$\lVert u\cdot \nabla u\rVert_{L^2(\Om_h)}\lesssim \lVert u\rVert_{L^\infty(\Om_h)}\lVert\nabla u\rVert_{L^2(\Om_h)}. $$
	This, together with $ \nabla u\in L^\infty(0, T; L^2(\Om_h))$, gives us that $u\cdot \nabla u\in L^\infty(0, T; L^2(\Om_h))$.
	At this stage, we apply the Stokes estimate \rfb{stokesmain} again for $s=2=q$ and obtain 
	\begin{equation}\label{H2estimate}
	\begin{aligned}
	\lVert u\rVert_{H^2(\Om_h)}+\lVert\m p_0\rVert_{H^1(\Om_h)}&\lesssim \lVert\prt_t u+u\cdot \nabla u\rVert_{L^2(\Om_h)}+\lVert\prt_t h\rVert_{H^{\frac{3}{2}}(0, L)}\\
	&\lesssim \lVert\prt_t u\rVert_{L^2(\Om_h)}+\lVert u\cdot \nabla u\rVert_{L^2(\Om_h)}+\lVert\prt_t h\rVert_{H^{\frac{3}{2}}(0, L)}.
	\end{aligned}
	\end{equation}
	Since we already notice that $\prt_t u\in L^\infty(0, T; L^2(\Om_h))$, then \rfb{H2estimate} shows that $u\in L^\infty(0, T; H^2(\Om_h))$ and $p_0\in L^\infty(0, T; H^1(\Om_h))$.
	
	Based on the decomposition \rfb{eq:presdec}, now we define $p_1$ in such a way that \eqref{beameq} is satisfied. 
	Recall that we have the zero mean condition \rfb{zeromeansource} for the source term $\phi(u,p,h)$ defined in \rfb{dynamiceq}.
	This, together with \rfb{eq:presdec}, gives us that $\sigma(u, p)=\sigma(u, p_0)-p_1 \mathbb{I}_{2\times 2}$ and thereby we have
	\[
	Lp_1(t)=\int_0^L{\rm e_2}\cdot \sigma(u, p_0)(t,x, h(t,x)) (-\prt_x
	h\m{\rm e_1}+{\rm e_2})\dx\in L^\infty(0,T).
	\]
	Hence, with the current regularity of $u$ and $p$, we derive that the Cauchy stress tensor $\sigma(u, p)\in L^\infty(0, T; H^1(\Om_h))$. Thereby the source term of the structure $\phi(u, p, h)$ belongs to $L^\infty(0, T; H^{\frac{1}{2}}(0, L))$. Based on the beam equation \rfb{beameq} and the regularity of $h$, we finally have $\prt_x^4 h\in L^\infty(0, T; L^2(0, L))$. Therefore, the strong solution $(u, p, h)$ of the system \eqref{fluideq}--\eqref{initiald} is established in both space and time.

	{\bf Step 3}: {\em We have further regularity: $u\in L^4(0, T; H^{\frac{5}{2}}(\Om_h))$ and $p\in L^4(0, T; H^{\frac{3}{2}}(\Om_h))$.} Recalling the regularity of $\hat u$, we have $\prt_t\hat u\in L^\infty(0, T; L^2(\Om_1))\cap L^2(0, T; H^1(\Om_1))$. With the result in Step 2, we know that $u\in L^\infty(0, T; H^2(\Om_h))$ and then by Proposition \ref{relareg} we get $\prt_t u\in L^\infty(0, T; L^2(\Om_h))\cap L^2(0, T; H^1(\Om_h)) $ as well. By the interpolation inequality:
	\begin{equation}\label{inequse}
	\lVert\prt_t u\rVert_{L^4(0, T; H^{\frac{1}{2}}(\Om_h))}\lesssim \lVert\prt_t u\rVert_{L^\infty(0, T; L^2(\Om_h))}^{\frac{1}{2}}\lVert\prt_t u\rVert_{L^2(0, T; H^1(\Om_h))}^{\frac{1}{2}},	
	\end{equation} 
	we obtain further that $\prt_t u\in L^4(0, T; H^{\frac{1}{2}}(\Om_h))$. Now we use the Stokes estimate \rfb{stokesmain} for $s=\frac{5}{2}$ and $q=2$ once more and have
	\begin{equation}\label{higher}
	\begin{aligned}
	\lVert u\rVert_{H^{\frac{5}{2}}(\Om_h)}+\lVert\m p_0\rVert_{H^{\frac{3}{2}}(\Om_h)}&\lesssim \lVert\prt_t u+u\cdot \nabla u\rVert_{H^{\frac{1}{2}}(\Om_h)}+\lVert\prt_t h\rVert_{H^{2}(0, L)}\\
	&\lesssim \lVert\prt_t u\rVert_{H^{\frac{1}{2}}(\Om_h)}+\lVert u\cdot \nabla u\rVert_{H^{\frac{1}{2}}(\Om_h)}+\lVert\prt_t h\rVert_{H^{2}(0, L)}\\
	&\lesssim \lVert\prt_t u\rVert_{H^{\frac{1}{2}}(\Om_h)}+\lVert u\rVert_{L^\infty(\Om_h)}\lVert \nabla u\rVert_{H^{\frac{1}{2}}(\Om_h)}+\lVert\prt_t h\rVert_{H^{2}(0, L)}.
	\end{aligned}
	\end{equation}
	Combining with the above analysis, we obtain from \rfb{higher} that
	\begin{align*}
	&\lVert u\rVert_{L^4(0, T;H^{\frac{5}{2}}(\Om_h))}+\lVert\m p_0\rVert_{L^4(0, T; H^{\frac{3}{2}}(\Om_h))}\\&\lesssim \lVert\prt_t u\rVert_{L^4(0, T;H^{\frac{1}{2}}(\Om_h))}+ C(T)\left(\lVert  u\rVert_{L^\infty(0, T; H^2(\Om_h))}+\lVert \prt_th\rVert_{L^\infty(0, T; H^2(0, L))}\right).
	\end{align*} 
	Using the decomposition argument again, we obtain the regularity of $p$.
	
	Putting the above results together, we conclude that the regularity in \rfb{hatfinal} holds, which ends the proof.
\end{proof}

\subsection{Proof of Theorem~\ref{thm:main}}
Now we are in a position to finish the proof of Theorem \ref{thm:main}. The existence of a strong solution was obtained in Theorem \ref{themfinal}. Hence we are left to show {\em quantitatively} that the estimate is linear with respect to the higher order norms of the initial conditions.

Beginning from Step 2 in the proof of Theorem \ref{themfinal}, we derive the following estimate of $u$ in the spaces $L^\infty(0, T; H^2(\Om_h))$ and $L^4(0, T; H^{\frac{5}{2}}(\Om_h))$, which is linear with respect to the higher order norms of the initial conditions.
Based on the inequality \rfb{H2estimate}, we have 
\begin{align*}
&\lVert u\rVert_{L^\infty(0, T; H^2(\Om_h))}^2+\lVert\m p_0\rVert_{L^\infty(0, T; H^1(\Om_h))}^2\\
&\lesssim \lVert\prt_t u\rVert_{L^\infty(0, T; L^2(\Om_h))}^2+\lVert u\cdot \nabla u\rVert_{L^\infty(0, T; L^2(\Om_h))}^2+\lVert\prt_t h\rVert_{L^\infty(0, T; H^{\frac{3}{2}}(0, L))}^2.
\end{align*} 
Note that for the convective term above, we estimate by using Ladyzhenskaya inequality that
\begin{equation*}
\begin{aligned}
\lVert u\cdot \nabla u\rVert_{L^\infty(0, T; L^2(\Om_h))}^2
&\leq\lVert u\rVert_{L^\infty(0, T; L^4(\Om_h))}^2\lVert\nabla u\rVert_{L^\infty(0, T; L^4(\Om_h))}^2\\
&\lesssim \lVert u\rVert_{L^\infty(0, T; L^2(\Om_h))}\lVert\nabla u\rVert_{L^\infty(0, T; L^2(\Om_h))}^2\lVert\nabla^2 u\rVert_{L^\infty(0, T; L^2(\Om_h))}\\
&\leq \eps\lVert\nabla^2u\rVert_{L^\infty(0, T; L^2(\Om_h))}^2+C(\eps)\lVert u\rVert_{L^\infty(0, T; L^2(\Om_h))}^2\lVert\nabla u\rVert_{L^\infty(0, T; L^2(\Om_h))}^4.
\end{aligned}
\end{equation*} 
According to Proposition \ref{relareg}, we know that $\lVert \nabla u\rVert_{L^\infty(0, T; L^2(\Om_h))}\sim \lVert\nabla\hat u\rVert_{L^\infty(0, T; L^2(\Om_1))}$. Then we consider the following interpolation inequality:
$$\lVert\nabla\hat u\rVert_{L^\infty(0, T; L^2(\Om_1))}\lesssim \lVert\nabla\hat u\rVert_{L^2(0, T; L^2(\Om_1))}^{\frac{1}{2}}\lVert\nabla\hat u\rVert_{H^1(0, T; L^2(\Om_1))}^{\frac{1}{2}}
\lesssim C_0^{\frac{1}{4}} \lVert\nabla\hat u\rVert_{H^1(0, T; L^2(\Om_1))}^{\frac{1}{2}}, $$
This gives us that 
\begin{align}\label{convect1}
\lVert u\cdot \nabla u\rVert_{L^\infty(0, T; L^2(\Om_h))}^2
&\leq  \eps\lVert\nabla^2u\rVert_{L^\infty(0, T; L^2(\Om_h))}^2+C(\eps)C_0^2 \lVert\nabla\hat u\rVert_{H^1(0, T; L^2(\Om_1))}^2,
\end{align}
where we used the energy balance \rfb{energyhat}. Combining with the above estimate, 
we derive from Proposition \ref{initprop} and Theorem \ref{thm:main-disrete} that
\begin{equation*}
\begin{aligned}
&\lVert u\rVert_{L^\infty(0, T; H^2(\Om_h))}^2+\lVert\m p_0\rVert_{L^\infty(0, T; H^1(\Om_h))}^2
\\
&\leq c(C_0)\left(\lVert\hat u_0\rVert_
{H^2(\Om_1)}^2+ \lVert h_0\rVert_{H^4(0, L)}^2+\lVert h_1\rVert_{H^2(0,L)}^2+1\right).
\end{aligned}
\end{equation*}

Similarly, we can derive a 'linear estimate' from \rfb{higher}. We consider
\begin{align}\label{prefin2}
&\lVert u\rVert_{L^4(0, T; H^{\frac{5}{2}}(\Om_h))}^4+\lVert\m p_0\rVert_{L^4(0, T; H^{\frac{3}{2}}(\Om_h))}^4\nonumber\\
&\lesssim \lVert\prt_t u\rVert_{L^4(0, T; H^{\frac{1}{2}}(\Om_h))}^4+\lVert u\cdot \nabla u\rVert_{L^4(0, T; H^{\frac{1}{2}}(\Om_h))}^4+\lVert\prt_t h\rVert_{L^2(0, T; H^{2}(0, L))}^4.
\end{align}
By using the interpolation \rfb{inequse}, we deal with the convective term again as follows:
\begin{equation}\label{prefin1}
\begin{aligned}
\lVert u\cdot \nabla u\rVert_{L^4(0, T; H^{\frac{1}{2}}(\Om_h))}^4
&\lesssim \lVert u\cdot \nabla u\rVert_{L^\infty(0, T; L^2(\Om_h))}^{2}\lVert u\cdot \nabla u\rVert_{L^2(0, T; H^1(\Om_h))}^{2}.
\end{aligned}
\end{equation} 
Observe that 
$$\lVert u\cdot \nabla u\rVert_{L^2(0, T; H^1(\Om_h))}^{2}=\int_0^T\left(\lVert \nabla u\rVert_{L^2(\Om_h)}^4+\lVert u\cdot \nabla^2 u\rVert_{L^2(\Om_h)}^2\right)\dd t: ={\rm I+II}.$$ 
Now we estimate ${\rm I}$ and ${\rm II}$, respectively. We first have 
\begin{align*}
{\rm I}&\lesssim \int_0^T\lVert u\rVert_{L^2(\Om_h)}^{\frac{12}{5}}\lVert \nabla u\rVert_{H^{\frac{3}{2}}(\Om_h)}^{\frac{8}{5}}\dd t\\
&\lesssim \left(\int_0^T\lVert u\rVert_{L^2(\Om_h)}^4\dd t\right)^{\frac{3}{5}}\left(\int_0^T\lVert \nabla u\rVert_{H^{\frac{3}{2}}(\Om_h)}^{4}\right)^{\frac{2}{5}},
\end{align*} 
where we used H\"older's inequality and the interpolation inequality:
$$\lVert \nabla u\rVert_{L^2(\Om_h)}\lesssim \norm{ \nabla u}_{H^{-1}(\Om_h)}^{\frac{3}{5}}\lVert \nabla u\rVert_{H^{\frac{3}{2}}(\Om_h)}^{\frac{2}{5}}. $$
For ${\rm II}$, we also have
\begin{align*}
{\rm II}&\leq \int_0^T\lVert u\rVert_{L^4(\Om_h)}^2\lVert\nabla^2 u\rVert_{L^4(\Om_h)}^2\dd t\\
&\lesssim \int_0^T\lVert u\rVert_{L^2(\Om_h)}\lVert\nabla u\rVert_{L^2(\Om_h)}\lVert\nabla^2 u\rVert_{H^{\frac{1}{2}}(\Om_h)}^2\dd t\\
&\leq \left(\int_0^T\lVert \nabla^2 u\rVert_{H^{\frac{1}{2}}(\Om_h)}^4\dd t\right)^{\frac{1}{2}}\left(\int_0^T\lVert\nabla u\rVert_{L^2(\Om_h)}^2\dd t\right)^{\frac{1}{2}}.
\end{align*}
This implies using \eqref{convect1} and \rfb{prefin1} that
\begin{align*}
\lVert u\cdot \nabla u\rVert_{L^4(0, T; H^{\frac{1}{2}}(\Om_h))}^4
&\lesssim
\Bigg(\left(\int_0^T\lVert \nabla^2 u\rVert_{H^{\frac{1}{2}}(\Om_h)}^4\dd t\right)^{\frac{1}{2}}+\left(\int_0^T\lVert \nabla u\rVert_{H^{\frac{3}{2}}(\Om_h)}^{4}\right)^{\frac{2}{5}}\Bigg)
\\
&\times\norm{\nabla^2 u}_{L^\infty(0,T;L^2(\Omega_h))}\norm{\nabla \hat u}_{H^1(0,T;L^2(\Omega_1))}
\\
&\leq\eps\int_0^T\lVert u\rVert_{H^{\frac{5}{2}}(\Om_h)}^4\dd t +\left(\lVert u_0\rVert_
{H^2(\Om_1)}+ \lVert h_0\rVert_{H^4(0, L)}+\lVert h_1\rVert_{H^2(0,L)}+1\right)^4.
\end{align*}
Hence, we conclude from \rfb{prefin2} and the above that
\begin{align*}
&\lVert u\rVert_{L^4(0, T; H^{\frac{5}{2}}(\Om_h))}^4+\lVert\m p_0\rVert_{L^4(0, T; H^{\frac{3}{2}}(\Om_h))}^4\\
&\leq c(C_0)\left(\lVert\hat u_0\rVert_
{H^2(\Om_1)}^4+ \lVert h_0\rVert_{H^4(0, L)}^4+\lVert h_1\rVert_{H^2(0,L)}^4+1\right).
\end{align*}
This finishes the proof of Theorem~\ref{thm:main}.

\section*{Funding}
S. Schwarzacher and P. Su are partially supported by the ERC-CZ Grant CONTACT LL2105 funded by the Ministry of Education, Youth and Sport of the Czech Republic and the University Centre UNCE/SCI/023 of Charles University. S. Schwarzacher also acknowledges the support of the VR
Grant 2022-03862 of the Swedish Science Foundation.

\section*{Compliance with Ethical Standards}
\smallskip
\par\noindent
{\bf Conflict of Interest}. The authors declare that they have no conflict of interest.

\smallskip
\par\noindent
{\bf Data Availability}. Data sharing is not applicable to this article as no datasets were generated or analyzed during the current study.

\end{document}